\documentclass[12pt]{amsart}

\usepackage{amsmath,amssymb,amsfonts,graphics,amsthm}
\usepackage{verbatim}
\usepackage{fullpage}
\usepackage{tikz}
\newcommand{\JW}[1]{p_{#1}}

\newcommand{\N}{\mathbb{N}}
\newcommand{\Z}{\mathbb{Z}}
\newcommand{\R}{\mathbb{R}}

\newcommand{\C}{\mathbb{C}}
\newcommand{\F}{\mathbb{F}}
\newcommand{\G}{\mathbb{G}}
\newcommand{\hG}{\widehat{\mathbb{G}}}

\newcommand{\T}{\mathbb{T}}

\newcommand{\mc}{\mathcal}
\newcommand{\id}{\text{id}}
\newcommand{\Conv}{\text{Conv}}

\newcommand{\Tr}{\text{Tr}}
\newcommand{\TL}{\text{TL}}
\newcommand{\Mor}{\text{Mor}}
\newcommand{\Irr}{\text{Irr}}

\newtheorem{thm}{Theorem}[section]
\newtheorem{cor}[thm]{Corollary}
\newtheorem{lem}[thm]{Lemma}
\newtheorem{prop}[thm]{Proposition}

\theoremstyle{definition}
\newtheorem{defn}[thm]{Definition}

\theoremstyle{remark}
\newtheorem{rem}[thm]{Remark}
\newtheorem{notat}[thm]{Notation}

\newtheorem{ex}[thm]{Example}
\numberwithin{equation}{section}

\begin{document}

\title[Quantum automorphism groups]
{Reduced operator algebras of trace-preserving quantum automorphism groups}

\date{\today}

\author{Michael Brannan}

\address{Michael Brannan: Department of Mathematics, University  of Illinois at Urbana-Champaign, Urbana, IL 61801, USA}
\email{mbrannan@illinois.edu}
\urladdr{http://www.math.uiuc.edu/~mbrannan}


\begin{abstract}
Let $B$ be a finite dimensional C$^\ast$-algebra equipped with its canonical trace induced by the regular representation of $B$ on itself.  In this paper, we study various properties of the trace-preserving quantum automorphism group $\G$ of $B$.  We prove that the discrete dual quantum group $\hG$ has the property of rapid decay, the reduced von Neumann algebra $L^\infty(\G)$   has the Haagerup property and is solid, and that $L^\infty(\G)$ is (in most cases) a prime type II$_1$-factor.  As applications of these and other results, we deduce the metric approximation property, exactness, simplicity and uniqueness of trace for the reduced $C^\ast$-algebra $C_r(\G)$, and the existence of a multiplier-bounded approximate identity for the convolution algebra $L^1(\G)$. 
\end{abstract}
\maketitle

\section{Introduction} \label{section:intro}

Consider a finite set $X_n$ consisting of $n$ distinct points.  An elementary and well known fact from group theory is that the automorphism group of $C(X_n)$, the commutative algebra of complex valued functions over $X_n$, is precisely the permutation group $S_n = \text{Aut}(X_n)$.  In other words, the permutation group on $n$ points is the universal object within the category of groups acting on $C(X_n)$.
In \cite{Wa98}, Wang made the remarkable discovery that if one replaces the category of groups acting on $C(X_n)$ by the category of \textit{compact quantum groups} acting on $C(X_n)$, then there still exists a universal object within this new category, called the \textit{quantum permutation group} $S_n^+$.  $S_n^+$ is a C$^\ast$-algebraic compact quantum group which contains the classical permutation group $S_n$ as a quantum subgroup.  Moreover, the inclusion $S_n \subseteq S_n^+$ is proper when $n \ge 4$. That is, finite sets with at least four elements admit genuinely ``quantum permutations''. 

More generally, given a (non-commutative) finite dimensional C$^\ast$-algebra $B$, one can ask about the existence of a universal object within the category of compact quantum groups acting on $B$.  This question was also considered  by Wang \cite{Wa98}, who showed that such an object exists provided we fix a faithful state $\psi:B \to \C$ and restrict to the subcategory of compact quantum groups acting on $B$ which preserve the state $\psi$.  The resulting object is called the \textit{quantum automorphism group} of the pair $(B,\psi)$, and is denoted by $\G_{\text{aut}}(B, \psi)$.  When $B = C(X_n)$ and $\psi$ is the uniform probability measure on $X_n$, we recover the quantum permutation group $S_n^+$. See Example \ref{ex:qperm}. 

Since the quantum groups $\G_{\text{aut}}(B, \psi)$ generally depend a great deal on the initial choice of state $\psi$, a natural first step in attempting to develop a structure theory for these objects is to restrict attention to certain ``canonical'' choices of states on $B$ (which should ideally correspond to uniform probability measures in the commutative case).   In \cite{Ba99,Ba02}, Banica showed that one very natural choice of state $\psi:B \to \C$, which ensures that the generators of the category of finite dimensional representations of $\G = \G_{\text{aut}}(B,\psi)$ satisfy the fewest relations (or are are as ``free'' as possible), is to  take $\psi$ to be the restriction to $B$ of the unique trace on $\mc L(B)$, where $B \hookrightarrow \mc L(B)$ via the left regular representation.  In particular, with the above choice of $\psi$ (or more generally by taking $\psi$ to be a \textit{$\delta$-form on $B$} -- see Definition \ref{defn_delta_form}), Banica showed that when $\dim B \ge 4$, the finite dimensional representation theory of $\G$ admits a nice description using Temperley-Lieb algebras.   Using this link with Temperley-Lieb algebras, he obtained a complete description of the fusion rules for the irreducible unitary representations of $\G$ and from this deduced many interesting structure results, including the fact that $\G$ is co-amenable if and only if $\dim B \le 4$.  

The goal of this paper is to investigate the operator algebraic structure of the quantum groups $\G = \G_{\text{aut}}(B,\psi)$, where $\psi:B \to \C$ is the canonical trace described above.  Our focus will be on properties of the reduced C$^\ast$-algebras $C_r(\G)$ and reduced von Neumann algebras $L^\infty(\G) = C_r(\G)''$, as well as the quantum convolution algebras $L^1(\G) = L^\infty(\G)_*$.  Note that in the co-amenable (i.e., $\dim B \le 4$) case, the structure of $\G$ is already well understood.  If $\dim B = n \le 3$, then $B = C(X_n)$ and $\G = S_n$.  If $\dim B =4$, then either $B =M_2(\C)$ or $C(X_4).$  In the first case, $\G$ turns out to be isomorphic to $SO(3)$, the classical $\ast$-automorphism group of $M_2(\C)$ (see \cite{Ba99}).  In the second case, $\G = S_4^+$, which has been extensively studied in \cite{BaBi} and \cite{BaCo08}.  In particular,  Banica and Collins \cite{BaCo08} have obtained a concrete model for the reduced (= full, due to co-amenability) C$^\ast$-algebra $C_r(S_4^+)$ in the form of an explicit embedding $C_r(S_4^+) \hookrightarrow M_4(\C) \otimes C(SU(2))$. 

In the non-co-amenable (i.e., $\dim B \ge 5$) regime, much less is known about the structure of $\G$ and its operator algebras.  Among the known results, we note that De Rijdt and Vander Vennet \cite{DeVa} have obtained a detailed description of the probabilistic boundaries associated to the discrete dual quantum group $\hG$.  In addition, combinatorial ``Weingarten'' formulas for the Haar state on the quantum permutation groups $S_n^+$ were obtained in \cite{BaCo07}. The large $n$ asymptotics of these formulas have found many interesting applications in free probability theory. See for example \cite{BaCuSp, BaCuSp2, KoSp}.  Perhaps most striking of these applications is K\"ostler and Speicher's free analogue of de Finetti's theorem \cite{KoSp}, which asserts that an infinite sequence of random variables in a W$^\ast$-probability space is identically distributed and free with amalgamation if and only if its joint $\ast$-distribution is invariant under the action of the quantum permutation groups $\{S_n^+\}_{n\ge 1}$.     

Returning now to a fixed quantum automorphism group $\G = \G_{\text{aut}}(B,\psi)$ with $\dim B \ge 5$, basic results in co-amenability theory imply that $C_r(\G)$ is non-nuclear, $L^\infty(\G)$ is non-injective and $L^1(\G)$ is a Banach algebra which fails to have a bounded approximate identity. However, based on the above connections between $S_n^+$ and free independence, as well as some compelling evidence provided by work of Vaes and Vergnioux \cite{VaVe} on free orthogonal quantum groups, more is conjectured to be true.   Namely, it is expected that $L^\infty(\G)$ is a full, prime type II$_1$-factor and $C_r(\G)$ is a simple C$^\ast$-algebra with unique trace when $\dim B \ge 5$  (see \cite{BaCo08},  Introduction).

In this paper we show that if $\dim B \ge 8$, then $L^\infty(\G)$ is a full type II$_1$-factor (Theorem \ref{thm:factoriality_tracial}) and $C_r(\G)$ is simple with unique trace (Corollary \ref{cor:simplicity}) -- verifying the above conjecture in all but six cases.  We also prove that $L^\infty(\G)$ is a solid von Neumann algebra (Corollary \ref{cor:solidity_prime}): the relative commutant of any diffuse injective von Neumann subalgebra of $L^\infty(\G)$ is injective.  Combining this with the above factoriality result, we conclude that $L^\infty(\G)$ is a prime factor as soon as $\dim B \ge 8$.  The solidity of $L^\infty(\G)$ is established by proving that $L^\infty(\G)$ has property (AO)$^+$ (Theorem \ref{thm:AO}), which implies solidity by a celebrated result of Ozawa \cite{Oz}.  See Section \ref{section:AO} for the relevant details and definitions.  Our proof of property (AO)$^+$ follows very closely the method of Vergnioux \cite{Ve05}, where property (AO)$^+$ was established for universal discrete quantum groups.                 

In addition to considering the algebraic structure of $L^\infty(\G)$ and  $C_r(\G)$, we also investigate various approximation properties for these operator algebras and the convolution algebra $L^1(\G)$. We prove that when $\dim B \ge 5$, $L^\infty(\G)$ has the Haagerup property (Theorem \ref{thm:HAP_Kac}), $C_r(\G)$ (respectively $L^\infty(\G)$) has the (respectively weak$^\ast$) metric approximation property, and that $L^1(\G)$ admits a central approximate identity which is contractive in the multiplier-norm (Theorem \ref{thm:further_approx}).  To obtain these approximation properties, we first construct a suitable family of $L^2$-compact  completely positive multipliers on $L^\infty(\G)$, establishing the Haagerup property.  We then prove that the discrete dual quantum group $\hG$ has the property of rapid decay, which amounts to proving a type of Haagerup inequality for $L^\infty(\G)$  (Theorem \ref{thm:RD}).  The norm control provided by the property of rapid decay for $\hG$ allows us to truncate our completely positive multipliers down to finite rank contractions in such a way as to yield the (weak$^\ast$) metric approximation property.  (We also use the property of rapid decay in the proof of simplicity for $C_r(\G)$).  The existence of the claimed approximate identity for $L^1(\G)$ then follows by a standard duality argument.  We also observe that $C_r(\G)$ is an exact C$^\ast$-algebra (Corollary \ref{cor:exactness}), by piecing together some results on monoidal equivalence for compact quantum groups from \cite{DeVa} and \cite{VaVe}.

The remaining parts of this paper are organized as follows: Section \ref{section:prelims} fixes some notation and contains the basic facts on compact and discrete quantum groups that will be needed.  In Section \ref{section:QAG} we recall the definition of the quantum automorphism groups $\G = \G_{\text{aut}}(B,\psi)$, $\delta$-forms  and describe the representation theory of $\G$ when $\psi$ is a $\delta$-form.  Section \ref{section:AP} deals with approximation properties for the reduced operator algebras of $\G$ (assuming $\psi$ is the canonical trace on $B$) and the property  of rapid decay for $\hG$.
In Section \ref{section:AS} we prove our factoriality, fullness and  and simplicity results for the reduced operator algebras of $\G$.  We end Section \ref{section:AS} by establishing property (AO)$^+$
for $L^\infty(\G)$, deducing solidity and primeness results from this. The final section (Section \ref{section:appendix}) is an appendix containing a proof of Theorem \ref{thm:decomp_T}, which is a technical result required in our study of factoriality and fullness for $L^\infty(\G)$.   
      
\subsection*{Acknowledgements}  The author wishes to thank T.  Banica, B. Collins, P. Fima, A. Freslon, J. A. Mingo, R. Speicher and R. Vergnioux for useful conversations and comments at various stages of this project.      

\section{Preliminaries} \label{section:prelims}

\subsection{Notation}
For a Hilbert space $H$, we take its inner product to be conjugate-linear in the first variable and write $\mc B(H)$ for the C$^\ast$-algebra of bounded linear operators on $H$.  Given $\xi, \eta \in H$, we write $\omega_{\eta,\xi} \in \mc B(H)_*$ for the $\sigma$-weakly continuous linear functional $T \mapsto \langle \eta|T\xi \rangle$.   For tensor products, we write $\otimes$ for the minimal tensor product of C$^\ast$-algebras or Hilbertian tensor product of Hilbert spaces, $\overline{\otimes}$ for the spatial tensor product of von Neumann algebras and $\otimes_{\text{alg}}$ for the algebraic tensor product.  If $X$ is an operator space,  we write $\mc {CB}(X)$ for the algebra of completely bounded linear maps on $X$.   We take the natural numbers $\N$ to include $0$ and for $m \ge 1$ we write $[m]:=\{1,2, \ldots, m \}$.  We also use standard leg numbering notation for operators on multiple tensor products.  For example, if $H_1, H_2, H_3$ are Hilbert spaces and $T \in \mc B(H_1 \otimes H_2)$, then $T_{12} = T \otimes \id_{H_3} \in \mc B(H_1 \otimes H_2 \otimes H_3)$, $T_{13} =  (\id \otimes \sigma^*)T_{12} (\id \otimes \sigma) \in \mc B(H_1 \otimes H_3 \otimes H_2)$ where $\sigma:H_3 \otimes H_2 \to H_2 \otimes H_3$ is the tensor flip map and $T_{23}$ is defined similarly. 

\subsection{Compact and Discrete Quantum Groups}  We present here a very brief summary of the basic facts on compact and discrete quantum groups that will be needed for this paper.  Our main reference for this will be \cite{Wo2} and the excellent book \cite{Ti}.

\begin{defn} \label{defn:CQG}
A \textit{compact quantum group} is a pair $\G = (A,\Delta)$, where $A$ is a unital C$^\ast$-algebra and $\Delta:A \to A \otimes A$ is a unital $\ast$-homomorphism (called a \textit{coproduct}) satisfying the coassociativity condition
\[(\id_A \otimes \Delta) \circ \Delta = (\Delta \otimes \id_A) \circ \Delta \]
and the cancellation property \[ \overline{\text{span}}\{\Delta(A)(1_A \otimes A)\} = A \otimes A = \overline{\text{span}}\{\Delta(A)(A \otimes 1_A)\}.\]
\end{defn} 

Every compact quantum group $\G = (A,\Delta)$ admits a \textit{Haar state}, which is a state $h:A \to \C$ defined uniquely by the following bi-invariance condition with respect to the coproduct: \[(h \otimes \id_A)\Delta (a) = (\id_A \otimes h)\Delta(a) = h(a)1_A \qquad (a \in A).\]
Let $L^2(\G)$ be the Hilbert space obtained from the GNS construction applied to the sesquilinear form on $A$ defined by $\langle a|b \rangle_h = h(a^*b)$, and let $\lambda:A \to \mc B(L^2(\G))$ be the associated \textit{left regular representation}.  Recall that $\lambda$ is given by $\lambda(a)\Lambda_h(b) = \Lambda_h(ab)$, where $a,b \in A$ and $\Lambda_h:A \to L^2(\G)$ is the GNS map.  The \textit{reduced C$^\ast$-algebra} and \textit{reduced von Neumann algebra} of $\G$ are the concrete operator algebras acting on $L^2(\G)$ given by \[C_r(\G) := \lambda(A) \subseteq \mc B(L^2(\G)) \qquad \text{and} \qquad L^\infty(\G) := C_r(\G)'',\] respectively.  

Observe that the $\Delta$-invariance of the Haar state $h$ implies that there exists a normal faithful $\ast$-homomorphism $\Delta_r:L^\infty(\G) \to L^\infty(\G) \overline{\otimes} L^\infty(\G)$ given by $\Delta_r \circ \lambda = (\lambda \otimes \lambda) \circ \Delta$, turning $(C_r(\G),\Delta_r)$ into a compact quantum group (the \textit{reduced} version of $\G$).  The pre-adjoint of $\Delta_r$ induces a completely contractive Banach algebra structure on the predual $L^1(\G) := L^\infty(\G)_*$. $L^1(\G)$ is called the \textit{convolution algebra} of $\G$.  $\G$ is said to be of \textit{Kac-type} if the Haar state $h$ is tracial.  In this case, $L^\infty(\G)$ is a finite von Neumann algebra with faithful Haar trace $h$.

Let $H$ be a finite dimensional Hilbert space.  A  \textit{representation} of $\G$ on $H$ is an element $U \in \mc B(H) \otimes A$ such that \begin{align} \label{eqn:rep}(\id \otimes \Delta)U = U_{12}U_{13}.
\end{align}  By fixing an orthonormal basis $\{e_i\}_{i=1}^d$ for $H$ and writing $U = [u_{ij}] \in M_d(A)$ relative to this basis,  \eqref{eqn:rep} is equivalent to requiring that $\Delta(u_{ij}) = \sum_{k=1}^d u_{ik} \otimes u_{kj}$ for $1 \le i,j \le d.$
$U$ is called a \textit{unitary representation} if $U$ is, in addition, a unitary element of $M_d(A)$.  If $U^1 = [u^1_{i(1)j(1)}] \in M_{d(1)}(A)$ and $U^2 = [u^2_{i(2)j(2)}] \in M_{d(2)}(A)$ are (unitary) representations, then their \textit{tensor product} $U^1\boxtimes U^2 = [u^1_{i(1)j(1)}u^2_{i(2)j(2)}] \in M_{d(1)d(2)}(A)$ and \textit{direct sum} $U^1 \oplus U^2 \in M_{d(1)+d(2)}(A)$ are also (unitary) representations.  The vector space \[\Mor(U^1,U^2) = \{T \in \mc B(H_1,H_2) \ | \ (T \otimes 1_A)U^1 = U^2(T \otimes 1_A)\},\]
is called the space of \textit{morphisms} (or \textit{intertwiners}) between $U^1$ and $U^2$.  $U^1$ and $U^2$ are called \textit{(unitarily) equivalent} if there exists an invertible (unitary) operator $T \in \Mor(U^1,U^2)$, and in this case we write $U^1 \cong U^2$.  If $U = [u_{ij}] \in M_d(A)$ is a representation, then so is $\overline{U} = [u^*_{ij}]$.  The representation $\overline{U}$ is called the \textit{conjugate} of $U$.  Note that the passage from $U$ to its conjugate $\overline{U}$ does not preserve unitarity in general.  A representation  $U$ is called \textit{irreducible} if $\Mor(U,U) = \C \id$.  Note that $U$ is irreducible if and only if $\overline{U}$ is irreducible.    

The following is a fundamental result in the theory of compact quantum groups. 

\begin{thm} \label{thm_reducibility}
Every irreducible representation of a compact quantum group $\G  = (A,\Delta)$ is finite dimensional and equivalent to a unitary one if it is invertible.  Moreover, every unitary representation is unitarily equivalent to a direct sum of irreducibles.
\end{thm}

Let $\{U^\alpha = [u^\alpha_{ij}]_{1 \le i,j \le d_\alpha} : \alpha \in \Irr(\G) \}$ be a maximal family of pairwise inequivalent finite-dimensional irreducible unitary representations of $\G = (A, \Delta)$, with $0 \in \Irr(\G)$ denoting the index corresponding to the trivial representation $1_A \in A$ and $U^{\overline{\alpha}}$ denoting the representative of the equivalence class of $\overline{U^\alpha}$.  If $\mc A \subset A$ denotes the linear span of $\{u^\alpha_{ij}: \ 1 \le i,j \le d_\alpha, \alpha \in \Irr(\G) \}$, then $\mc A$ is a Hopf $\ast$-algebra on which $h$ is faithful, $\mathcal A$ is norm dense in $A$, and the set $\{u^\alpha_{ij}: \ 1 \le i,j \le d_\alpha, \alpha \in \Irr(\G) \}$ is a linear basis for $\mc A$.  The coproduct $\Delta:\mc A \to \mc A \otimes_{\textrm{alg}} \mc A $ is just the restriction $\Delta|_{\mc A}$,  the \textit{coinverse} $\kappa: \mc A \to \mc A$ is the antihomomorphism given by 
$\kappa(u_{ij}^\alpha)= (u_{ji}^\alpha)^*$,
and the \textit{counit} $\epsilon:\mc A \to \C$ is the $\ast$-character given by $\epsilon(u_{ij}^\alpha) = \delta_{ij}$.
The Hopf $\ast$-algebra $\mc A$ is unique in the sense that if $\mathcal B \subseteq A$ is any other dense Hopf $\ast$-subalgebra, then $\mathcal B = \mathcal A$ (see \cite[Theorem 5.1]{BeMuTu0}).    

For each $\alpha \in \Irr(\G)$, fix an isometric morphism $t_{\alpha} \in \Mor(1, U^\alpha \boxtimes U^{\overline{\alpha}})$ (which exists and is unique up to scalar multiplication by $\T$).  Let $j_\alpha:H_\alpha \to H_{\overline{\alpha}}$ be the invertible conjugate-linear
map defined by $j_{\alpha}\xi = (\xi^*\otimes \id)t_\alpha$, $(\xi \in H_\alpha)$.  Renormalize $j_\alpha$ so that the positive operator $Q_\alpha := j_\alpha^*j_\alpha > 0$ satisfies $\Tr(Q_\alpha) = \Tr(Q_{\alpha}^{-1})$ and $j_{\overline{\alpha}}j_\alpha = \pm \id_{H_\alpha}$.  The operators $\{Q_\alpha\}_{\alpha \in \Irr(\G)}$ can then be used to describe the  \textit{Schur orthogonality relations} for the matrix elements of irreducible unitary representations of $\G$, which are given by  \begin{align} \label{eqn:Haar} h((u_{ij}^\alpha)^*u_{kl}^\beta) = \frac{\delta_{\alpha\beta}\delta_{jl}(Q_{\alpha}^{-1})_{ki}}{\Tr Q_\alpha} \quad \text{and} \quad h(u_{ij}^\alpha(u_{kl}^\beta)^*) = \frac{\delta_{\alpha\beta}\delta_{ik}(Q_{\alpha})_{lj}}{\Tr Q_\alpha},  \end{align} where $\alpha, \beta \in \Irr(\G)$ and $1 \le i,j \le d_\alpha, 1 \le k,l \le d_\beta$.  In particular, setting  $L^2_\alpha(\G) = \text{span}\{\Lambda_h(u_{ij}^\alpha): 1 \le i,j \le d_\alpha\}$, we obtain the $L^2(\G)$-decomposition \begin{align} \label{eqn:L2decomp} L^2(\G) = \bigoplus_{\alpha \in \Irr(\G)} L^2_\alpha(\G) \cong \bigoplus_{\alpha \in \Irr(\G)} H_\alpha \otimes H_{\overline{\alpha}}, 
\end{align}
where the second isomorphism is given on each direct summand $L^2_\alpha(\G)$ by  
\begin{align} \label{eqn:L2unitarymap}
\Lambda_h\big((\omega_{\eta, \xi} \otimes \id_{A})U^\alpha\big) \mapsto \xi \otimes (1 \otimes \eta^*)t_{\overline{\alpha}} \qquad (\alpha \in \Irr(\G), \ \xi, \eta \in H_\alpha).
\end{align} Note that when $\G$ is of Kac-type, then $Q_\alpha = \id_{H_\alpha}$ for each $\alpha \in \Irr(\G)$ and $\{d_\alpha^{1/2}\Lambda_h(u_{ij}^\alpha) : 1 \le i,j \le d_\alpha\}$ is then an orthonormal basis for $L^2_\alpha(\G)$.  For future reference, we will always write $P_\alpha$ for the orthogonal projection from $L^2(\G)$ onto $L^2_\alpha(\G) \cong H_\alpha \otimes H_{\overline{\alpha}}$. 

Denote by $C_u(\G)$ the universal enveloping C$^\ast$-algebra of the Hopf $\ast$-algebra $\mc A$ and by $\pi_u:\mc A \to C_{u}(\G)$ the universal representation.  Then there is a $\ast$-homomorphism $\Delta_u:C_u(\G) \to C_u(\G) \otimes C_u(\G)$ defined by $\Delta_u \circ \pi_u = (\pi_u \otimes \pi_u) \circ \Delta$, making $(C_u(\G), \Delta_u)$ a compact quantum group (the \textit{universal} version of $\G$).  In the universal setting, we write $C_{\text{alg}}(\G)$ for the canonical dense Hopf $\ast$-subalgebra $\pi_u(\mc A) \subseteq C_u(\G)$.  We say that $\G$ is \textit{co-amenable} if the regular representation $\lambda:C_u(\G) \to C_r(\G)$ is an isomorphism.  The compact quantum groups encountered in this paper will all appear canonically in universal form.  I.e., given by $\G = (A,\Delta)$ where $A = C_u(\G)$ and $\Delta = \Delta_u$.

A \textit{discrete quantum group} is the dual a compact quantum group.  Let $\G = (A,\Delta)$ be a compact quantum group and let $\{U^\alpha\}_{\alpha \in \Irr(\G)}$ (with $U^\alpha \in \mc B(H_\alpha) \otimes A)$ be a maximal family of pairwise inequivalent irreducible unitary representations of $\G$.  Define \[\C[\hG] = \bigoplus_{\alpha \in \Irr(\G)}\mc B(H_\alpha), \qquad  c_0(\hG) = c_0 -  \bigoplus_{\alpha \in \Irr(\G)}\mc B(H_\alpha),   \qquad \ell^\infty(\hG) =\prod_{\alpha \in \Irr(\G)} \mc B(H_\alpha), \] and let $\hat P_\alpha \in \ell^\infty(\hG)$ denote the minimal central projection corresponding to the factor $\mc B(H_\alpha) \in \ell^\infty(\hG)$.     

The unitary operator 
\begin{align} \label{eqn:multunitary}
\mathbb V = \bigoplus_{\alpha \in \Irr(\G)} U^\alpha \in M(c_0(\hG) \otimes A), 
\end{align}
where $M(B)$ denotes the multiplier algebra of a C$^\ast$-algebra $B$, implements the duality between $\G$ and $\hG$.  There is a \textit{dual coproduct} $\hat \Delta: \ell^\infty(\hG) \to \ell^\infty(\hG) \overline{\otimes} \ell^\infty(\hG)$ determined by \[(\hat \Delta \otimes \id_A) \mathbb V = \mathbb V_{13}\mathbb V_{23},\] or equivalently by the relation $\hat \Delta(x)S = S x$ for each $x \in \ell^\infty(\hG)$, $S \in \Mor(U^\alpha, U^\beta \boxtimes U^\gamma)$.  The left and right invariant \textit{Haar weights} on $\ell^\infty(\hG)$,  $\hat h_L$ and $\hat h_R$, are given by \begin{align} \label{eqn:weights}
\hat h_L(x) = \sum_{\alpha \in \Irr(\G)} m_\alpha \Tr(Q_\alpha \hat P_\alpha x), \qquad \hat h_R(x) = \sum_{\alpha \in \Irr(\G)} m_\alpha \Tr(Q_\alpha^{-1} \hat P_\alpha x) \qquad (x \in \C[\hG]), 
\end{align}  
where $m_\alpha = \Tr(Q_\alpha) = \Tr(Q_\alpha^{-1})$ denotes the \textit{quantum dimension} of $H_\alpha$.   Finally, we note that one can show that $\hat \Delta (Q) = Q \otimes Q$, where $Q = \sum_{\alpha \in \text{Irr}(\G)} \hat P_\alpha Q_\alpha$ is the positive unbounded multiplier of $c_0(\hG)$ associated to the family of matrices $\{Q_\alpha\}_{\alpha \in \text{Irr}(\G)}$.

\section{Quantum Automorphism Groups of Finite Dimensional C$^\ast$-algebras.} \label{section:QAG}

We now turn to the central objects of this paper -- quantum automorphism groups of finite dimensional C$^\ast$-algebras.  To define these objects, we first recall the notion of an action of a compact quantum group on a unital C$^\ast$-algebra.   

\begin{defn} \label{defn:action}
Let $B$ be a unital C$^\ast$-algebra and $\G= (A,\Delta)$ a compact quantum group.  
\begin{enumerate}
\item A \emph{right action} of $\G$ on $B$ is a unital $\ast$-homomorphism $\alpha:B \to B \otimes A$ satisfying \[(\alpha \otimes \id)\alpha = (\id \otimes \Delta) \alpha \qquad \text{and} \qquad \overline{\text{span}}\{\alpha(B)(1 \otimes A)\} = B \otimes A. \]
\item If $\psi:B \to \C$ is a positive linear functional on $B$, a \emph{right action} $\alpha$ of $\G$ on the pair $(B,\psi)$ is a right action of $\G$ on $B$  which is also $\psi$-invariant.  That is, \[(\psi \otimes \id_{A})  \alpha(b) = \psi(b)1_{A}, \qquad (b \in B).\]
\end{enumerate}
\end{defn}

\begin{defn} (\cite{Wa98,Ba99,Ba02}) \label{defn:QAG}
Let $(B, \psi)$ be a finite dimensional C$^\ast$-algebra equipped with a faithful state $\psi$.  The \emph{quantum automorphism group of $(B, \psi)$}, denoted by $\G = \G_{\text{aut}}(B, \psi)$, is the universal quantum group defined by a right action of $\G$ on $(B,\psi)$, say $\alpha:B \to B \otimes C_u(\G)$,  with the following properties:
\begin{enumerate}
\item $C_u(\G)$ is defined as the universal C$^\ast$-algebra generated by \[\{(\omega \otimes \id)\alpha(a): \omega \in B^*, a \in B\}.\]
\item The action $\alpha$ is \emph{universal} in the sense that if $\beta:B \to B \otimes A'$ is an action of another compact quantum group $\G' = (A',\Delta')$ on $(B,\psi)$, then there exists a unital $\ast$-homomorphism $\pi:C_u(\G) \to A'$ such that $\beta = (\id_B \otimes \pi)\alpha$.
\end{enumerate}
\end{defn}

\begin{rem} \label{rem:QAG}
As was shown in \cite{Wa98, Ba99, Ba02}, it is possible to give a more ``concrete'' description of the quantum group $\G = \G_{\text{aut}}(B,\psi)$ in terms of generators and relations.  Let $n = \dim B$ and fix an orthonormal basis $\{b_i\}_{i=1}^n$ for $B$ relative to the inner product $\langle \cdot | \cdot \rangle_\psi$  induced by $\psi$.  Let $m:B \otimes B \to B$ be the multiplication map and $\nu:\C \to B$ the unit map.  Then $C_u(\G)$ is the universal C$^\ast$-algebra with $n^2$ generators $\{u_{ij}: 1 \le i,j \le n\}$ subject to the relations which make \begin{align} \label{eqn:relationsQAG} U = [u_{ij}] \in M_n(C_u(\G)) \text{ unitary}, \quad  m \in \Mor(U^{\boxtimes 2}, U) \quad \text{and} \quad \nu \in \Mor(1, U).\end{align}
The universal property of $C_u(\G)$ then allows us to define a coproduct $\Delta:C_u(\G) \to C_u(\G) \otimes C_u(\G)$ by the formula \[ \Delta(u_{ij}) = \sum_{k=1}^n u_{ik}\otimes u_{kj} \qquad (1 \le i,j \le n),\] which turns $(C_u(\G), \Delta)$ into a compact quantum group with fundamental representation $U$ acting on the Hilbert space $H = (B, \langle \cdot|\cdot \rangle_{\psi})$.  The universal action $\alpha$ of $\G$ on $(B,\psi)$ is uniquely determined by $\alpha(b_i) = \sum_{j=1}^n b_j \otimes u_{ji}$, $(1 \le i \le n)$. 
\end{rem}

\begin{ex} \label{ex:qperm} Consider the simplest situation where $X_n = \{1, \ldots, n\}$, $B = C(X_n)$ and $\psi$ is the uniform probability measure on $X_n$.  Taking as an orthogonal basis the $n$ standard projections in $C(X_n)$, it can be shown using Remark \ref{rem:QAG} that $C_u(\G)$ is the universal C$^\ast$-algebra generated by $n^2$ projections $\{u_{ij}: 1 \le i,j \le n\}$ with the property that the rows and columns of the matrix $U = [u_{ij}]$ form partitions of unity.  (An $n \times n$ matrix $U$ over a C$^\ast$-algebra $A$ satisfying the above properties is called a \textit{magic unitary}.)  On the other hand, $C_u(S_n^+)$, the universal C$^\ast$-algebra associated to the quantum permutation group $S_n^+$, admits the same description by \cite{Wa98}, Theorem 3.1.  Therefore $\G_{\text{aut}}(C(X_n),\psi) = S_n^+$.
\end{ex}

As mentioned in the introduction, this paper deals almost exclusively  with finite dimensional C$^\ast$-algebras equipped with certain canonical tracial states, called  \textit{tracial $\delta$-forms} \cite{Ba99,Ba02}.    

\begin{defn}\label{defn_delta_form}
Let $B$ be a finite dimensional C$^\ast$-algebra equipped with a faithful state $\psi:B \to \C$ and  let $\delta >0$.  We call $\psi$ a \emph{$\delta$-form} if for the inner products on $B$ and $B \otimes B$ implemented by $\psi$ and $\psi\otimes \psi$, respectively, the multiplication map $m:B \otimes B \to B$ satisfies $mm^* = \delta^2 \id_B$.  
\end{defn}

The reason for considering $\delta$-forms is two-fold. On the one hand, the class of $\delta$-forms yields many of the interesting concrete examples of quantum automorphism groups (including quantum permutation groups -- see Example \ref{exs:exs} below).   On the other hand, it turns out that the concrete monoidal W$^\ast$-category of finite dimensional unitary representations of $\G_{\text{aut}}(B,\psi)$ admits a very tractable description in this setting (see Section \ref{sect:reptheory}).

\begin{rem} \label{rem:m_nu_relns} Let $\psi$ be a $\delta$-form on  a finite dimensional C$^\ast$-algebra $B$.  One can readily verify (see for example \cite{Ba02}, Theorem 1) that in addition to usual relations \[ mm^* = \delta^2\id_B, \quad \nu^*\nu = 1, \quad m(m \otimes \id_B) = m(\id_B \otimes m), \quad  m(\id_B \otimes \nu)= m(\nu \otimes \id_B) = \id_B \] satisfied by the multiplication map $m:B \otimes B \to B$ and the unit map $\nu:\C \to B$, the additional useful relation
\begin{align*}
m^*m = (m \otimes \id_B)(\id_B \otimes m^*),
\end{align*} 
also holds.
\end{rem}
We now consider some examples.
\begin{ex} \label{exs:exs}
\begin{enumerate}
\item If $B = C(X_n)$ where $X_n = \{1,2, \ldots, n\}$ and $\psi$ is the uniform probability measure on $X_n$, then $\psi$ is a $\delta$-form with $\delta^2 = \dim B = n$.  Note that by Example \ref{ex:qperm}, $\G_{\text{aut}}(B,\psi) = S_n^+$.  
\item If $B = M_n(\C)$ and $\psi = \text{Tr}(Q \cdot)$ with $Q > 0$, then $\psi$ is a $\delta$-form with $\delta^2 = \text{Tr}(Q^{-1})$.
\item \label{generic} If $(B, \psi)$ is arbitrary, then $B$ decomposes as the direct sum $B = \bigoplus_{i=1}^k M_{n_i}(\C)$ and $\psi$ can be written as $\psi = \bigoplus_{i=1}^k \text{Tr}(Q_i \cdot)$ with $Q_i > 0$.  Then $\psi$ is a $\delta$-form if and only if $\text{Tr}(Q_i^{-1}) = \delta^2$ for $1 \le i \le k$.  See \cite{Ba02}.
\end{enumerate}
\end{ex}

Using the notation from Example \ref{exs:exs}\eqref{generic} and the  inequality $\dim B \le \sum_{1 \le i \le k} \Tr(Q_i)\Tr(Q_i^{-1})$, it follows that $\delta^2 \ge \dim B$ for any $\delta$-form $\psi$.  If $\psi$ is moreover a trace, then $\delta^2 = \dim B$ and $\psi$ is precisely the restriction to $B$ of the unique trace on $\mc L(B)$, described in Section \ref{section:intro}.   In this case, we call $\psi$ the \textit{$\delta$-trace} or \textit{tracial $\delta$-form} on $B$.  Note that when $\psi$ is a trace, $\G_{\text{aut}}(B,\psi)$ is a compact quantum group of Kac-type \cite{Wa98}.        

We remind the reader that for $n= \dim B \le 3$, $\G = \G_{\text{aut}}(B,\psi)$ is just the usual permutation group $S_n$ \cite{Wa98}, so for the remainder of the paper we will assume that $\dim B \ge 4$. 

\subsection{Representation Theory and the $2$-cabled Temperley-Lieb Category} \label{sect:reptheory}  

Let $\delta \ge 2$ and $\psi$ be a $\delta$-form on a finite dimensional C$^\ast$-algebra $B$.  In this section we review the description of the representation theory of $\G = \G_{\text{aut}}(B, \psi)$ obtained by Banica in \cite{Ba99, Ba02}.  In the following paragraphs it will be convenient to use the language of (concrete)  monoidal W$^\ast$-categories.  We refer to \cite{Wo88} and \cite{GhLiRo} for the appropriate definitions. 

In \cite{Ba02} (see \cite{Ba99} for the tracial case) it is shown that the fundamental representation $U$ of $\G$ is equivalent to $\overline{U}$, and therefore the concrete monoidal W$^\ast$-category $R$ of all finite dimensional unitary representations of $\G$ is the completion of the concrete monoidal W$^\ast$-category \[R_0 = \big\{ \{U^{\boxtimes k}\}_{k \in \N}, \{\Mor(U^{\boxtimes k}, U^{\boxtimes l})\}_{k,l \in \N} \big\},\]
whose objects are the tensor powers $U^{\boxtimes k}$ of $U$ and whose morphisms are the associated spaces of intertwiners.  Moreover, $R_0$ is generated as  a monoidal W$^\ast$-category by the morphisms $m \in \Mor(U^{\boxtimes 2}, U)$, $\nu \in \Mor(1,U)$ and $\id_B \in \Mor(U,U)$.    

Consider the universal graded C$^\ast$-algebra $(\text{TL}_{k,l}(\delta))_{(k,l) \in \N \times \N}$ given by $\TL_{k,l}(\delta) = \{0\}$ if $k-l$ is odd, and generated by elements $t(k,l) \in \TL_{k+l, k+l+2}(\delta)$ with the following relations (where $1_n$ denotes the unit of the C$^\ast$-algebra $\TL_n(\delta):=\TL_{n,n}(\delta)$): 
\begin{align*}
t(k,l)^*t(k,l) &= 1_{k+l} \\
t(k,l+1)^*t(k+1,l) &= \delta^{-1} 1_{k+l+1} \\
t(r,k+l+2)t(r+k,l) &= t(r+k+2,l)t(r,k+l) \\
t(r,k+l+2)^*t(r+k+2,l) &= t(r+k,l)t(r,k+l)^*.
\end{align*} 
It is well known that each (necessarily finite dimensional!) vector space $\text{TL}_{k,l}(\delta)$ has the following planar-diagrammatic interpretation:  $\text{TL}_{k,l}(\delta)$ is the $\C$-vector space spanned by a basis of $(k,l)$-Temperley-Lieb diagrams $\{D_\pi\}_{\pi \in NC_2(k+l)}$. (That is, diagrams $D_\pi$ consisting of two parallel rows of points -- $k$ points on the bottom row and $l$ points on the top row -- which are connected by a non-crossing pairing $\pi \in NC_2(k+l)$).  We refer to  the book  \cite{KaLi} for details on the spaces $TL_{k,l}(\delta)$.  In terms of diagrams, the generators $t(k,l) \in \TL_{k+l,k+l+2}(\delta)$ are given by the (scaled) diagrams \[t(k,l) = \delta^{-1/2} \underbrace{\big| \ \big| \ \cdots \big|}_{\text{$k$ dots}} \ \bigcup \underbrace{\ \big| \ \cdots \big| \ \big|}_{\text{$l$ dots}}, \qquad t(k,l)^* = \delta^{-1/2} \underbrace{\big| \ \big| \ \cdots \big|}_{\text{$k$ dots}} \ \bigcap \underbrace{\ \big| \ \cdots \big| \ \big|}_{\text{$l$ dots}}. \]
The composition $D_\pi D_\sigma \in \TL_{k,l}(\delta)$ of diagrams $D_\pi \in \TL_{s,l}(\delta)$, $D_\sigma \in \TL_{k,s}(\delta)$ is obtained by the following procedure. First stack the diagram $D_\pi$ on top of $D_\sigma$, connecting the bottom row of $s$ points on $D_\pi$ to the top row of $s$ points on $D_\sigma$.  The result is a new planar diagram, which may have a certain number $c$ of internal loops.  By removing these loops, we obtain a new diagram $D_\rho \in \TL_{k,l}(\delta)$.  The product $D_\pi D_\sigma \in \TL_{k,l}(\delta)$ is then defined to be $\delta^c D_\rho$.  The involution $D_\pi \mapsto D_\pi^*$ is just the conjugate linear extension of the operation of turning diagrams upside-down.

The collection $\TL(\delta) = \{\N, \{\TL_{k,l}(\delta)\}_{k,l \in \N}\}$ has a natural structure as a monoidal W$^\ast$-category and is called the {\it Temperley-Lieb category} (with parameter $\delta$).  Related to this is the \textit{$2$-cabled Temperley-Lieb category} $\TL^2(\delta) = \{\N, \{\TL_{k,l}^2(\delta)\}_{k,l \in \N}\}$, where $\TL_{k,l}^2(\delta):=TL_{2k,2l}(\delta)$ for each $k,l \in \N$.  A fundamental result of \cite{Ba02} is that there is an isomorphism of  monoidal W$^\ast$-categories \[\pi:R_0 \to  \text{TL}^2(\delta) \quad \text{given by} \quad \pi(U^{\boxtimes k}) = k, \quad \pi(\Mor(U^{\boxtimes k}, U^{\boxtimes l})) = \TL_{2k,2l}(\delta) \qquad (k,l \in \N).\]
This is defined on the generating morphisms $\nu,m, \id_B$ of $R_0$ by \begin{align*}&\pi(\nu) = t(0,0)= \delta^{-1/2} \ \bigcup \in \text{TL}_{0,2}(\delta), \quad \pi(m) = \delta t(1,1)^* = \delta^{1/2} \ \big| \bigcap \big| \in \text{TL}_{4,2}(\delta), \\
&\pi(\id_B) = 1_2 = \big| \ \big| \in \TL_{2}(\delta). 
\end{align*}
From now on we will omit $\pi$ in our notation and simply view $\id_B$, $\nu$ and $m$ as generators of (a concrete faithful representation of) $\text{TL}^2(\delta)$.    

Using the above isomorphism, a description of the irreducible unitary representations of $\G$ and their fusion rules can be obtained.  See \cite{Ba02} and Theorem 4.1 of \cite{Ba99} for the tracial case.

\begin{thm} \label{thm:fusion_rules_QAG}
There exists a bijection $\Irr(\G) \cong \N$ and a maximal family $\{U^k\}_{k \in \N}$ of pairwise inequivalent irreducible finite dimensional unitary representations of $\G$ such that:
\begin{enumerate}
\item \label{eqn:selfconjugate} $\overline{U^k} \cong U^k$ for each $k \in \N$.
\item $U^0 = 1_{C_u(\G)}$ and the fundamental representation $U$ decomposes as $U \cong U^0 \oplus U^1$.
\item \label{eqn:fusionrules}$U^n \boxtimes U^k \cong \bigoplus_{r=0}^{2\min\{n,k\}}U^{k+n-r}$ for all $n,k \in \N$.  In particular, $\G$ has the same fusion rules as $SO(3)$.
\item \label{dims}The sequence of representation dimensions  $\{d_k = \dim(U^k)\}_{k \in \N}$ is given by the recursion \[d_0 = 1, \qquad d_1 = \dim B-1, \qquad d_1 d_k = d_{k+1} +d_{k} + d_{k-1} \qquad (k \ge 1).  \]
\end{enumerate}
\end{thm}

\subsection{Explicit Models} \label{expmodels}
To facilitate our investigation of certain structural properties $\G$, it will be useful to have explicit models for the irreducible representations $\{U^k\}_{k \in \N}$ of $\G$, as well as concrete expressions for certain morphisms between their tensor products.

Consider the $k$th tensor power $U^{\boxtimes k}$ of the fundamental representation of $\G$.  Using the isomorphism $R_0 = \TL^2(\delta)$, we can consider the Jones-Wenzl projection $p_{2k} \in \Mor(U^{\boxtimes k},U^{\boxtimes k}) = \TL_{2k}(\delta)$.  Put $H_k = p_{2k}B^{\otimes k}$ and define $U^k$ to be the subrepresentation of $U^{\boxtimes k}$ obtained by restricting to the invariant subspace $H_k \subset B^{\otimes k}$.  One can inductively show that $\{U^k\}_{k \in \N}$ forms a complete family of pairwise inequivalent irreducible unitary representations of $\G$ satisfying the hypotheses of Theorem \ref{thm:fusion_rules_QAG}.   For a good introduction to the basic properties of the Jones-Wenzl projections $\{p_y\}_{y \in \N}$, we refer the reader to the book \cite{KaLi}.  Using the notation from Section \ref{sect:reptheory}, note that $p_y \in \TL_y(\delta)$ is given explicitly by $p_0 = 1$, $p_1 = 1_1$, $p_2 = 1_2-\nu\nu^* = 1_2-t(0,0)t(0,0)^*$, and (more generally) \[p_y = 1_y - \bigvee_{0 \le r \le y-2} t(r,y-r-2)t(r,y-r-2)^* = \  \underbrace{\begin{tikzpicture}[baseline]\draw (.2,-.5)--(.2,.5);
    \node at (.6,0) {$\cdots$};
    \draw (1,-.5)--(1,.5);
\end{tikzpicture}}_{y} \
 - \bigvee_{0 \le r \le y-2} \Bigg(  \delta^{-1} \underbrace{\begin{tikzpicture}[baseline]
    \draw (.2,-.5)--(.2,.5);
    \node at (.6,0) {$\cdots$};
    \draw (1,-.5)--(1,.5);
\end{tikzpicture}}_{r} \
\begin{tikzpicture}[baseline]    \draw (1.2,-.5) arc (180:0:1mm);
    \draw (1.2,.5) arc (-180:0:1mm);
\end{tikzpicture} \
\underbrace{\begin{tikzpicture}[baseline]
    \draw (1.6,-.5) -- (1.6,.5);
    \node at (2,0) {$\cdots$};
    \draw (2.3,-.5) -- (2.3,.5);
\end{tikzpicture}}_{y-r-2} \ \Bigg)\qquad (y \ge 2).\]  Note that a  useful consequence of the above formulae for $\{p_y\}_{y \in \N}$ is the following \textit{absorption property}: $(p_{x} \otimes p_{y})p_{x+y} = p_{x+y}$  $\forall x,y \in \N$.   This property will be used throughout the  paper.

Given $n,k, l \in \N$ such that $U^l \subset U^n \boxtimes U^k$ (i.e., $U^l$ is equivalent to a subrepresentation of $ U^n \boxtimes U^k$), we now proceed to construct explicit morphisms $\rho_{l}^{n \boxtimes k} \in \Mor(U^l,U^n \boxtimes U^k) = (p_{2n} \otimes p_{2k}) \TL_{2l,2(n+k)}(\delta)p_{2l}$.  To do this we first require some notation.  

\begin{notat}\begin{enumerate}
\item Let $\delta \ge 2$ and let $0 < q \le 1$ be such that $\delta = q+q^{-1}$.  Recall that the \textit{$q$-numbers} and \textit{$q$-factorials} are defined by \[[a]_q = \frac{q^a-q^{-a}}{q-q^{-1}} = \frac{q^{-a+1}(1-q^{2a})}{1-q^2}, \qquad [a]_q! = [a]_q[a-1]_q \ldots [1]_q \qquad (a \in \N).\] 
Of course when $q =1$,  the above formulas reduce to $[a]_q = a$ and $[a]_q! = a!$.  Observe that $\delta = [2]_q$, $[3]_q = \delta^2-1$, and $[2k+1]_q = m_k$ - the quantum dimension of $U^k$.  
\item Given $a \in \N$, $1_{a}$ will denote the unit of the algebra $\text{TL}_{a}(\delta)$.  Thus $1_{2a} = \id_{B^{\otimes a}}$ via the isomorphism $R_0 \cong \TL^2(\delta)$.  More generally, if $x \le a$, $1_x$ will denote any collection of $x$ parallel through-strings in a Temperley-Lieb diagram belonging to $\TL_a(\delta)$.
\item For $r \ge 1$, let $t_r = [r+1]_q^{-1/2}(p_r \otimes p_r)\cup^{(r)}$, where $\cup^{(r)}$ denotes $r$ nested cups in $ \TL_{0,2r}(\delta)$.  In terms of planar diagrams,  \[t_r = [r+1]_q^{-1/2}\begin{tikzpicture}[baseline = -6]
	\node (first) at (-1,0) [rectangle,draw] {$\;\; \JW{r} \;\;$};
	\node (second) at (1,0) [rectangle,draw] {$\;\; \JW{r} \;\;$};
	\draw (first.-45) .. controls ++(-60:6mm) and ++(-120:6mm) .. (second.-135);
	\draw (first.-90) .. controls ++(-60:9mm) and ++(-120:9mm) .. (second.-90);
	\draw (first.-135) .. controls ++(-60:12mm) and ++(-120:12mm) .. (second.-45);

\end{tikzpicture} \in \TL_{0,2r}(\delta).\]
\end{enumerate} 
\end{notat}
Now if $U^l \subset U^n\boxtimes U^k$, then by Theorem \ref{thm:fusion_rules_QAG} there is a unique $0 \le r\le 2\min\{n,k\}$ such that $l = n+k - r$.  Define $\rho_{l}^{n \boxtimes k} \in \Mor(U^l,U^n \boxtimes U^k)$ by setting \begin{align} \label{eqn:rho}
\rho_{l}^{n \boxtimes k} = (p_{2n} \otimes p_{2k}) (1_{2n-r} \otimes t_{r} \otimes 1_{2k-r})p_{2l}. 
\end{align}
Using the results of \cite{KaLi}, Section 9.10, it follows that $\big(\rho_{l}^{n \boxtimes k}\big)^*\rho_{l}^{n \boxtimes k} = C_{(n,k,l)}p_{2l},$ where \begin{align} \label{eqn:norm_intertwiner}C_{(n,k,l)} = \frac{[2n+2k-r +1]_q![2n-r]_q![r]_q![2k-r]_q!}{[r+1]_q[2l+1]_q[2n]_q![2k]_q![2l]_q!} > 0.
\end{align} 
Therefore $C_{(n,k,l)}^{-1/2}\rho_{l}^{n \boxtimes k}$ is an \textit{isometric morphism} from $H_l$ onto the subspace of $H_n \otimes H_k$ equivalent to $H_l$.  Moreover, this isometry is unique  (up to scalar multiplication by $\T$) since the inclusion $U^l \subset U^n \boxtimes U^k$ is multiplicity-free by Theorem \ref{thm:fusion_rules_QAG}.  

When we study the property of rapid decay and factoriality, it will be useful for us to have a description of the quantity $1_{2n-r} \otimes t_{r} \otimes 1_{2k-r}$ in formula \eqref{eqn:rho} in terms of the basic morphisms $m, \nu, 1_2$ and the Jones-Wenzl projections $\{p_y\}_{y \in \N}$.  This can be done recursively as follows: if $r=2s$ is even, then $t_{2s}$ is given by the initial condition and recursion
\begin{align}
\label{eqn:t2o}t_0 = 1 \in \C, \qquad t_{2} &= [3]_q^{-1/2} (p_2 \otimes p_2) (m^*\nu),  \qquad p_2 = 1_2 - \nu\nu^*, \\
\label{eqn:t2k}t_{2(k+l)} &=\Big(\frac{[2k+1]_q[2l+1]_q}{[2k+2l+1]_q}\Big)^{1/2} (p_{2k+2l} \otimes p_{2l+2k})(1_{2k} \otimes t_{2l} \otimes 1_{2k})t_{2k} \qquad (k,l \in \N).
\end{align}   If $r=2s+1$ is odd, then  \begin{align} \label{oddt}1_{2n-r} \otimes t_{r} \otimes 1_{2k-r} &= \Big( \frac{[2s+1]_q}{[2s+2]_q[2]_q}\Big)^{1/2}(1_{2n-2s-1} \otimes p_{2s+1} \otimes p_{2s+1} \otimes 1_{2k-2s-1})  \\
& \qquad \times (1_{2n-2s-2} \otimes (1_2 \otimes t_{2s} \otimes 1_2)m^* \otimes 1_{2k-2s-2}),  \notag
\end{align} 
which can be readily verified using the fact that $m^* = [2]_q(1_1 \otimes t_1 \otimes 1_1)$.
Note furthermore that $C_{(k,k,0)} = 1$ for each $k \in \N$ by equation \eqref{eqn:norm_intertwiner}.  Thus each $t_{2k} \in \Mor(1, U^k \boxtimes U^k)$ is an isometric morphism, and each $ \rho_{l}^{n \boxtimes k} = (p_{2n} \otimes p_{2k}) (1_{2n-r} \otimes t_{r} \otimes 1_{2k-r})p_{2l}$ is a contraction.  

To conclude this section let us also fix once and for all an orthonormal basis  $\{e_i\}_{i=1}^{d_1}$ for $H_1 = B \ominus \C1$ and write $U^1 = [u_{ij}^1]$ relative to this basis. The equivalence $U^1 \cong \overline{U^1}$ given by Theorem \ref{thm:fusion_rules_QAG} ensures that there is an invertible matrix $F_1 \in \text{GL}(d_1,\C)$ such that $U^1 = (F_1 \otimes 1) \overline{U^1}(F_1^{-1} \otimes 1)$.  Moreover, since $U^1$ is irreducible, we can assume that $\overline{F_1}F_1 = c1$ where $c = \pm 1$ (see \cite{Ba0}, Page 242).  By possibly replacing  $F_1$ by $zF_1$ for some $z \in \T$, we may also assume that \begin{align} \label{eqn:t2} t_2 =\text{Tr}(F_1^*F_1)^{-1/2} \sum_{i=1}^{d_1} e_i \otimes F_1e_i \in \Mor(1,U^1\boxtimes U^1). \end{align}
This follows because both sides of equation \eqref{eqn:t2} are evidently isometries in the one dimensional space $\Mor(1,U^1\boxtimes U^1)$.
Finally, note that when $\psi$ is the $\delta$-trace on $B$, then $\G_{\text{aut}}(B,\psi)$ is of Kac-type.  We therefore can and will assume that $F_1 \in \mc U(d_1)$ is unitary.  Also note that when $\psi$ is the $\delta$-trace, formula \eqref{eqn:t2o} yields $\sigma t_2 = t_2$, where $\sigma$ is the tensor flip-map.  On the other hand, formula \eqref{eqn:t2} yields $\sigma t_2 = ct_2$.  Consequently  $\overline{F_1}F_1 = 1$.  We leave the easy details to the reader.  

\section{Approximation Properties and The Property of Rapid Decay} \label{section:AP}

This section is devoted to proving some approximation properties for the reduced operator algebras and convolution algebras associated to the quantum groups $\G_{\text{aut}}(B, \psi)$, where $\psi$ is the $\delta$-trace on $B$.  

\subsection{The Haagerup Property} \label{section:HAP}

Our first goal is to prove that $L^\infty(\G)$ has the Haagerup property.  We begin by recalling this approximation property.

\begin{defn} \label{def:HAP}
A finite von Neumann algebra  $(M,\tau)$ has the \textit{Haagerup  property} if there exists a net $\{\Phi_t\}_{t \in \Lambda}$ of normal unital completely positive $\tau$-preserving maps on $M$ such that
\begin{enumerate}
\item For each $t \in \Lambda$, the $L^2$-extension $\hat{\Phi}_t: L^2(M) \to L^2(M)$ is a compact operator.  
\item For each $\xi \in L^2(M)$, $\lim_{t \in \Lambda}\|\hat{\Phi}_t\xi -\xi\|_{L^2(M)} = 0.$ 
\end{enumerate}  
\end{defn}

\begin{thm} \label{thm:HAP_Kac}
Let $B$ be a finite dimensional C$^\ast$-algebra with $\delta$-trace $\psi$ and let $\G = \G_{\text{aut}}(B, \psi)$.  Then $L^\infty(\G)$ has the Haagerup property.
\end{thm}  

To prove Theorem \ref{thm:HAP_Kac}, we begin by recalling a standard method for constructing normal completely positive maps on $L^\infty(\G)$ from states on $C_u(\G)$.  See for example \cite{Br}, Lemma 3.4.

\begin{prop}\label{prop:states_NUCP}
Let $\G$ be a compact quantum group and $\varphi \in C_u(\G)^*$ be a state.  Then there exists a unique normal unital completely positive $h$-preserving map $M_\varphi \in \mc{CB}(L^\infty(\G))$ defined by \[M_\varphi \lambda(a) = \lambda \big(( \varphi \otimes \id_{C_u(\G)} )\Delta a\big) \qquad (a \in C_u(\G)).\]  Moreover, $M_\varphi C_r(\G) \subseteq C_r(\G)$. 
\end{prop}

\begin{rem} \label{rem:multipliers}
It is easily verified that the map $M_\varphi$ defined in Proposition \ref{prop:states_NUCP} satisfies the relation $\Delta_r M_\varphi = (M_\varphi \otimes \id_{L^\infty(\G)})\Delta_r$.  A map $T \in \mc{CB}(L^\infty(\G))$ such that $\Delta_r T = (T \otimes \id_{L^\infty(\G)})\Delta_r$ (respectively $\Delta_r T = (\id_{L^\infty(\G)} \otimes T)\Delta_r$) is called a \textit{completely bounded left (respectively right) multiplier} of $L^\infty(\G)$.  This terminology originates from the classical situation where $L^\infty(\G) = VN(\Gamma)$ is the reduced von Neumann algebra of a discrete group $\Gamma$.  In this case, the above definition reduces to the one for completely bounded Fourier multipliers on $\Gamma$.  See \cite{CoHa}.  Note that for (unital) \textit{completely positive} left multipliers, a recent result of Daws \cite{Da11} shows that Proposition \ref{prop:states_NUCP} is actually a characterization of such multipliers. 
\end{rem} 

Now let $\{U^\alpha = [u_{ij}^\alpha] : \alpha \in \Irr(\G)\}$ be a maximal family of irreducible unitary representations of a compact quantum group $\G$ and let $\chi_\alpha = (\text{Tr} \otimes \id_{C_u(\G)})U^\alpha \in C_u(\G)$ be the irreducible character of the representation $U^\alpha$.  We say that $T \in \mc {CB}(L^\infty(\G))$ is a \textit{central} completely bounded multiplier if $T$ is both a left and right completely bounded multiplier of $L^\infty(\G)$.  This is equivalent to saying that for each $\alpha \in \Irr(\G)$, there is a constant $c_\alpha\in \C$ such that $T\lambda(u^\alpha_{ij}) = c_\alpha\lambda(u^\alpha_{ij})$ for all $1\le i,j \le d_\alpha$.  The following result (proved in \cite{Br}) will be our main tool for proving Theorem \ref{thm:HAP_Kac}.  It shows that when $\G$ is of Kac type, central completely positive multipliers can be obtained from states on the (generally much smaller and more tractable) subalgebra $C^*\big(\chi_\alpha: \alpha \in \Irr(\G)\big) \subset C_u(\G)$ generated by the irreducible characters of $\G$. 

\begin{thm}[\cite{Br}, Theorem 3.7] \label{thm:average_Kac}
Let $\G$ be a compact quantum group of Kac type and let $\varphi \in C^*\big(\chi_\alpha: \alpha \in \Irr(\G)\big)^*$ be a state.  Then there exists a unique normal unital completely positive $h$-preserving map $T_\varphi \in \mc{CB}(L^\infty(\G))$ defined by \[T_\varphi \lambda(u_{ij}^\alpha) = \frac{\varphi(\chi_{\alpha}^*)}{d_\alpha}\lambda(u_{ij}^\alpha) \qquad (\alpha \in \Irr(\G), 1 \le i,j \le d_\alpha).\]   Moreover, $T_\varphi  C_r(\G) \subseteq C_r(\G)$. 
\end{thm}

\begin{rem}
Our proof in \cite{Br} of Theorem \ref{thm:average_Kac} uses the traciality of the Haar state in an essential way.  Indeed, the proof relies on both the boundedness of the coinverse $\kappa:\mc A \to \mc A$ and the existence of an $h$-preserving conditional expectation $E:L^\infty(\G)\overline{\otimes}L^\infty(\G) \to \Delta_r(L^\infty(\G))$. Both of these conditions imply that $h$ is tracial.  
\end{rem}

Let us now return to the quantum automorphism groups $\G = \G_{\text{aut}}(B, \psi)$.  Theorem \ref{thm:average_Kac} suggests that we study the (state space of the) C$^\ast$-algebra $C^*\big(\chi_k: k \in \N\big) \subset C_u(\G)$, where $\chi_k$ is the character of the irreducible representation $U^k$ given by Theorem \ref{thm:fusion_rules_QAG}.  Since $\G$ has commutative fusion rules and each irreducible representation $U^k$ is equivalent to $\overline{U^k}$ by Theorem \ref{thm:fusion_rules_QAG}, it follows from general character theory that $C^*\big(\chi_k: k \in \N\big)$ is commutative and that each character $\chi_k$ is self-adjoint.  In the following paragraphs we will obtain a more precise description of this C$^\ast$-algebra.

Let $\{S_k\}_{k=0}^\infty$ denote the (dilated) Chebyshev polynomials of the second kind, which are defined by the recursion \[S_0(x)=1, \qquad S_1(x) = x, \qquad S_1(x)S_k(x) = S_{k+1}(x) +S_{k-1}(x) \qquad (k \ge 1). \]  These are the monic orthogonal polynomials for Wigner's semicircle law, which is the probability measure supported on $[-2,2]$ with density $\frac{\sqrt{4-x^2}}{2\pi}$.  Following Section 4 of \cite{KuMiSp}, we let $\{\Pi_k\}_{k=0}^\infty$ be the family of polynomials defined by $\Pi_0(x) = 1$ and $\Pi_k(x) = S_k(x-2) +S_{k-1}(x-2)$ for each $k \ge 1$.  Then we have the recursion \begin{align} \label{eqn:recursion_Pi} \Pi_1(x)\Pi_k(x) = \Pi_{k+1}(x) +\Pi_{k}(x) + \Pi_{k-1}(x) \qquad (k \ge 1).  
\end{align}
The polynomials $\{\Pi_k\}_{k \in \N}$ are the monic orthogonal polynomials for the free Poisson law (with parameter 1), which is the probability measure supported on $[0,4]$ with density given by $\frac{\sqrt{4x-x^2}}{2\pi x}$.  Observe that one can equivalently define the family $\{\Pi_k\}_{k \in \N}$ by setting $\Pi_k(x) = S_{2k}(\sqrt{x})$.  To see this, note that $S_{2k}(x)$ is a polynomial of degree $2k$ which contains only terms of even degree.  Therefore we can define a degree $k$ polynomial $\tilde{\Pi}_k$ by setting $\tilde{\Pi}_k(x) = S_{2k}(\sqrt{x})$.  It is then easy to verify that $\tilde{\Pi}_0 = \Pi_0 = 1$, $\tilde{\Pi}_1(x) = x-1 = \Pi_1(x)$, and that the $\tilde\Pi_k$'s satisfy the recursion \eqref{eqn:recursion_Pi}.         

\begin{prop} \label{prop:char_alg}
Let $\chi \in C_u(\G)$ be the character of the fundamental representation $U$ of $\G = \G_{\text{aut}}(B, \psi)$.  Then $\chi = \chi^*$ and there is a $\ast$-isomorphism \[ \Psi: C^*\big(\chi_k:k \in \N\big) \to C(\sigma) \qquad \text{given by} \qquad \Psi:\chi_k \mapsto \Pi_k|_{\sigma}, \]  where $\sigma:=\sigma(\chi) \subset \R$ is the spectrum of $\chi$. 
\end{prop}

\begin{proof}
By Theorem \ref{thm:fusion_rules_QAG}, $U \cong 1 \oplus U^1$ and  each $\chi_k$ is self-adjoint.  Therefore $\chi = 1 + \chi_1 = (1 + \chi_1)^* = \chi^*$.  Moreover, the fusion rules for $\{U^k\}_{k \in \N}$ yield the character relations $\chi_1 \chi_k = \chi_{k+1} + \chi_k + \chi_{k-1}$ for all $k \ge 1$.  Thus $C^*\big(\chi_k: k \in \N\big) = C^*\big(1,\chi\big)$.  Let $\Psi:C^*\big(\chi_k: k \in \N\big) = C^*\big(1,\chi) \to C(\sigma)$ be the Gelfand isomorphism.  Then for any $t \in \sigma$ and $k\ge 1$, we have $\Psi\chi_1(t) \Psi\chi_k(t) = \Psi(\chi_1 \chi_k)(t) = \Psi\chi_{k+1}(t) + \Psi\chi_k(t) + \Psi\chi_{k-1}(t)$, with the initial conditions $\Psi\chi_0(t) = \Psi(1_{C_u(\G)})(t) = 1$ and $\Psi\chi_1(t) = \Psi(\chi-\chi_0)(t) = t-1$.  Comparing with the initial conditions and recursion (\ref{eqn:recursion_Pi}) for the polynomials $\{\Pi_k\}_{k \in \N}$, we obtain $\Psi\chi_k = \Pi_k|_{\sigma}$ ($k \in \N$). 
\end{proof}

This next proposition gives us some useful information about the spectrum $\sigma \subset \R$ of $\chi \in C_u(\G)$, at least in the $\delta$-trace case.

\begin{prop} \label{prop:spectrum}
Let $\psi$ be the $\delta$-trace on $B$ and let $\sigma = \sigma(\chi) \subset \R$ be as in Proposition \ref{prop:char_alg}.  Then $[0,\dim B] \subseteq \sigma$.
\end{prop}

\begin{rem}
We suspect that the equality $[0,\dim B] = \sigma$ always holds in Proposition \ref{prop:spectrum}, but we are unable to show that $\chi \ge 0$ in general.  
\end{rem}

\begin{proof}
Let $\mu$ be the spectral measure of $\chi$ relative to the Haar state $h$.  Since the irreducible characters of a compact quantum group  form an orthonormal family with respect to the $L^2(\G)$-inner product, Proposition \ref{prop:char_alg} yields \[\int_{\text{supp}(\mu)} \Pi_k(t)\Pi_l(t)d\mu(t) = h(\Pi_k(\chi)^*\Pi_l(\chi)) = h(\chi_k^*\chi_l) = \delta_{kl}  \qquad (k,l \in \N).\]  I.e., $\{\Pi_k\}_{k \in \N}$ are the monic orthogonal polynomials for $\mu$. Therefore $\mu$ must be the free Poisson law,  and in particular $[0,4]=\text{supp}(\mu) \subseteq \sigma$.

We now show that $[4,\dim B] \subset \sigma$.  Write $B$ as a direct sum of matrix algebras, say \[B = C(X_{0}) \bigoplus \bigoplus_{1 \le i \le k, \ n_i \ge 2}M_{n_i}(\C),\] where $C(X_0)$ denotes the maximal commutative ideal of $B$ (which may be zero).  Let $n_0 = |X_0|$ so that $\dim B = n_0 + \sum_{i=1}^k n_i^2$.   We will assume for the remainder of the proof that both of the components $C(X_{0})$ and $\bigoplus_{1 \le i \le k, \ n_i \ge 2}M_{n_i}(\C)$ are non-zero.  (The trivial modification of our argument for the remaining two cases is left to the reader).  

Let $\G_0 = S_{n_0}^+$ and let $\G_i = SO(n_i)$ for $1 \le i \le k$.  For each $i$, let $\beta_i:B_i \to B_i \otimes C_u(\G_i)$ be the canonical action of $\G_i$ on $(B_i,\psi_i)$, where $B_i$ is the $i$th direct summand of $B$ and $\psi_i = \psi|_{B_i}$.  (In other words, $\beta_0$ is the fundamental action of $S_{n_0}^+$ on $C(X_{0})$  and $\beta_i$ is the standard action of $SO(n_i)$ on $M_{n_i}(\C)$ by inner automorphisms.)  Now let $\G' = \bigotimes_{i=0}^{k} \G_i$ be the tensor product quantum group (see \cite{Ti}, Chapter 6) and consider the diagonal action \[\beta:B \to B \otimes C_u(\G'); \qquad  \beta = \bigoplus_{0 \le i \le k} (\id_B \otimes \iota_i)\beta_i,\] of $\G'$ on $(B, \psi)$, where $\iota_i:C_u(\G_i) \hookrightarrow \C_u(\G')$ is the canonical embedding.  Invoking  Definition \ref{defn:QAG}, we see that there exists a unital $\ast$-homomorphism $\pi:C_u(\G) \to C_u(\G')$ such that $\beta = (\id_B \otimes \pi)\alpha$, where $\alpha:B \to B \otimes C_u(\G)$ is the universal action of $\G$ on $(B, \psi)$.  Moreover, since $\beta$ is faithful in the sense of \cite{Wa98}, it follows that $\pi$ is surjective.  In particular, if $U'$ denotes the representation of $\G'$ associated to the action $\beta$ and $\chi'$ is its character, then $\chi' = \pi(\chi)$ and $\sigma(\chi') \subseteq \sigma$.  It therefore suffices to show that $[4,\dim B] \subseteq \sigma(\chi')$.   

Let $\chi_{(0)} \in C_u(\G_0)$ denote the fundamental character of $\G_0$, and let $\chi_{(i)} \in C(SO(n_i))$ denote the character of the standard representation of $SO(n_i)$ on $\C^{n_i}$.  A simple calculation using the definition of $\G'$ and the action $\beta$ shows that \[\chi' = \chi_{(0)} \otimes \underbrace{1 \otimes \ldots \otimes 1}_{k \text{ times}} + \sum_{i=1}^k \underbrace{1 \otimes \ldots \otimes 1}_{i \text{ times}} \otimes \chi_{(i)}^2 \otimes \underbrace{1 \otimes \ldots \otimes 1}_{k-i \text{ times}},\] and therefore (by tensor independence) $\sigma(\chi') = \sigma(\chi_{(0)}) + \sum_{i=1}^k \sigma(\chi_{(i)}^2)$.  Moreover, an elementary exercise in linear algebra shows that for $i \ge 1$, $\chi_{(i)} = \chi_{(i)}^*$, $\|\chi_{(i)}\| = n_i$ and $[0,n_i] \subset \sigma(\chi_{(i)}).$  Therefore \begin{align} \label{eqn:spec}\sigma(\chi') = \sigma(\chi_{(0)}) + \sum_{i=1}^k [0,n_i^2] = \sigma(\chi_{(0)}) + [0,\dim B - n_0].
\end{align}

To finish the proof, we analyze the spectrum $\sigma(\chi_{(0)})$ of $\chi_{(0)} \in C_u(\G_0)$.  If $1 \le n_0 \le 4$, then $4 \ge n_0 = \epsilon(\chi_{(0)}) \in \sigma(\chi_0)$, and \eqref{eqn:spec} yields  $[4,\dim B] \subset [n_0, \dim B] \subseteq \sigma(\chi')$.  If $n_0 \ge 5$, write $n_0 = 2s+r$ where $s \ge 2$ and $r \in \{0,1\}$.  Since $C_u(\G_0) = C_u(S_{n_0}^+)$ is the universal C$^\ast$-algebra generated by an $n_0 \times n_0$ magic unitary matrix  $[v_{ij}]$, there exists a unique surjective $\ast$-homomorphism  $\pi_0:C_u(\G_0) \to \ast_{l=1}^s C^*(\Z_2)$ which sends $[v_{ij}]$ to the block-diagonal matrix $W \in M_{n_0}\big(\ast_{l=1}^s C^*(\Z_2)\big)$ given by
\[ W = \left\{\begin{array}{cc}\bigoplus_{j=1}^s \left(\begin{array}{cc} p_j & 1-p_j \\
1-p_j & p_j  \end{array} \right) & \text{if } r= 0,\\
1 \oplus \bigoplus_{j=1}^s\left(\begin{array}{cc} p_j & 1-p_j \\
1-p_j & p_j  \end{array} \right) & \text{if } r=1. \end{array} \right. \]
Here $p_j$ is the projection $\frac{1}{2}(1-g_j)$, where $g_j$ is the canonical unitary generator of the $j$th copy of $C^*(\Z_2)$ in the free product $\ast_{l=1}^s C^*(\Z_2)$.  We then have $\pi_0(\chi_{(0)}) = r1 + \sum_{j=1}^s 2p_j$.  Since the spectrum of a sum of $s$ free projections is well known to equal $[0,s]$ (see \cite{dlHRoVa}, Example 4), we obtain $\sigma(\pi_0(\chi_{(0)})) = [r, 2s + r] \subset \sigma(\chi_{(0)})$.  Using this last inclusion in (\ref{eqn:spec}), we conclude that $\sigma(\chi') \supset [r,\dim B] \supset [4,\dim B]$.  
\end{proof}

We are now ready to prove Theorem \ref{thm:HAP_Kac}.

\begin{proof}[Proof of Theorem \ref{thm:HAP_Kac}]
If $1 \le \dim B \le 4$ then $\G$ is co-amenable by \cite{Ba99}, Corollary 4.2, and it follows from \cite{BeMuTu}, Theorem 1.1, that $L^\infty(\G)$ is injective.  In particular, $L^\infty(\G)$ has the Haagerup property.    

Assume now that $\dim B \ge 5$ and fix $4 < t_0 < 5$.  From the isomorphism $C^*\big(\chi_k: k \in \N\big) \cong C(\sigma)$ given by Proposition \ref{prop:char_alg} together with the fact that $[0,\dim B] \subseteq \sigma$ from Proposition \ref{prop:spectrum}, it follows that for each $t \in [t_0, \dim B)$, there is a unique state $\varphi_t \in C^*\big(\chi_k: k \in \N\big)^*$ given by \[\varphi_t(\chi_k) = \Pi_k(t) \qquad (k \in \N).\]  Consider the irreducible representation $U^k = [u^k_{ij}]$ of $\G$ with label $k$.  From part \eqref{dims} of Theorem \ref{thm:fusion_rules_QAG} and the recursion \eqref{eqn:recursion_Pi}, we have $d_k = \Pi_k(\dim B)$ for all $k \in \N$.  Applying Theorem \ref{thm:average_Kac} to the family of states $\{\varphi_t\}_{t \in [t_0,\dim B)} \subset C^*\big(\chi_k: k \in \N\big)^*$, we obtain a net of central completely positive multipliers $\{\Phi_t\}_{t \in [t_0, \dim B)} \subset \mc{CB}(L^\infty(\G))$ defined  by \[\Phi_t(\lambda(u^k_{ij})) = \frac{\Pi_k(t)}{\Pi_k(\dim B)}\lambda(u_{ij}^k) \qquad (k \in \N, 1 \le i,j \le d_{k}).\]  Here $[t_0, \dim B) \subset \R$ is directed in the usual way.  

We  now show that this net of normal unital completely positive maps is $L^2$-compact and converges pointwise in $L^2$-norm to the identity map.  Using the decomposition $L^2(\G) = \bigoplus_{k \in \N} L^2_k(\G)$ given by \eqref{eqn:L2decomp}, each $L^2$-extension $\hat{\Phi}_t \in \mc B(L^2(\G))$ is given by $\hat{\Phi}_t = \sum_{k \in \N}  \frac{\Pi_k(t)}{\Pi_k(\dim B)} P_k$.  Since $\lim_{t \to \dim B} \frac{\Pi_k(t)}{\Pi_k(\dim B)} = 1$ for each $k$ and $\|\hat \Phi_t\| = 1$ for each $t$, it follows that $\lim_{t \to \dim B}\|\hat{\Phi}_t \xi- \xi\|_{L^2(\G)} = 0$ for each $\xi \in L^2(\G)$.  To see that each $\hat{\Phi}_t$ is compact, observe that since $\hat{\Phi}_t$ is the orthogonal direct sum of the finite rank projections $\{P_k\}_{k \in \N}$, it suffices to show that $\Big\{ \frac{\Pi_k(t)}{\Pi_k(\dim B)} \Big\}_{k \in \N} \in c_0(\N)$ for each $t \in [t_0,\dim B)$.  To check this last fact, note that $0 < \frac{\Pi_k(t)}{\Pi_k(\dim B)} = \frac{S_{2k}(\sqrt{t})}{S_{2k}(\sqrt{\dim B})}$,
and an easy exercise (see for example \cite{Br}, Proposition 4.4) shows that there is a constant $A(t_0) > 0$ (depending only on $t_0$) such that \[\frac{S_{2k}(\sqrt{t})}{S_{2k}(\sqrt{\dim B})} \le A(t_0) \Big( \frac{\sqrt{t}}{\sqrt{\dim B}}\Big)^{2k} = A(t_0) \Big( \frac{t}{\dim B}\Big)^{k} \longrightarrow 0 \qquad (k \to \infty). \]
\end{proof}

\subsection{The Property of Rapid Decay} \label{section:RD}

Given a finitely generated discrete group $\Gamma$ with symmetric generating set $S = \{\gamma_1, \ldots, \gamma_s\}$, there is a natural length function $\ell = \ell_S: \Gamma \to \N$ defined by assigning to each $\gamma \in \Gamma$ the length of its shortest representation as a reduced word in the elements of $S$.  The group $\Gamma$ is said to have \textit{the property of rapid decay (with respect to $\ell$)}, or just property RD, if there exists a polynomial $P \in \R_+[x]$ such that for each $n \in \N$ and each  function $f:\{\gamma \in \Gamma  \ | \ \ell(\gamma) = n\} \to \C$ supported on words in $\Gamma$ of length $n$, the following inequality holds \[\|\lambda(f)\|_{VN(\Gamma)} \le P(n)\|f\|_{\ell^2(\Gamma)}.\]
It is a standard fact that property RD is essentially independent of the initial generating set $S$ for $\Gamma$.  Perhaps the most famous example of a group with property RD is the non-abelian free group $\F_d$ on $d \ge 2$ generators, which was shown by Haagerup \cite{Ha} to have property RD with $P(x) = x+1$ for each $d\ge 2$.   

The notion of property RD for discrete quantum groups was defined and studied by Vergnioux in \cite{Ve}.  There he showed that property RD can only occur for unimodular discrete quantum groups, and he also exhibited the first (and what currently appear to be the only) truly ``quantum'' examples having this property.  The goal in this section is to prove that if $B$ is a finite dimensional C$^\ast$-algebra with $\delta$-trace $\psi$, then the dual unimodular discrete quantum group $\hG$ of $\G = \G_{\text{aut}}(B,\psi)$ has property RD.  We will use this result to obtain some finite rank approximation properties for $L^\infty(\G), C_r(\G)$ and $L^1(\G)$ in Theorem \ref{thm:further_approx}.  Property RD will also be a key tool in our analysis of simplicity and uniqueness of trace for $C_r(\G)$ in Section \ref{section:simplicity}.  

For our purposes, property RD for $\hG$ will be most conveniently formulated in dual terms, as an equivalent property of $L^\infty(\G)$.  More precisely, we will prove the following theorem.

\begin{thm} \label{thm:RD}
Let $\psi$ be the $\delta$-trace on a finite dimensional C$^\ast$-algebra $B$ and let $C_{n}(\G) = \text{span}\{(\omega_{\eta,\xi} \otimes \id)U^n : \xi, \eta \in H_n\} \subset C_u(\G)$ be the subspace of matrix elements of the $n$th irreducible representation of $\G$.  Then there exists a constant $D > 0$ (depending only on $\dim B$) such that \[\|\lambda(a)\|_{L^\infty(\G)} \le D(2n+1) \|\Lambda_h(a)\|_{L^2(\G)} \qquad (n \in \N, a \in C_{n}(\G)).\]  In particular, each orthogonal projection $P_n:L^2(\G) \to L^2_n(\G)$ is bounded as a map from $L^2(\G)$ into $C_r(\G)$ with $\|P_n\|_{L^2 \to L^\infty} \le D(2n+1)$.   
\end{thm}

\begin{rem} 
Note that the O$(2n+1)$ growth rate of the constants in Theorem \ref{thm:RD} is optimal.  To see this, consider the character $\chi_n \in C_{n}(\G)$ of the irreducible representation $U^n$.  Using the notation and arguments from the proof of Proposition \ref{prop:spectrum}, it follows that $\|\lambda(\chi_n)\| = \sup_{t \in [0,4]}|\Pi_n(t)| = \sup_{t \in [0,4]}|S_{2n}(\sqrt{t})|= S_{2n}(2) = 2n+1$.  Noting that $\|\Lambda_h(\chi_n)\|_{L^2(\G)} = 1$, we see that the O$(2n+1)$ bound is actually obtained by the sequence $\{\lambda(\chi_n)\}_{n \in \N}$.       
\end{rem}

Our proof of Theorem \ref{thm:RD} is based on the ideas in Section 4 of \cite{Ve}, where property RD is proved for the duals of free orthogonal and free unitary quantum groups.  The main difference between the present situation and the one in \cite{Ve} is that for quantum automorphism groups, the fusion rules $U^n\boxtimes U^k \cong \bigoplus_{r=0}^{2\min\{n,k\}}U^{k+n-r}$ imply that each tensor product $U^n\boxtimes U^k$ contains both \textit{even} and \textit{odd} length subrepresentations, whereas in the orthogonal case all subrepresentations of $U^n\boxtimes U^k$ have the same parity.  This has the effect of complicating some norm estimates.    

To begin the proof of Theorem \ref{thm:RD}, note that it suffices to show that there is a constant $D>0$ (depending only on $\dim B$) such that for each $n \in \N$ and each $a \in C_{n}(\G)$, 
\begin{align} \label{eqn:RD_reduction}
\|P_l  \lambda(a) P_k\|_{\mc B(L^2(\G))} \le D\|\Lambda_h(a)\|_{L^2(\G)} \qquad (k,l \in \N).
\end{align}
Indeed, assume that \eqref{eqn:RD_reduction} holds and take $n \in \N$,  $a \in C_{n}(\G)$ and $\xi \in L^2(\G)$.  Using the fusion rules from Theorem \ref{thm:fusion_rules_QAG}, it is not difficult to see that the sets $A_{l,n} = \{k \in \N: \ U^l \subset U^n \boxtimes U^k \}$ and $B_{k,n} = \{l \in \N: \ U^l \subset U^n \boxtimes U^k \}$ both have cardinality at most $2n+1$. Therefore
\begin{align*}
&\|\lambda(a)\xi\|_2^2 = \sum_{l \in \N}\|P_l\lambda(a)\xi\|_2^2 = \sum_{l \in \N}\Big\|\sum_{k \in \N: \ U^l \subset U^n \boxtimes U^k}P_l\lambda(a)P_k\xi\Big\|_2^2 \\
&\le \sum_{l \in \N} \Big( \sum_{k \in \N: \ U^l \subset U^n \boxtimes U^k} \|P_k\lambda(a)P_k\xi\|\Big)^{2} \le \sum_{l \in \N} \Big( \sum_{k \in \N: \ U^l \subset U^n \boxtimes U^k} D\|\Lambda_h(a)\|_2\|P_k\xi\|_2\Big)^{2} \\
&\le D^2\|\Lambda_h(a)\|_2^2\sum_{l \in \N}  \big|A_{l,n}\big| \Big(\sum_{k \in \N: \ U^l \subset U^n \boxtimes U^k} \|P_k\xi\|_2^2\Big) \\
&\le D^2\|\Lambda_h(a)\|_2^2 (2n+1) \sum_{k \in \N}\Big(\sum_{l \in \N: \ U^l \subset U^n \boxtimes U^k} \|P_k\xi\|_2^2\Big) \\
&= D^2\|\Lambda_h(a)\|_2^2 (2n+1) \sum_{k \in \N}\big|B_{k,n}\big|\|P_k\xi\|_2^2 \le D^2\|\Lambda_h(a)\|_2^2 (2n+1)^2  \|\xi\|_2^2, \end{align*}
which yields $\|\lambda(a)\|_{L^\infty(\G)} \le D(2n+1) \|\Lambda_h(a)\|_{L^2(\G)}$.

The remainder of this section is devoted to proving \eqref{eqn:RD_reduction}.  To do this, we utilize the Fourier transform defined by Podle\'s and Woronowicz in \cite{PoWo}.  Let $\G$ be a compact quantum group and let  $\mathbb V  \in M(c_0(\hG) \otimes C_u(\G))$ be the unitary given by \eqref{eqn:multunitary}.  Define 
\begin{align} \label{eqn:FT}
\mathcal F: C_{\text{alg}}(\G) \to \C[\hG]; \qquad \mathcal F(a) = (\id \otimes ah) \mathbb V^* \qquad (a \in C_u(\G)),
\end{align}  
where $ah \in C_u(\G)^*$  is given by $ah(b) := h(ba)$, ($b \in C_u(\G)$). 
Using the descriptions of $h$ and $\hat h_L$ given by \eqref{eqn:Haar} and  \eqref{eqn:weights}, it is readily checked that $\mathcal F: C_{\text{alg}}(\G) \to \C[\hG]$ is a linear bijection with inverse 
\begin{align} \label{eqn:IFT}
\mathcal F^{-1}:\C[\hG] \to C_{\text{alg}}(\G); \qquad \mathcal F^{-1}(x) = (x\hat h_L \otimes \id) \mathbb V \qquad (x \in \C[\hG]), 
\end{align}
where $x\hat h_L (y):= \hat h_L(yx)$, ($y \in \ell^\infty(\hG)$).  Indeed, a simple calculation shows that $\mc F (u_{ij}^\alpha) = m_{\alpha}^{-1}\sum_{1 \le r \le d_\alpha}(Q_\alpha^{-1})_{ir}e_{jr}^\alpha$, where $\{e_{ij}^\alpha\}$ denotes the system of matrix units for $\mc B(H_\alpha)$ associated to the matrix representation $U^\alpha = [u^\alpha_{ij}]$. This calculation also shows that $\mc F$ extends to a unitary isomorphism $\mathcal F : L^2(\G) \to \ell^2(\hG) := L^2(\C[\hG], \hat h_L)$.  

Recall that one can define a \textit{convolution product} $\Conv: \C[\hG] \otimes_{\text{alg}} \C[\hG] \to \C[\hG]$ using the dual coproduct $\hat\Delta$ as follows. Given $x,y \in \C[\hG]$, $\Conv(x \otimes y) \in \C [\hG]$ is defined to be the unique element, denoted by $x*y$,    
such that $(x*y)\hat h = (x\hat h \otimes y\hat h) \circ \hat \Delta$.  With respect to the usual algebra structure on $C_{\text{alg}}(\G)$ and the convolution product on $\C[\G]$, the Fourier transform $\mc F: C_{\text{alg}}(\G) \to (\C[\hG], \Conv)$  becomes an algebra isomorphism.  Indeed, since $(\hat \Delta \otimes \id)\mathbb V = \mathbb V_{13}\mathbb V_{23}$, we have
\begin{align*}
\mc F^{-1}(x*y) &= (x\hat h_L \otimes y\hat h_L \otimes id)(\hat \Delta \otimes \id)\mathbb V =  (x\hat h_L \otimes y\hat h_L \otimes id)(\mathbb V_{13}\mathbb V_{23}) \\
&=  (x\hat h_L \otimes \id)\mathbb V \cdot (y\hat h_L \otimes id)\mathbb V = \mc F^{-1} x \mc F^{-1} y.
\end{align*}   
As a final useful observation, we note that $\mc F P_l \mc F^{-1} = \hat P_l$, where $\{\hat P_l\}_{l \in \N}$ are the canonical minimal central 
projections in $\ell^\infty(\hG) = \prod_{l \in \N}\mc B (H_l) \subset \mc B(\ell^2(\hG))$.  In view of the above properties of the Fourier transform, we can transfer our required condition \eqref{eqn:RD_reduction} to an equivalent one on the dual quantum group $\hG$.  Namely, we must find a $D>0$ (depending only on $\dim B$) such that
\begin{align} \label{eqn:RD_reduction2}
\|\hat P_l (x*y)\|_{\ell^2(\hG)} \le D \|x\|_{\ell^2(\hG)} \|y\|_{\ell^2(\hG)} \qquad (x \in \hat P_n \C[\hG], \ y \in \hat P_k \C[\hG], \ l \in \N ).
\end{align}
The following useful result provides an alternate expression for the left-hand side of \eqref{eqn:RD_reduction2} and can be found in \cite{Ve}, Lemma 4.6.  For the convenience of the reader we include a proof.

\begin{lem} \label{lem:RD_L2norm}
Let $\G = \G_{\text{aut}}(B,\psi)$, where $\psi$ is a $\delta$-form on $B$, and let $w \in \hat P_n \C[\hG] \otimes \hat P_k \C[\hG]$. Then for any $l \in \N$ such that $U^l \subset U^n \boxtimes U^k$, we have \[
\|\hat P_l \Conv(w)\|_{\ell^2(\hG)} = \big( m_{n}m_km_l^{-1} \big)^{1/2} \|\hat\Delta(\hat P_l)w \hat\Delta(\hat P_l) \|_{\ell^2(\hG) \otimes \ell^2(\hG)}. \] 
\end{lem}

\begin{proof}
Since the inclusion $U^l \subset U^n \boxtimes U^k$ is  multiplicity-free by Theorem \ref{thm:fusion_rules_QAG}, there exists a unique $z \in \hat P_l\C[\hG]$ such that $\hat \Delta(\hat P_l) w \hat \Delta(\hat P_l)=\hat\Delta(z)(\hat P_n \otimes \hat P_k)$. (I.e., $z= C_{(n,k,l)}^{-1} (\rho_l^{n\boxtimes k})^*w\rho_l^{n\boxtimes k}$ using the notation of \eqref{eqn:rho}--\eqref{eqn:norm_intertwiner}). Then from the definition of convolution, we have $\hat P_l \Conv (w) = \Conv (\hat\Delta(\hat P_l)w \hat\Delta(\hat P_l)) = \Conv (\hat\Delta(z)(\hat P_n \otimes \hat P_k))  = z(\hat P_n *\hat P_k) = \lambda z$ for some $\lambda \in \C$.  The last equality follows because $\hat P_n *\hat P_k$ is evidently central. This yields
\begin{align*}
\|\hat P_l \Conv (w)\|_{\ell^2(\hG)}^2 &=\hat h\big( (\hat P_l \Conv (w))^*\hat P_l  \Conv (w)\big) =\bar \lambda \hat h \big(z^*  \Conv (w) \big) \\
&= \bar \lambda  \hat h \otimes \hat h \big(\hat \Delta(z)^* w \big)  = \bar \lambda  \hat h \otimes \hat h \big((\hat P_n \otimes \hat P_k)\hat\Delta(z)^*w\hat\Delta (\hat P_l) \big) \\
&=\bar \lambda  \hat h \otimes \hat h \big((\hat\Delta(\hat P_l)w \hat\Delta(\hat P_l))^*w \hat\Delta (\hat P_l) \big) =  \bar \lambda  \|\hat\Delta(\hat P_l)w \hat\Delta(\hat P_l) \|_{\ell^2(\hG)^{\otimes 2}}      
\end{align*}
Finally, the appropriate value of $\lambda$ is obtained from the equation \begin{align*}\lambda m_l^2 &= \lambda h_L(\hat P_l) = h_L(\hat P_l (\hat P_n *\hat P_k)) = m_n m_k (\text{Tr}_n \otimes \text{Tr}_k)\big((Q_n\hat P_n\otimes Q_k\hat P_k)\hat \Delta(\hat P_l)\big) \\
&= m_nm_k (\text{Tr}_n \otimes \text{Tr}_k)\big(\hat \Delta(Q_l\hat P_l)\big)  = m_nm_k \text{Tr}_{l}(Q_l) =  m_nm_km_l.
\end{align*}
\end{proof}

We now restrict to the unimodular case, so that $Q_k = \id_{H_k}$ and $m_k = d_k = \dim U^k$ for all $k$.  In this case, we have, $\|x\|_{\ell^2(\hG)} = d_{n}^{1/2}\|x\|_{\mc{HS}(H_n)}$ and  $\|\hat\Delta(\hat P_l)(x \otimes y) \hat\Delta(\hat P_l) \|_{\ell^2(\hG) \otimes \ell^2(\hG)} = (d_nd_k)^{1/2}\|\hat\Delta(\hat P_l)(x \otimes y) \hat\Delta(\hat P_l) \|_{\mc{HS}(H_n\otimes H_k)}$  for each $x \in \hat P_n \C[\hG]$, $y \in \hat P_k\C[\hG]$ (where $\|\cdot\|_{\mc{HS}(\cdot)}$ denotes the usual Hilbert-Schmidt norm).  Combining this fact with Lemma \ref{lem:RD_L2norm}, we can recast our required inequality \eqref{eqn:RD_reduction2} in terms of  Hilbert-Schmidt norms.  

\begin{prop} \label{prop:RD_HS}   There is a constant $D >0$ depending only on $\dim B$ such that for every $n,k,l \in \N$,
\[
\|\hat\Delta(\hat P_l)(x \otimes y) \hat\Delta(\hat P_l)\|_{\mc{HS}(H_n \otimes H_k)} \le D \big( d_l d_n^{-1}d_k^{-1}\big)^{1/2} \|x\|_{\mc{HS}(H_n)} \|y\|_{\mc{HS}(H_k)},  
\]  for all $x \in \hat P_n \C[\hG]$ and $y \in \hat P_k \C[\hG]$.
\end{prop}

Of course if $l \in \N$ is such that $U^l$ is \textit{not} a subrepresentation of $U^n\boxtimes U^k$, then the inequality in Proposition \ref{prop:RD_HS} is trivially valid since the left-hand side is then zero.  So to proceed towards proving Proposition \ref{prop:RD_HS}, let us fix an inclusion $U^l \subset U^n \boxtimes U^k$.  Recall that we defined in equation  \eqref{eqn:rho} a morphism $\rho_{l}^{n \boxtimes k} \in \Mor(U^l, U^n \boxtimes U^k) \subset \text{TL}_{2l,2n+2k}(\delta)$ which has the property that  $C_{(n,k,l)}^{-1/2}\rho_{l}^{n \boxtimes k}$ is an isometry from $H_l$ onto $\hat \Delta (\hat P_{l})(H_n \otimes H_k)$, where $C_{(n,k,l)}$ is the constant defined in \eqref{eqn:norm_intertwiner}.  We will use the morphism $\rho_{l}^{n \boxtimes k}$ to obtain the required norm estimate in Proposition \ref{prop:RD_HS}.  To do this, we first need the following lemma which gives a uniform estimate on the size of $C_{(n,k,l)}$.

\begin{lem} \label{lem:bounds_constants}
With the notation above, there exists a constant $D_0 > 0$ (independent of $n,k,l$ and only depending on $\dim B$)  such that $D_0 \le C_{(n,k,l)} \le 1$.
\end{lem}

\begin{proof}
Since $\rho_{l}^{n \boxtimes k}$ is a contraction, the upper bound $C_{(n,k,l)} \le 1$ is immediate.  The lower bound can be obtained by analyzing the expansion of $C_{(n,k,l)}$ in terms of the $q$-numbers.  This, however, has already been done by Vaes and Vergnioux in  \cite{VaVe}, Lemma A.6.  Note that in \cite{VaVe}, the authors work with a concrete representation of the Temperley-Lieb algebras which is generally different from the one considered here.  However, since all representations under consideration are faithful, their estimates transfer directly to our context.      
\end{proof} 

\begin{proof}[Proof of Proposition \ref{prop:RD_HS}]
With $U^l \subset U^n \boxtimes U^k$ fixed, let $0 \le r\le 2\min\{n,k\}$ be such that $l=k+n-r$.  From Lemma \ref{lem:bounds_constants} and the discussion preceding it, it follows that
\begin{align*}\|\hat\Delta(\hat P_l)(x \otimes y) \hat\Delta(\hat P_l)\|_{\mc{HS}(H_n \otimes H_k)} &= C_{(n,k,l)}^{-1}\big\|(\rho_l^{n \boxtimes k})^*(x \otimes y)\rho_l^{n \boxtimes k}\big\|_{\mc{HS}(H_l)} \\
& \le D_0^{-1} \big\|(\rho_l^{n \boxtimes k})^*(x \otimes y)\rho_l^{n \boxtimes k}\big\|_{\mc{HS}(H_l)}.
\end{align*}
We now consider two possible cases: either $r$ is odd or even.  

Suppose first that $r = 2s+1$ is odd.  Identify $H_l$ with the highest weight subspace $\hat\Delta(\hat P_l)(H_{n-s-1} \otimes H_{k-s})$, and $H_n$, $H_k$ with the highest weight subspaces of $H_{n-s-1} \otimes H_{1}^{\otimes (s+1)}$ and $H_{1}^{\otimes s} \otimes H_{k-s}$, respectively.  Now recall Section \ref{expmodels} where we fixed an orthonormal basis $\{e_i\}_{i=1}^{d_1}$ for $H_1$, a unitary $F_1$ such that $(F_1 \otimes 1)\overline{U^1}(F_1^{-1} \otimes 1) = U^1$ and an isometric morphism $t_2 = d_1^{-1/2}\sum_i e_i\otimes F_1e_i \in \Mor(1,U^1 \boxtimes U^1)$.  Let $e_{ij}$ and $f_{ij}$ denote the standard matrix units for $\mc B(H_1)$ relative to the orthonormal bases $\{e_i\}_i$ and $\{f_i = F_1e_i\}_i$, respectively.  Given functions $i,j:[s] \to [d_1]$, we will also write $e_i$ and $f_i$ for the tensors $e_{i(1)}\otimes \ldots e_{i(s)}$ and $f_{i(1)}\otimes \ldots \otimes f_{i(s)} \in H_1^{\otimes s}$, respectively, and write $e_{ij}$ and $f_{ij}$ for the matrix units $e_{i(1)j(1)}\otimes \ldots \otimes e_{i(s)j(s)}$ and $f_{i(1)j(1)}\otimes \ldots \otimes f_{i(s)j(s)} \in \mc B(H_1^{\otimes s})$, respectively.  With these identifications, we can uniquely express \[x = \sum_{\substack{1 \le i_0,j_0 \le d_1 \\ i,j:[s] \to [d_1]}}x_{i_0,i,j_0,j}\otimes e_{i_0j_0} \otimes e_{ij}, \qquad y = \sum_{u,v:[s] \to [d_1]} f_{uv} \otimes y_{u,v}, \]
where $\{x_{i_0,i,j_0,j}\}_{i_0,i,j_0,j} \subset\mc B(H_{n-s-1})$ and $\{y_{u,v}\}_{u,v} \subset\mc B(H_{k-s})$.  Note that $\|x\|_{\mc {HS}(H_n)}^2 = \sum_{i,j,i_0,j_0} \|x_{i_0,i,j_0,j}\|_{\mc{HS}(H_{n-s-1})}^2$ and $\|y\|_{\mc {HS}(H_k)}^2 = \sum_{u,v} \|y_{u,v}\|_{\mc{HS}(H_{k-s})}^2$.  By repeatedly using the expression \eqref{eqn:t2} for $t_2$ in the recursion formula \eqref{eqn:t2k},  it follows that $[2s+1]_q^{1/2}t_{2s} = (p_{2s}\otimes p_{2s})\sum_{i:[s]\to [d_1]}e_i \otimes f_{\check{i}},$ where for any multi-index $i = (i(1), \ldots,i(s-1), i(s))$, $\check{i}:= (i(s), i(s-1), \ldots, i(1))$. Plugging this   expression for $t_{2s}$ into \eqref{oddt}, we obtain   
\begin{align*}1_{1} \otimes t_{2s+1} \otimes 1_{1} =[2s+2]_q^{-1/2}[2]_q^{-1/2}(1_1 \otimes p_{2s+1} \otimes p_{2s+1} \otimes 1_1) 
 \big(1_{2} \otimes \sum_{i:[s] \to [d_1]} e_i \otimes f_{\check{i}} \otimes 1_{2}\big)m^*,  
\end{align*} 
which then yields
\begin{align*}
 &(\rho_l^{n \boxtimes k})^*(x \otimes y)\rho_l^{n \boxtimes k} \\
&= p_{2l}(1_{2n-2s-1} \otimes t_{2s+1}^* \otimes 1_{2k-2s-1}) \\ &\qquad  \times \sum_{\substack{i,j,u,v:[s]\to[d_1] \\
1 \le i_0,j_0 \le d_1}}(x_{i_0,i,j_0,j} \otimes e_{i_0j_0} \otimes e_{ij} \otimes f_{uv} \otimes  y_{u,v})  (1_{2n-2s-1} \otimes t_{2s+1} \otimes 1_{2k-2s-1})p_{2l} \\
&=[2]_q^{-1}[2s+2]_q^{-1}p_{2l}\sum_{\substack{i,j:[s]\to[d_1] \\
1 \le i_0,j_0 \le d_1}}\big(x_{i_0,i,j_0,j} \otimes (m \otimes 1_{2k-2s-2})(e_{i_0j_0} \otimes y_{\check{i},\check{j}}) (m^* \otimes 1_{2k-2s-2})\big)p_{2l} \\
&= [2]_q[2s+2]_q^{-1}p_{2l}\sum_{\substack{i,j:[s]\to[d_1] \\
1 \le i_0,j_0 \le d_1}}\big(x_{i_0,i,j_0,j} \otimes (\rho^{1 \boxtimes (k-s)}_{k-s})^*(e_{i_0j_0} \otimes y_{\check{i},\check{j}})\rho^{1 \boxtimes (k-s)}_{k-s}\big)p_{2l}.
\end{align*}
Note that in the last equality above, we have also used that $(p_2\otimes p_{2k-2s})(m^*\otimes 1_{2k-2s-2})p_{2k-2s}= [2]_q\rho_{k-s}^{1 \boxtimes(k-s)} $, which follows from  \eqref{eqn:rho}.  If we now take the Hilbert-Schmidt norm of the above expression and use the triangle and Cauchy-Schwarz inequalities, we obtain
\begin{align*}
& \|(\rho_l^{n \boxtimes k})^*(x \otimes y)\rho_l^{n \boxtimes k}\|_{\mc {HS}(H_l)} \\
& \le [2]_q[2s+2]_q^{-1}\Big\|\sum_{\substack{i,j:[s]\to[d_1] \\
1 \le i_0,j_0 \le d_1}}x_{i_0,i,j_0,j} \otimes (\rho^{1 \boxtimes (k-s)}_{k-s})^*(e_{i_0j_0} \otimes y_{\check{i},\check{j}})\rho^{1 \boxtimes (k-s)}_{k-s}\big)\Big\|_{\mc {HS}(H_{n-s-1} \otimes H_{k-s})} \\
&\le [2]_q[2s+2]_q^{-1} \Big(\sum_{i,j,i_0,j_0} \|x_{i_0,i,j_0,j}\|_{\mc {HS}(H_{n-s-1})}^2 \Big)^{1/2} \\ 
& \qquad \times \Big(\sum_{i,j,i_0,j_0} \|(\rho^{1 \boxtimes (k-s)}_{k-s})^*(e_{i_0j_0} \otimes y_{\check{i},\check{j}})\rho^{1 \boxtimes (k-s)}_{k-s}\|_{\mc {HS}(H_{k-s})}^2 \Big)^{1/2}  \\
&= [2]_q[2s+2]_q^{-1} \|x\|_{\mc {HS}(H_n)} \\
& \qquad \times C_{(1,k-s,k-s)} \Big(\sum_{i,j,i_0,j_0} \|\hat\Delta(\hat P_{k-s})(e_{i_0j_0} \otimes y_{\check{i},\check{j}})\hat\Delta(\hat P_{k-s})\|_{\mc {HS}(H_1 \otimes H_{k-s})}^2 \Big)^{1/2} \\
&\le  [2]_q[2s+2]_q^{-1} \|x\|_{\mc {HS}(H_n)}  \Big(\sum_{i,j,i_0,j_0} \|e_{i_0j_0} \otimes y_{\check{i},\check{j}}\|_{\mc {HS}(H_1 \otimes H_{k-s})}^2 \Big)^{1/2} \\
&= [2]_q[2s+2]_q^{-1} \|x\|_{\mc {HS}(H_n)}  \Big(\sum_{i,j,i_0,j_0} \| y_{\check{i},\check{j}}\|_{\mc {HS}(H_{k-s})}^2 \Big)^{1/2} \\
&= [2]_q[3]_q[2s+2]_q^{-1} \|x\|_{\mc {HS}(H_n)} \|y\|_{\mc {HS}(H_k)}.
\end{align*}

Now consider the second case where $r=2s$ is even.  We will only sketch this situation, as it is similar to (and in fact simpler than) the previous case.  We now identify $H_l$ with the highest weight subspace $\hat\Delta(\hat P_l)(H_{n-s} \otimes H_{k-s})$, and $H_n$, $H_k$ with the highest weight subspaces of $H_{n-s} \otimes H_{1}^{\otimes s}$ and $H_{1}^{\otimes s} \otimes H_{k-s}$, respectively.  Using the notation from the previous case, write       
\[x = \sum_{i,j:[s] \to [d_1]}x_{i,j}\otimes e_{ij}, \qquad y = \sum_{u,v:[s] \to [d_1]} f_{uv} \otimes y_{u,v}, \]
where $\{x_{i,j}\}_{i,j} \subset\mc B(H_{n-s})$ and $\{y_{u,v}\}_{u,v} \subset\mc B(H_{k-s})$.  Since $[2s+1]_q^{1/2} t_{2s} = (p_{2s} \otimes p_{2s})
 \sum_{i:[s] \to [d_1]} e_{i}\otimes f_{\check{i}}$, it is easy to verify that
\begin{align*}
&(\rho_l^{n \boxtimes k})^*(x \otimes y)\rho_l^{n \boxtimes k} \\
&= p_{2l}(1_{2n-2s} \otimes t_{2s}^* \otimes 1_{2k-2s})\Big(\sum_{i,j,u,v:[s] \to [d_1]} x_{i,j} \otimes e_{ij} \otimes f_{uv} \otimes y_{u,v}\Big)(1_{2n-2s} \otimes t_{2s} \otimes 1_{2k-2s})p_{2l} \\
&=[2s+1]^{-1} p_{2l}\Big(\sum_{i,j:[s] \to [d_1]} x_{i,j} \otimes y_{\check{i}, \check{j}}\Big)p_{2l}.
\end{align*}
Taking the Hilbert-Schmidt norm then yields \[\|(\rho_l^{n \boxtimes k})^*(x \otimes y)\rho_l^{n \boxtimes k}\|_{\mc {HS}(H_l)} \le [2s+1]^{-1}_q \|x\|_{\mc {HS}(H_{n})} \|y\|_{\mc {HS}(H_{k})}.\]  

From the preceding two estimates we finally obtain the inequality \[\|\hat\Delta(\hat P_l)(x \otimes y) \hat\Delta(\hat P_l)\|_{\mc{HS}(H_n \otimes H_k)} \le [r+1]_q^{-1}[2]_q[3]_qD_0^{-1}\|x\|_{\mc {HS}(H_{n})} \|y\|_{\mc {HS}(H_{k})},  \]
where $l=n+k-r$.  To complete the proof, we therefore need to show that there exist constants $0 < D' \le D''$ (depending only on $q = q(\dim B)$) such that \[D' \le [r+1]_q^2d_ld_n^{-1}d_k^{-1} = \frac{[r+1]_q^2[2n+2k-2r+1]_q}{[2n+1]_q[2k+1]_q} \le D''.\]   
But since
\begin{align*}
&\frac{[r+1]_q^2[2n+2k-2r+1]_q}{[2n+1]_q[2k+1]_q}  
= \frac{(1-q^{2r+2})^2(1-q^{2(2n+2k-2r)+2})}{(1-q^2)(1-q^{4n+2})(1-q^{4k+2})}, 
\end{align*}
the existence of the constants $D',D''$ is clear.
\end{proof}

\subsection{Further Approximation Properties} \label{section:FAP}

We will now use the property of rapid decay and the Haagerup property for $L^\infty(\G)$ to prove some additional approximation properties for $L^\infty(\G)$, $C_r(\G)$ and the convolution algebra $L^1(\G)$.  In this section 
we will assume that $\dim B \ge 5$ (i.e., $\G$ is not co-amenable), so that $L^\infty(\G)$ is non-injective, $C_r(\G)$ is non-nuclear, and $L^1(\G)$ does not have a bounded approximate identity (see \cite{BeMuTu}, Theorem 1.1 and \cite{BeTu}, Theorem 3.1). 

Recall that a Banach space $X$ has the \textit{metric approximation property} if there exists a net of finite rank contractions $\{\Phi_t:X \to X\}_{t \in \Lambda}$ such that $\lim_t \Phi_t   = \id_X$ pointwise in norm.  If $X$ is a dual Banach space, we say that $X$ has the \textit{weak$^\ast$ metric approximation property} if there is a net of weak$^\ast$ continuous finite rank contractions $\{\Phi_t:X \to X\}_{t \in \Lambda}$ such that $\lim_t \Phi_t   = \id_X$ pointwise in the weak$^\ast$ topology.  Finally, recall that the \textit{left multiplier norm} on $L^1(\G)$ is the norm  \[
\|\omega\|_{M_l(L^1(\G))} := \sup\{\|\omega * \omega'\|_{L^1(\G)}: \|\omega'\|_{L^1(\G)}=1\} \qquad (\omega \in L^1(\G)) 
,\] where $\ast := (\Delta_r)_*$ denotes the convolution product on $L^1(\G)$.  Obviously $\|\omega\|_{M_l(L^1(\G))} = \|L_\omega\|_{\mc B(L^1(\G))} \le \|\omega\|_{L^1(\G)}$, where $L_\omega$ is the left-multiplication operator induced by $\omega \in L^1(\G)$.

\begin{thm} \label{thm:further_approx}
The following approximation properties hold for the quantum groups $\G$.
\begin{enumerate}
\item \label{MAP} $C_r(\G)$ has the metric approximation property.
\item \label{MBAI} $L^1(\G)$ has a central approximate identity $\{\omega_n\}_{n \ge 1}$ such that $\sup_n\|\omega_n\|_{M_l(L^1(\G))} \le 1$.
\item \label{WMAP} $L^\infty(\G)$ has the weak$^\ast$ metric approximation property.
\end{enumerate}
\end{thm}

Our proof of this theorem uses standard truncation techniques for completely positive maps using property RD.  Compare with \cite{Ha} and \cite{Br}. 

\begin{proof}[Proof of \eqref{MAP}]
We will use the notation and results from Theorem \ref{thm:HAP_Kac} and its proof.  Let $\{\Phi_t\}_{t \in [t_0, \dim B)} \subset \mc{CB}(L^\infty(\G))$ be the net of normal unital completely positive maps constructed in the proof of Theorem \ref{thm:HAP_Kac}, whose $L^2$-extensions are given by \[\hat{\Phi}_t = \sum_{k \in \N} \frac{\Pi_k(t)}{\Pi_k(\dim B)} P_k.\]  Recall that by Theorem \ref{thm:average_Kac}, $\Phi_t$ restricts to a unital completely positive map on $C_r(\G)$. Moreover, since $\lim_{t \to \dim B} \frac{\Pi_k(t)}{\Pi_k(\dim B)} = 1$ and $\|\Phi_t\|_{\mc B(C_r(\G))} = 1$ for all $t$, it follows that $\lim_{t \to \dim B} \Phi_t = \id_{C_r(\G)}$ pointwise in norm. 

Now, for each $n \in \N$ consider the finite rank truncation \[\hat{\Phi}_{t,n} = \sum_{k \le n } \frac{\Pi_k(t)}{\Pi_k(\dim B)} P_k \in \mc B(L^2(\G)), \] and let $\Phi_{t,n}:L^\infty(\G) \to C_r(\G)$ be the associated finite rank map determined by $\hat{\Phi}_{t,n}\circ \Lambda_h = \Lambda_h \circ \Phi_{t,n}$.  Using property RD (Theorem \ref{thm:RD}), we have $\|P_k\|_{\mc B(C_r(\G))} \le  \|P_k\|_{L^2 \to L^\infty} \le D(2k+1)$ for each $k \in \N$. Therefore 
\begin{align*}&\|\Phi_{t} - \Phi_{t,n}\|_{\mc B(C_r(\G))} = \Bigg\| \sum_{k \ge n+1} \frac{\Pi_k(t)}{\Pi_k(\dim B)} P_k \Bigg\|_{\mc B(C_r(\G))} \\
&\le \sum_{k \ge n+1}  \frac{\Pi_k(t)}{\Pi_k(\dim B)} \|P_k\|_{\mc B(C_r(\G))} \le D \sum_{k \ge n+1}  \frac{\Pi_k(t)}{\Pi_k(\dim B)}(2k+1) \\ 
&\le A(t_0)D \sum_{k \ge n+1}(2k+1)\Big(\frac{t}{\dim B}\Big)^k \longrightarrow  0 \qquad (n \to \infty, \ t \in [t_0,\dim B)).
\end{align*}
Consequently $\Phi_t = \lim_{n \to \infty} \Phi_{t,n}$ in operator norm  and $\lim_{n\to \infty}\|\Phi_{t,n}\| = \|\Phi_t\|=1$.  Therefore if we set $\tilde{\Phi}_{t,n} = \|\Phi_{t,n}\|^{-1}\Phi_{t,n}$, we obtain a family of finite rank contractions $\{\tilde{\Phi}_{t,n}\}_{t,n}$ which also satisfies $\lim_{n \to \infty} \tilde{\Phi}_{t,n} = \Phi_t$ in norm.  Since we already know that $\lim_{t \to \dim B} \Phi_t = \id_{C_r(\G)}$ in the point-norm topology, we conclude that the set $\{\tilde{\Phi}_{t,n}\}_{t,n}$ contains $\id_{C_r(\G)}$ in its point-norm closure.  We can now easily extract a sequence of finite rank contractions from $\{\tilde{\Phi}_{t,n}\}_{t,n}$ which yields the metric approximation property.  Indeed, just choose any sequence of numbers $\{t(n)\}_{n \in \N} \subset [t_0,\dim B)$ such that \[\lim_{n \to \infty} t(n) = \dim B \quad \text{and} \quad \lim_{n \to \infty} \sum_{k \ge n+1} (2k+1)\Big(\frac{t(n)}{\dim B}\Big)^k = 0.\]  Then the sequence $\{\Psi_n\}_{n \in \N} \subset \mc B(C_r(\G))$ where $\Psi_n = \tilde{\Phi}_{t(n),n}$ does the job.

{\it Proof of \eqref{MBAI}.}   Identify $L^1(\G)$ with the closure of $C_{\text{alg}}(\G)$ with respect to the norm \[\|\omega\|_{L^1(\G)}:= h(|\omega|) = \sup\{|h(\omega x)|  : x \in C_{\text{alg}}(\G), \|\lambda(x)\|_{L^\infty(\G)}=1\}.\]  Consider the sequence $\{\omega_n\}_{n \in \N} \subset Z(L^1(\G))$, where $\omega_n = \|\Phi_{t(n),n}\|^{-1}\sum_{k \le n} \Pi_k(t(n))\chi_k$ and $\{t(n)\}_{n \in \N}$ is the sequence chosen in the proof of part \eqref{MAP}.  Since $\lim_n t(n) = \dim B$, $\lim_n\|\Phi_{t(n),n}\| = 1$ and $\chi_k* u_{ij}^l  = u_{ij}^l * \chi_k = \frac{\delta_{k,l}}{\Pi_l(\dim B)} u_{ij}^l$ for each irreducible matrix element $u_{ij}^l \in C_{\text{alg}}(\G) \subset L^1(\G)$, it follows that \begin{align} \label{eqn:AI}L_{\omega_n} \omega = \omega_n*\omega = \omega*\omega_n \longrightarrow \omega \qquad (n \to \infty, \ \omega \in C_{\text{alg}}(\G) \subset L^1(\G)).
\end{align}
To prove that \eqref{eqn:AI} holds for each $\omega \in L^1(\G)$, the usual density argument shows that it is suffices to observe that $\|\omega_n\|_{M_l(L^1(\G))}$ is uniformly bounded.  This, however, is true because an easy duality calculation shows that $ (L_{\omega_n})^* = \Psi_n \in \mc B(L^\infty(\G))$ and therefore  $\|\omega_n\|_{M_l(L^1(\G))} = \|L_{\omega_n}\| = \|\Psi_n\|_{\mc B(L^\infty(\G))} = \|\Psi_n\|_{\mc B(C_r(\G))}  = 1$.  The result now follows.

\textit{Proof of \eqref{WMAP}.}  Since the maps $\Psi_n = (L_{\omega_n})^*$ are $\sigma$-weakly continuous finite rank contractions, and $\|L_{\omega_n}\omega - \omega\|_{L^1(\G)} \to 0$ for all $\omega \in L^1(\G) = L^\infty(\G)_*$, it follows by duality that $\lim_n\Psi_n = \id_{L^\infty(\G)}$ pointwise $\sigma$-weakly.    
\end{proof}

\subsection{A Remark on Exactness} \label{section:exact}

We close this section with a few remarks on the exactness of the reduced C$^\ast$-algebras $C_r(\G)$.  We suspect that Corollary \ref{cor:exactness} below may already be known by the experts, but we could not find a reference.  

In \cite{BiDeVa}, the notion of \textit{monoidal equivalence} for compact quantum groups was introduced, and in \cite{VaVe}, Vaes and Vergnioux showed that exactness of reduced C$^\ast$-algebras of compact quantum groups is preserved by monoidal equivalence.

\begin{thm}[\cite{VaVe}, Theorem 6.1]\label{thm:monequiv_exactness}
Let $\G$ and $\G_1$ be monoidally equivalent compact quantum groups.  Then $C_r(\G)$ is exact if and only if $C_r(\G_1)$ is exact.  
\end{thm}   

Now let $(B, \psi)$ be a finite dimensional C$^\ast$-algebras equipped with a (possibly non-tracial) $\delta$-form $\psi:B \to \C$ and consider the quantum automorphism group $\G = \G_{\text{aut}}(B, \psi)$.  In \cite{DeVa}, Section 9.3, it is shown that $\G$ is monoidally equivalent to $\G_1 = \G_{\text{aut}}(M_2(\C), \text{Tr}(Q \cdot))$, where $Q \in M_2(\C)$ is any positive invertible matrix satisfying $\text{Tr}(Q^{-1}) = \delta^2$.  Using \cite{So}, we can assume that $\G_1$ is isomorphic to $SO_q(3)$ where $0 < q \le 1$ is such that $q+q^{-1} = \delta$.  Since $SO_q(3)$ is well known to be co-amenable, we have $C_r(\G_1) \cong C_u(\G_1)$ is nuclear (therefore exact) by \cite{BeMuTu}, Theorem 1.1.  This yields the following corollary.

\begin{cor} \label{cor:exactness}
Let $\G = \G_{\text{aut}}(B, \psi)$ be as above.  Then $C_r(\G)$ is an exact C$^\ast$-algebra. 
\end{cor}

\section{Algebraic Structure} \label{section:AS} 

We now turn to some algebraic questions concerning the operator algebras $L^\infty(\G)$ and $C_r(\G)$ associated to the trace-preserving quantum automorphism groups $\G = \G_{\text{aut}}(B, \psi)$.  We will first prove a factoriality and fullness result for $L^\infty(\G)$, followed by a simplicity and uniqueness of trace result for $C_r(\G)$.  We conclude this section by studying property (AO)$^+$ for $L^\infty(\G)$, and use this  prove that $L^\infty(\G)$ is always a solid von Neumann algebra.

\subsection{Factoriality and Fullness} \label{section:factoriality}

Let $(M,\tau)$ be a II$_1$-factor (with unique faithful normal trace $\tau$).  Given a sequence $\{x_n\}_{n=0}^\infty \subset M$, we say that $\{x_n\}_{n=0}^\infty$ is  \textit{asymptotically central} if $\|x_ny-yx_n\|_2 \to 0$ for each $y \in M$.  We say that $\{x_n\}_{n=0}^\infty$  is \textit{asymptotically trivial} if $\|x_n-\tau(x_n)1\|_2 \to 0$.  Recall that $M$ is said to be a \textit{full} factor if every bounded asymptotically central sequence in $M$ is asymptotically trivial.  Concerning factoriality and fullness for $L^\infty(\G)$, we obtain the following result.

\begin{thm} \label{thm:factoriality_tracial}
Let $\psi$ be the $\delta$-trace on a finite dimensional C$^\ast$-algebra $B$ and let $\G= \G_{\text{aut}}(B,\psi)$.  If $\dim B \ge 8$, then $L^\infty(\G)$ is a full type II$_1$-factor.
\end{thm}

\begin{rem}
Note that there are six cases corresponding to $\dim B = 5,6,7$ (i.e., $B=C(X_n)$ with $n=5,6,7$ and $B= C(X_n)\oplus M_2(\C)$ with $n = 1,2,3$) that are not addressed by Theorem \ref{thm:factoriality_tracial}. The question of factoriality remains open in these cases.  See Remark \ref{rem:failure} for more details.  

Note also that when $\dim B \le 4$, then $L^\infty(\G)$ is never a factor since it is commutative unless $\G = S_4^+$.  But in this case the matrix model of Banica and Collins \cite{BaCo08} (mentioned in Section \ref{section:intro}) shows that $L^\infty(S_4^+) \hookrightarrow M_4(\C) \overline{\otimes} L^\infty(SU(2))$, which forces $Z(L^\infty(S_4^+))$ to be infinite dimensional.  
\end{rem}

The arguments in our proof of Theorem \ref{thm:factoriality_tracial} are based on the ideas of \cite{VaVe}, Section 7,  where factoriality for free orthogonal quantum groups is studied.  A fundamental difference between the present situation and the one for free orthogonal quantum groups is that the fundamental representation of $\G$ is \textit{reducible} in our case.  This means that this representation does not provide an optimal set of generators for $L^\infty(\G)$ with which to check commutators (i.e., factoriality).  We remedy this by working instead with the matrix elements of the irreducible representation $U^1$ of $\G$ labeled by $1 \in \N$.   

Recall that in Section \ref{sect:reptheory} we fixed an orthonormal basis $\{e_i\}_{i=1}^{d_1}$ for the Hilbert space $H_1$, with respect to which the irreducible representation $U^1 \in \mc B(H_1) \otimes C_u(\G)$ can be written as $U^1 = [u^1_{ij}]$.  Observe that by Theorem \ref{thm:fusion_rules_QAG}, 
every irreducible representation of $\G$ is contained in some tensor power of $U^1$ and therefore $L^\infty(\G) = \{\lambda(u_{ij}^1): 1 \le i,j \le d_1\}''$.  To study factoriality and fullness, we  therefore consider the bounded linear map \begin{align*} T:L^2(\G) &\to H_1 \otimes L^2(\G) \otimes H_1 \\
T(y\xi_0) &= \sum_{i,j=1}^{d_1}  F_1e_j \otimes \Big( \lambda(u^1_{ij})y\xi_0 - y\lambda(u^1_{ij})\xi_0\Big) \otimes e_i \qquad (y \in L^\infty(\G)),  \end{align*} 
where $\xi_0 = \Lambda_h(1) \in L^2(\G)$ is the cyclic and separating vector for $L^\infty(\G)$ and $F_1 \in \mc U(d_1)$ is the unitary matrix from \eqref{eqn:t2}.  

\begin{rem} \label{rem:bbelow}The connection between the operator $T$ and the factoriality or fullness of $L^\infty(\G)$ is obvious: $L^\infty(\G)$ is a factor if and only if  $\ker T = \C\xi_0$.  If $T|_{\C\xi_0^\perp}$ is bounded below, then $L^\infty(\G)$ is moreover a full factor.  Indeed, if $C>0$ is such that $\|T\xi\|_{H_1 \otimes L^2(\G) \otimes H_1} \ge C\|\xi\|_{L^2(\G)}$ for all $\xi \in \C\xi_0^\perp$ and $\{x_n\}_{n=1}^\infty \subset L^\infty(\G)$ is an asymptotically central sequence, then \[\|x_n-h(x_n)1\|_{L^2(\G)} \le C^{-1}\|T(x_n-h(x_n)1)\|_{H_1 \otimes L^2(\G) \otimes H_1} \longrightarrow 0 \qquad (n \to \infty). \]     
\end{rem} 

\subsection{Analysis of the operator $T$} The commutator operator $T$ compares the left and right action of the operators $\{\lambda(u^1_{ij})\}_{1 \le i,j \le d_1}$ on $L^2(\G)$.  At the level of representation theory, this can be interpreted in terms of comparing certain decompositions of tensor products of irreducible representations.  This leads us to consider the following isometric morphisms (defined for $k \ge 1$ using formulas \eqref{eqn:rho}--\eqref{eqn:norm_intertwiner}):

\begin{align*} 
\phi^{(+1)}_{k,L} &=C_{(1,k+1,k)}^{-1/2} \rho_{k}^{1\boxtimes (k+1)}= \Big( \frac{[3]_q[2k+1]_q}{[2k+3]_q}\Big)^{1/2} (p_2 \otimes p_{2k+2})(t_2 \otimes 1_{2k})p_{2k}, \\
\phi^{(+1)}_{k,R} &=C_{(k+1,1,k)}^{-1/2} \rho_{k}^{(k+1)\boxtimes 1} = \Big( \frac{[3]_q[2k+1]_q}{[2k+3]_q}\Big)^{1/2} (p_{2k+2} \otimes p_{2})(1_{2k} \otimes t_2)p_{2k}, \\
\phi^{(0)}_{k,L}&=C_{(1,k,k)}^{-1/2} \rho_{k}^{1\boxtimes k} = \Big( \frac{[2k]_q}{[2k+2]_q}\Big)^{1/2}(p_2 \otimes p_{2k})(m^*\otimes 1_{2k-2})p_{2k} ,\\
\phi^{(0)}_{k,R} &=C_{(k,1,k)}^{-1/2} \rho_{k}^{k\boxtimes 1} =  \Big( \frac{[2k]_q}{[2k+2]_q}\Big)^{1/2}(p_{2k} \otimes p_{2})(1_{2k-2} \otimes m^*)p_{2k},\\
\phi^{(-1)}_{k,L} &=C_{(1,k-1,k)}^{-1/2} \rho_{k}^{1\boxtimes (k-1)} = (p_2 \otimes p_{2k-2})p_{2k},\\
\phi^{(-1)}_{k,R} &=C_{(k-1,k,k)}^{-1/2} \rho_{k}^{(k-1)\boxtimes 1} = (p_{2k-2} \otimes p_{2})p_{2k}.
\end{align*}

The following theorem shows that, relative to the decomposition $L^2(\G) \cong \bigoplus_{k \in \N} H_k \otimes H_k$ given by \eqref{eqn:L2decomp}--\eqref{eqn:L2unitarymap}, the commutator operator $T$ is block-tridiagonal and built quite naturally from the above list of isometric morphisms.      

\begin{thm} \label{thm:decomp_T}
There is a decomposition $T = T^{(+1)} + T^{(0)} +T^{(-1)}$, where for each $\alpha \in \{0, \pm 1\}$, \begin{align*} T^{(\alpha)}\xi_0 = 0, \qquad T^{(\alpha)}(H_k \otimes H_k) &\subseteq H_1 \otimes H_{k+\alpha} \otimes H_{k+\alpha} \otimes H_1 \qquad (k \ge 1),
\end{align*} and 
\begin{align*}
T^{(+1)}|_{H_k \otimes H_k} &= \Big(\frac{[2k+3]_q}{[2k+1]_q}\Big)^{1/2}\Big(\phi^{(+1)}_{k,L} \otimes \phi^{(+1)}_{k,R} - \sigma\phi^{(+1)}_{k,R} \otimes \sigma^* \phi^{(+1)}_{k,L} \Big), \\
T^{(0)}|_{H_k \otimes H_k} &=\phi^{(0)}_{k,L} \otimes \phi^{(0)}_{k,R} - \sigma\phi^{(0)}_{k,R} \otimes \sigma^* \phi^{(0)}_{k,L}, \\
T^{(-1)}|_{H_k \otimes H_k} &= \Big(\frac{[2k-1]_q}{[2k+1]_q}\Big)^{1/2}\Big(\phi^{(-1)}_{k,L} \otimes \phi^{(-1)}_{k,R} - \sigma\phi^{(-1)}_{k,R} \otimes \sigma^* \phi^{(-1)}_{k,L} \Big),
\end{align*}
where $\sigma:H_k \otimes H_1 \to H_1 \otimes H_k$ is the tensor flip map.
\end{thm}
The proof of Theorem \ref{thm:decomp_T} is rather long and tedious, so we delay it to the Appendix (Section \ref{section:appendix}).

Before commencing the proof of Theorem \ref{thm:factoriality_tracial}, we need one more remark.

\begin{rem} \label{rem:tmaps}
At various places during the proofs of Theorems \ref{thm:factoriality_tracial} and \ref{thm:decomp_T}, it will be useful to have slight modifications of the recursions  \eqref{eqn:t2o}--\eqref{eqn:t2} for the isometies $t_{2k} \in \Mor(1,U^k\boxtimes U^k)$.  Namely, we claim that $t_2 = [3]_q^{1/2}(1_2 \otimes p_2)m^*\nu =  [3]_q^{1/2}(p_2 \otimes 1_2)m^*\nu$ and more generally,
\begin{align} \label{eqn:1}t_{2k} &=\Big(\frac{[2k-1]_q[3]_q}{[2k+1]_q}\Big)^{1/2} (1_{2k} \otimes p_{2k})(1_{2k-2} \otimes t_{2} \otimes 1_{2k-2})t_{2k-2} \\
\label{eqn:2}&= \Big(\frac{[2k-1]_q[3]_q}{[2k+1]_q}\Big)^{1/2} (p_{2k} \otimes 1_{2k})(1_{2k-2} \otimes t_{2} \otimes 1_{2k-2})t_{2k-2}. \end{align}
While the above equalities are perhaps ``diagrammatically'' obvious, we can nonetheless algebraically verify these claims as follows.  If $k \ge 2$, let $\tilde{t}_{2k} \in \Mor(1, U^{(k-1)} \boxtimes U^1 \boxtimes U^k)$ and $\tilde{\tilde{t}}_{2k} \in \Mor(1, U^k \boxtimes U^1 \boxtimes U^{(k-1)})$ be the morphisms appearing on the right sides of \eqref{eqn:1} and \eqref{eqn:2}, respectively.  Since $\Mor(1, U^{(k-1)} \boxtimes U^1 \boxtimes U^k)$ is one-dimensional and $t_{2k} \in \Mor(1,U^k \boxtimes U^k) \subseteq \Mor(1, U^{(k-1)} \boxtimes U^1 \boxtimes U^k)$, we must have $\tilde{t}_{2k} = z t_{2k}$ for some $z \in \C$.  Moreover, since $\overline{z} = \tilde{t}_{2k}^*t_{2k} = t_{2k}^*t_{2k} =1$, we must have $z = 1$.  A similar argument gives $\tilde{\tilde{t}}_{2k} = t_{2k}$.  The case of $k=1$ goes similarly.  
\end{rem}

We will now use the tridiagonal decomposition of $T$ given by Theorem \ref{thm:decomp_T} to show that if the dimension of $B$ is sufficiently large, then $T$ is bounded below on $\C\xi_0^\perp \subset L^2(\G)$.  By Remark \ref{rem:bbelow}, this will prove that $L^\infty(\G)$ is a full factor.  To obtain this bound, we actually prove that $T^{(+1)}|_{\C\xi_0^\perp}$ is bounded below by some constant which tends to $\infty$ as $\dim B \to \infty$, while $(T^{(0)} + T^{(-1)})|_{\C\xi_0^\perp}$ is bounded above by a constant which is uniformly bounded as a function of $\dim B$.  An application of the triangle inequality will then show that $T|_{\C\xi_0^\perp}$ is bounded below for $\dim B$ sufficiently large.  

\begin{notat} Given $\xi \in L^2(\G)$, denote by $\xi_k$ the component of $\xi$ belonging to the subspace $L^2_k(\G) \cong H_k \otimes H_k$, and let $\xi^0 = \xi - \langle \xi_0|\xi \rangle\xi_0$ be the orthogonal projection of $\xi$ onto $\C\xi_0^\perp.$  We then have $T \xi = T\xi^0 = T^{(+1)}\xi^0 + (T^{(0)} + T^{(-1)})\xi^0$.  Similarly, given $\eta \in H_1 \otimes L^2(\G) \otimes H_1$, let $\eta_k$ be the component of $\eta$ belonging to the subspace $H_1 \otimes H_k \otimes H_k \otimes H_1$.
\end{notat}

We start by finding a uniform upper bound for the norm of  $(T^{(0)} + T^{(-1)})|_{\C\xi_0^\perp}$.

\begin{lem} \label{lem:upper_bound}
\[\big\|(T^{(0)} + T^{(-1)})|_{\C\xi_0^\perp}\big\| \le2(1+q).\]
\end{lem}

\begin{proof}
Theorem \ref{thm:decomp_T} implies that for any $\xi \in L^2(\G)$, 
\begin{align*}
\|T^{(0)}\xi^0\|^2 &= \sum_{k=1}^\infty \|(T^{(0)}\xi^0)_k\|^2 = \sum_{k=1}^\infty \|T^{(0)}\xi_k\|^2 \le   \sum_{k=1}^\infty\big\|T^{(0)}|_{H_k \otimes H_k}\big\|^2\|\xi_k\|^2 \\
&\le  \sum_{k=1}^\infty 4 \|\xi_k\|^2 =  4 \|\xi^0\|^2, \\
\|T^{(-1)}\xi^0\|^2 &= \sum_{k=0}^\infty \|(T^{(-1)}\xi^0)_k\|^2 = \sum_{k=1}^\infty \|T^{(-1)}\xi_k\|^2 \le  \sum_{k=1}^\infty\big\|T^{(-1)}|_{H_k \otimes H_k}\big\|^2\|\xi_k\|^2  \\
&\le \sum_{k=1}^\infty\frac{4[2k-1]_q}{[2k+1]_q}\|\xi_k\|^2 \le  4q^2\|\xi^0\|^2, 
\end{align*}
where in the last line we have used the fact that 
\[k \mapsto \frac{[2k-1]_q}{[2k+1]_q} = q^2\Big(\frac{1-q^{4k-2}}{1-q^{4k+2}}\Big) \quad \text{is increasing and} \quad \sup_{k}  \frac{[2k-1]_q}{[2k+1]_q} = q^2.\]
The result now follows from the triangle inequality.
\end{proof}

Next we deal with more difficult task of bounding $T^{(+1)}|_{\C\xi_0^\perp}$ from below.  Fix $\xi \in L^2(\G)$.  Then $\|T^{(+1)}\xi^0\|^2 = \sum_{k=1}^\infty \|T^{(+1)}\xi_k\|^2$ and   for each $k \ge 1$, Theorem \ref{thm:decomp_T} gives   
\begin{align*}
&\|T^{(+1)}\xi_k\|^2 \\
 &= \frac{[2k+3]_q}{[2k+1]_q}\Big\|\big(\phi^{(+1)}_{k,L} \otimes \phi^{(+1)}_{k,R} - \sigma\phi^{(+1)}_{k,R} \otimes \sigma^* \phi^{(+1)}_{k,L}\big)\xi_k \Big\|^2 \\
&=  \frac{[2k+3]_q}{[2k+1]_q} \Big(\|(\phi^{(+1)}_{k,L}\otimes \phi^{(+1)}_{k,R})\xi_{k}\|^2 + \|(\sigma\phi^{(+1)}_{k,R} \otimes \sigma^*\phi^{(+1)}_{k,L})\xi_{k}\|^2 \\
& \qquad - 2 \text{Re} \big\langle (\sigma\phi^{(+1)}_{k,R} \otimes \sigma^* \phi^{(+1)}_{k,L})^*(\phi^{(+1)}_{k,L}\otimes \phi^{(+1)}_{k,R})\xi_{k} | \xi_{k} \big\rangle\Big) \\
&= \frac{[2k+3]_q}{[2k+1]_q} \Big(2\|\xi_{k}\|^2  - 2 \text{Re} \big\langle \big((\phi_{k,R}^{(+1)})^*\sigma^*\phi_{k,L}^{(+1)}\big)\otimes \big((\phi_{k,L}^{(+1)})^*\sigma\phi_{k,R}^{(+1)}\big)\xi_{k} | \xi_{k} \big\rangle\Big) \\
&\ge \frac{[2k+3]_q}{[2k+1]_q} \Big(2\|\xi_{k}\|^2  - 2  \|(\phi_{k,L}^{(+1)})^*\sigma\phi_{k,R}^{(+1)}\|^2\|\xi_{k}\|^2\Big).
\end{align*}
Therefore 
\begin{align}\|T^{(+1)}\xi^0\|^2 &\ge 2 \sum_{k=1}^\infty\frac{[2k+3]_q}{[2k+1]_q} \big(1-\|(\phi_{k,L}^{(+1)})^*\sigma\phi_{k,R}^{(+1)}\|^2\big)\|\xi_k\|^2 \notag \\
\label{eqn:lowbd}& \ge 2\min_{k \ge 1}\Big\{\frac{[2k+3]_q}{[2k+1]_q} \big(1-\|(\phi_{k,L}^{(+1)})^*\sigma\phi_{k,R}^{(+1)}\|^2)\Big\}\|\xi^0\|^2,  \end{align} 
showing that we need to find some useful upper bounds for the norms $\|(\phi_{k,L}^{(+1)})^*\sigma\phi_{k,R}^{(+1)}\|$, where
\begin{align} \label{eqn_flip}
(\phi_{k,L}^{(+1)})^*\sigma\phi_{k,R}^{(+1)} 
&=  \frac{[3]_q[2k+1]_q}{[2k+3]_q} p_{2k}(t_2^* \otimes 1_{2k})(p_2 \otimes p_{2k+2})\sigma (1_{2k} \otimes t_2)p_{2k} \qquad (k \ge 1).
\end{align}

As a first observation, note that the na\"ive estimate $\|(\phi_{k,L}^{(+1)})^*\sigma\phi_{k,R}^{(+1)}\| \le \frac{[3]_q[2k+1]_q}{[2k+3]_q}$ is of no use since the right hand side of this inequality always exceeds $1$.  To obtain a sharper estimate, we need to unravel the right-hand side of \eqref{eqn_flip} using the recursive structure of the Jones-Wenzl projections $\{p_y\}_{y \in \N}$.  We will do this by first finding a suitable two-step recursion formula which expresses the projection $p_{2k+2}$ in terms of the projection $p_{2k}$.  Plugging this recursion into the right-hand side of \eqref{eqn_flip} will yield an expression whose norm can be more sharply bounded from above.  The required two-step recursion for $p_{2k+2}$ will be obtained by iterating twice the following (one-step) recursion for the Jones-Wenzl projections $\{p_y\}_{y \in \N}$ due to Frenkel and Khovanov (\cite{FrKh}, Equation 3.8):
\begin{align} \label{eqn_FK_recursion}
p_y = \Big(1_y - \sum_{r=1}^{y-1} (-1)^{y-r-1}\frac{[2]_q[r]_q}{[y]_q}(1_{r-1} \otimes t_1 \otimes 1_{y-r-1} \otimes t_1^*) \Big)(p_{y-1} \otimes 1_1) \qquad (y \ge 2).
\end{align}
Setting $y=2k+2$ in \eqref{eqn_FK_recursion} and then applying \eqref{eqn_FK_recursion} again to the resulting $p_{2k+1}$ term yields 
\begin{align*}
p_{2k+2} &= \Big(1_{2k+2} - \sum_{r=1}^{2k+1} (-1)^{2k+1-r}\frac{[2]_q[r]_q}{[2k+2]_q}(1_{r-1} \otimes t_1 \otimes 1_{2k+1-r} \otimes t_1^*) \Big) \\
&\qquad \times \Big(1_{2k+2} - \sum_{l=1}^{2k} (-1)^{2k-l}\frac{[2]_q[l]_q}{[2k+1]_q}(1_{l-1} \otimes t_1 \otimes 1_{2k-l} \otimes t_1^* \otimes 1_1) \Big) (p_{2k} \otimes 1_2). 
\end{align*}
Multiplying this expression by $1_2 \otimes p_{2k}$ on the left, observing that $p_{2k+2} = (1_2 \otimes p_{2k})p_{2k+2}$ and $p_{2k}\text{TL}_{y, 2k}(\delta) = \text{TL}_{2k,y}(\delta)p_{2k} =  0$ for all $y < 2k$, we see that most terms in the above sums vanish.  That is, 

\begin{align*}
p_{2k+2} &=  (1_2 \otimes p_{2k})\Big(1_{2k+2} -\frac{[2]_q}{[2k+2]_q}(t_1 \otimes 1_{2k} \otimes t_1^*) + \frac{[2]_q^2}{[2k+2]_q}(1_{1} \otimes t_1 \otimes 1_{2k-1} \otimes t_1^*)  \Big) \\
&\qquad\times \Big(1_{2k+2} - \sum_{l=1}^{2k} (-1)^{2k-L}\frac{[2]_q[l]_q}{[2k+1]_q}(1_{l-1} \otimes t_1 \otimes 1_{2k-L} \otimes t_1^* \otimes 1_1) \Big) (p_{2k} \otimes 1_2)\\
&=(1_2 \otimes p_{2k})\Bigg[1_{2k+2} + \frac{[2]_q}{[2k+1]_q}(t_1 \otimes 1_{2k-1} \otimes t_1^* \otimes 1_1) \\
&\qquad  -\frac{[2]_q^2}{[2k+1]_q}(1_{1} \otimes t_1 \otimes 1_{2k-2} \otimes t_1^* \otimes 1_1) -\frac{[2]_q}{[2k+2]_q}(t_1 \otimes 1_{2k} \otimes t_1^* )  \\
& \qquad+\frac{[2]_q[2k]_q}{[2k+1]_q[2k+2]_q}(t_1 \otimes 1_{2k-1} \otimes t_1^* \otimes 1_1) + \frac{[2]_q^2}{[2k+2]_q}(1_{1} \otimes t_1 \otimes 1_{2k-1} \otimes t_1^*) \\
&\qquad + \frac{[2]_q^3}{[2k+1]_q[2k+2]_q}((1_1 \otimes t_1 \otimes 1_1)t_1 \otimes 1_{2k-2} \otimes t_1^*(1_1 \otimes t_1^* \otimes 1_1))  \\
& \qquad - \frac{[2]_q^2[2k]_q}{[2k+1]_q[2k+2]_q}(1_{1} \otimes t_1 \otimes 1_{2k-2} \otimes t_1^* \otimes 1_1) \Bigg](p_{2k} \otimes 1_2). \\
\end{align*}
After collecting like terms, this gives
\begin{align*}
p_{2k+2}&= (1_2 \otimes p_{2k})\Bigg[\underbrace{1_{2k+2}}_{(a)} + \Big(\frac{[2]_q}{[2k+1]_q} + \frac{[2]_q[2k]_q}{[2k+1]_q[2k+2]_q}\Big)\underbrace{(t_1 \otimes 1_{2k-1} \otimes t_1^* \otimes 1_1)}_{(b)} \\
& \qquad  - \Big(\frac{[2]_q^2}{[2k+1]_q} + \frac{[2]_q^2[2k]_q}{[2k+1]_q[2k+2]_q}\Big)\underbrace{(1_{1} \otimes t_1 \otimes 1_{2k-2} \otimes t_1^* \otimes 1_1)}_{(c)}  \\
& \qquad -\frac{[2]_q}{[2k+2]_q}\underbrace{(t_1 \otimes 1_{2k} \otimes t_1^* )}_{(d)}  + \frac{[2]_q^2}{[2k+2]_q} \underbrace{(1_{1} \otimes t_1 \otimes 1_{2k-1} \otimes t_1^*)}_{(e)} \\
&\qquad+ \frac{[2]_q^3}{[2k+1]_q[2k+2]_q}\underbrace{((1_1 \otimes t_1 \otimes 1_1)t_1 \otimes 1_{2k-2} \otimes t_1^*(1_1 \otimes t_1^* \otimes 1_1))}_{(f)} \Bigg](p_{2k} \otimes 1_2).
\end{align*}

If we now insert the above expression for $p_{2k+2}$ into equation \eqref{eqn_flip} and use the fact that $p_2t_1 = 0$, we see that the contributions coming from $(b), (d)$ and $(e)$ are all zero.  Thus,

\begin{align*}
&(\phi_{k,L}^{(+1)})^*\sigma\phi_{k,R}^{(+1)} \\
&= \frac{[3]_q[2k+1]_q}{[2k+3]_q} \Bigg[\underbrace{p_{2k}(t_2^* \otimes p_{2k})\sigma (1_{2k} \otimes t_2)p_{2k}}_{ :=\alpha \text{ [resulting from } (a)]} \\
& \qquad - \Big(\frac{[2]_q^2}{[2k+1]_q} + \frac{[2]_q^2[2k]_q}{[2k+1]_q[2k+2]_q}\Big)  \\
& \qquad \qquad \times \underbrace{p_{2k}(t_2^* \otimes p_{2k})\big(1_2 \otimes (1_{1} \otimes t_1 \otimes 1_{2k-2} \otimes t_1^* \otimes 1_1)\big)\sigma (1_{2k} \otimes t_2)p_{2k}}_{:=\beta \text{ [resulting from } (c)]} \\
& \qquad + \frac{[2]_q^3}{[2k+1]_q[2k+2]_q} \\
& \qquad \qquad \times \underbrace{p_{2k}(t_2^* \otimes p_{2k})\big(1_2 \otimes ((1_1 \otimes t_1 \otimes 1_1)t_1 \otimes 1_{2k-2} \otimes t_1^*(1_1 \otimes t_1^* \otimes 1_1))\big)\sigma (1_{2k} \otimes t_2)p_{2k}}_{:=\gamma  \text{ [resulting from } (f)]}  \Bigg] \\
&= \frac{[2k+1]_q}{[2k+3]_q} p_{2k} \sigma^* p_{2k} - \Bigg(\frac{1 + \frac{[2k]_q}{[2k+2]_q} }{[2k+3]_q}\Bigg) p_{2k}(1_{2k-2} \otimes m)\sigma^*(m^* \otimes 1_{2k-2})p_{2k} \\
&\qquad  + \frac{[2]_q}{[2k+3]_q[2k+2]_q}p_{2k}\sigma p_{2k},
\end{align*}
where, to obtain the last equality, we have substituted the following three identities for the quantities $\alpha, \beta$ and $\gamma$:
\begin{align*}
[3]_q \alpha &= p_{2k}\Big(\sum_{i=1}^{d_1} e_i^* \otimes (F_1e_i)^* \otimes p_{2k}\Big)\Big(\sum_{j=1}^{d_1} e_j \otimes p_{2k} \otimes F_1e_j\Big) \\
&= p_{2k}\sum_{i=1}^{d_1}( (F_1e_i)^* \otimes p_{2k})(p_{2k} \otimes F_1e_i) =  p_{2k}\sigma^*p_{2k}, \\
[2]_q^2[3]_q \beta  &=[3]_q p_{2k}(t_2^* \otimes p_{2k})\big(1_2 \otimes m^*\otimes 1_{2k-4} \otimes m \big)\sigma (1_{2k} \otimes t_2)p_{2k} \qquad \text{(using \eqref{oddt})} \\
& = p_{2k}\Big(\sum_{i=1}^{d_1} e_i^* \otimes (F_1e_i)^* \otimes p_{2k}\Big)\big(1_2 \otimes m^*\otimes 1_{2k-4} \otimes m \big)\Big(\sum_{j=1}^{d_1} e_j \otimes p_{2k} \otimes F_1e_j\Big) \\
& = p_{2k}\sum_{i=1}^{d_1}( (F_1e_i)^* \otimes p_{2k})(1_{2k} \otimes m)((m^* \otimes 1_{2k-2})p_{2k} \otimes F_1e_i) \\
& = p_{2k}(1_{2k-2} \otimes m)\sum_{i=1}^{d_1}( (F_1e_i)^* \otimes 1_{2k+2})((m^* \otimes 1_{2k-2})p_{2k} \otimes F_1e_i) \\
&=p_{2k}(1_{2k-2} \otimes m)\sigma^*(m^* \otimes 1_{2k-2})p_{2k},
\end{align*}
and 
\begin{align*}
[3]_q[2]_q^2 \gamma &= [3]_q^2 p_{2k}(t_2^* \otimes p_{2k})\big(1_2 \otimes t_2 \otimes 1_{2k-2} \otimes t_2^*\big)\sigma (1_{2k} \otimes t_2)p_{2k} \\
& \qquad \qquad \text{(using the fact that $t_2 = [3]_q^{-1/2}[2]_q(p_2 \otimes p_2)(1_1 \otimes t_1\otimes 1_1)t_1$)}\\
&= [3]_q^2 p_{2k}\big((t_2^* \otimes 1_{2})(1_2 \otimes t_2) \otimes 1_{2k-2} \otimes t_2^*\big)\sigma (1_{2k} \otimes t_2)p_{2k} \\
&=[3]_q p_{2k}(p_2 \otimes 1_{2k-2}\otimes t_{2}^*)\sigma (1_{2k} \otimes t_2)p_{2k} \qquad \text{(since $(t_2^* \otimes 1_{2})(1_2 \otimes t_2) = [3]_q^{-1}p_2$)}\\
&=p_{2k}\Big(p_2 \otimes 1_{2k-2}\otimes \sum_{i=1}^{d_1}e_i^* \otimes (F_1e_i)^*\Big)\Big(\sum_{j=1}^{d_1}e_j \otimes p_{2k} \otimes F_1e_j\Big) \\
&=p_{2k}(p_{2} \otimes 1_{2k-2})\sum_{i=1}^{d_1}(1_{2k}\otimes e_i^*)(e_i \otimes p_{2k}) = p_{2k}\sigma p_{2k}. \\
\end{align*}
A simple application of the triangle inequality using the fact that $\|m\| = \delta = [2]_q$ now yields our desired norm estimate. 
\begin{lem} \label{lem:lowbd} For each $k \ge 1$,  
\begin{align}
\|(\phi_{k,L}^{(+1)})^*\sigma\phi_{k,R}^{(+1)}\| \le  \frac{[2k+1]_q + [2]_q^2\Big(1 + \frac{[2k]_q}{[2k+2]_q} \Big) +  \frac{[2]_q}{[2k+2]_q}}{[2k+3]_q}. 
\end{align}
\end{lem}

Returning to inequality \eqref{eqn:lowbd} with Lemma \ref{lem:lowbd} now at hand, we obtain the lower bound

\begin{align*}\|T^{(+1)}\xi^0\|^2 &\ge 2 \|\xi^0\|^2 \min_{k \ge 1} \Bigg\{ \frac{[2k+3]_q}{[2k+1]_q}- \frac{\Big([2k+1]_q + [2]_q^2\Big(1 + \frac{[2k]_q}{[2k+2]_q} \Big) +  \frac{[2]_q}{[2k+2]_q}\Big)^2}{[2k+3]_q[2k+1]_q}  \Bigg\}.
\end{align*}\
By expanding the square in the previous line and using the fact that   $\min_k \frac{[2k+3]_q}{[2k+1]_q} = q^{-2}$, $\max_k \frac{[2k+1]_q}{[2k+3]_q} = q^2$ and that $y \mapsto [y]_q$ is increasing, we finally obtain
\begin{align*}
\|T^{(+1)}\xi^0\|^2 &\ge 2 \Bigg(q^{-2} - q^2 -  \frac{ [2]_q^4(1 + q^{2})^2}{[3]_q[5]_q} - \frac{[2]_q^2}{[4]_q^2[3]_q[5]_q}  -  \frac{2[2]_q^2(1 + q^2)}{[5]_q}  \\
&- \frac{2[2]_q}{[4]_q[5]_q}  - \frac{2[2]_q^3(1+q^2)}{[4]_q[3]_q[5]_q}\Bigg) \|\xi^0\|^2 =: C(q) \|\xi^0\|^2,   \end{align*}   

If we now combine the preceding inequality with Lemma \ref{lem:upper_bound}, we have
\begin{align} \|T\xi\| &= \|T^{(+1)}\xi^0 + (T^{(0)} + T^{(-1)})\xi^0\| \ge \|T^{(+1)}\xi^0\| - \|(T^{(0)} + T^{(-1)})\xi^0\| \notag \\
\label{eqn:lowerbound} &\ge \big(C(q)^{1/2} - 2(1+q)\big)\|\xi^0\| \qquad (\xi \in L^2(\G)).\end{align}

We are now ready to prove Theorem \ref{thm:factoriality_tracial}.  The remaining ingredient we need is the following lemma, which will also be useful when studying the simplicity of $C_r(\G)$ in Section \ref{section:simplicity}.  Recall that $0 < q < 1$ was defined so that $\delta = \sqrt{\dim B}  = q+q^{-1}$.  Note that $q = q(\delta) = \frac{\delta - \sqrt{\delta^2-4}}{2}$ is a decreasing function of $\delta$.  

\begin{lem} \label{lem:increasing} Consider the function $f(\delta) := [3]_{q(\delta)}^{-1/2}(C(q(\delta))^{1/2} - 2(1+q(\delta)))$.  Then $f$ is an increasing function on the interval $[\sqrt{8},\infty)$ and $f(\sqrt{8}) > 0$.   
\end{lem} 

\begin{proof}
From the definition of $C(q)$, it follows that $f(\delta) = g(q(\delta))$ where \begin{align*}
g(q) &= \sqrt{2}\Bigg[\frac{q^{-2}}{[3]_q} - \frac{q^2}{[3]_q} -  \frac{ [2]_q^4(1 + q^{2})^2}{[3]_q^2[5]_q} - \frac{[2]_q^2}{[4]_q^2[3]_q^2[5]_q}  -  \frac{2[2]_q^2(1 + q^2)}{[3]_q[5]_q}  \\
&\qquad - \frac{2[2]_q}{[3]_q[4]_q[5]_q}  - \frac{2[2]_q^3(1+q^2)}{[4]_q[3]_q^2[5]_q}\Bigg]^{1/2} - \frac{2(1+q)}{[3]_q^{1/2}}.
\end{align*}
When $\delta = \sqrt{8}$, $q = \frac{\sqrt{8} - 2}{2} \approx 0.4142$ and direct substitution yields $f(\delta) \approx 0.1111 >0.$  It is now just an elementary calculus exercise to verify that $g$ is a decreasing function of $q \in (0, \frac{\sqrt{8} - 2}{2}]$ and therefore $f$ is an increasing function of $\delta \in [\sqrt{8}, \infty)$. 
\end{proof}

\begin{proof}[Proof of Theorem \ref{thm:factoriality_tracial}]
Since $f(\sqrt{8}) > 0$ and $f$ is increasing by Lemma \ref{lem:increasing}, inequality \eqref{eqn:lowerbound} shows that $T|_{\C\xi_0^\perp}$ is bounded below by $[3]_{q(\delta)}^{1/2}f(\delta) > 0$ whenever $\dim B = \delta^2 \ge 8$.  Thus $L^\infty(\G)$ is a full factor by Remark \ref{rem:bbelow}. 
\end{proof}

\begin{rem} \label{rem:failure}  Note that if $\delta^2 = 5,6$ or $7$, one can check that $[3]_{q(\delta)}^{1/2}f(\delta) \le -0.4386 < 0$.  So the preceding calculations yield no information about the factoriality or fullness of $L^\infty(\G)$ when $5 \le \dim B \le 7$. 
\end{rem}

\subsection{Simplicity of $C_r(\G)$} \label{section:simplicity}

We now consider the reduced C$^\ast$-algebra $C_r(\G)$ and show that, at least in most cases, it is simple with unique trace.  To do this, we adapt the ``conjugation by generators'' method used in \cite{VaVe}, Section 7.  

Write $W^1 = (\id \otimes \lambda)U^1$ where $U^1$ is the  irreducible representation of $\G$ labeled by $1 \in \N$ and consider the unital completely positive map $\Phi:C_r(\G) \to C_r(\G)$ defined by  
\begin{align*}
\Phi(x) = \frac{1}{2[3]_q}(\text{Tr} \otimes \id) \Big((W^1)^*(1 \otimes x)W^1 + W^1(1 \otimes x)(W^1)^* \Big) \qquad (x\in C_r(\G)).
\end{align*}    
Observe that $\Phi$ is $\tau$-preserving for any tracial state $\tau$ on $C_r(\G)$ and that $\Phi(\mc I) \subseteq \mc I$ for any closed two-sided ideal $\mc I \subseteq C_r(\G)$.  Moreover, a simple calculation shows that the $L^2$-extension of $\Phi$, denoted by $\hat{\Phi}$, is given by
\begin{align} \label{eqn:Phi_T} \hat\Phi\Lambda_h(a)  = \Big(\id_{L^2(\G)} - \frac{1}{2[3]_q}T^*T\Big) \Lambda_h(a) \qquad (a \in C_u(\G)), 
\end{align}   
where $T:L^2(\G) \to H_1 \otimes L^2(\G) \otimes H_1$ is the commutator map introduced in Section \ref{section:factoriality}. 
The key result of this section is the following proposition.

\begin{prop} \label{prop:converge_to_trace}
Let $\G = \G_{\text{aut}}(B,\psi)$, where $B$ is a finite dimensional C$^\ast$-algebra with $\delta$-trace $\psi$.  If $\dim B \ge 8$, then \[\lim_{n \to \infty} \Phi^n(x) = h(x)1_{C_r(\G)} \qquad (x \in C_r(\G)).\]
\end{prop}

\begin{proof}
First observe that it is sufficient to prove this result for  $x_0$ belonging to the dense Hopf $\ast$-subalgebra $\lambda(C_{\text{alg}}(\G)) \subset C_r(\G)$.  To verify this claim, assume the above proposition is true for all $x_0 \in \lambda(C_{\text{alg}}(\G))$.  Let  $x \in C_r(\G)$ and $\epsilon >0$ be arbitrary, and choose $a \in C_{\text{alg}}(\G)$ such that $x_0 = \lambda(a)$ satisfies $\|x-x_0\| \le \epsilon$.  If we apply the triangle inequality to the norm of $\Phi^n(x) - h(x)1$, we obtain 
\begin{align*}
&\limsup_{n \to \infty} \|\Phi^n(x) - h(x)1\| \\
&\le \limsup_{n \to \infty} \big( \|\Phi^n(x-x_0)\| + \|\Phi^n(x_0)-h(x_0)1\| + |h(x_0-x)|\big) \le 2 \epsilon. \end{align*}   
This proves the claim, since $x$ and $\epsilon$ were arbitrary.

Now take $a \in C_{\text{alg}}(\G)$ and consider $x_0 = \lambda(a)$.  Since $a \in  C_{\text{alg}}(\G)$, there is an $r \ge 0$ such that $x_0 \in \bigoplus_{k=0}^r \lambda(C_{k}(\G))$, where $C_{k}(\G)$ is the subspace of  $C_{\text{alg}}(\G)$ spanned by the matrix elements of the irreducible representation $U^k$.  Now let $n \in \N$ and consider $\Phi^n(x_0) - h(x_0)1 = \Phi^n(x_0-h(x_0)1)$, which evidently belongs to $\bigoplus_{k=1}^{r+2n} \lambda(C_{k}(\G))$.  Writing $\Phi^n(x_0-h(x_0)1) = \sum_{k=1}^{r+2n} \lambda(z_k)$ with $z_k \in C_{k}(\G)$ and using property RD (Theorem \ref{thm:RD}), we obtain   
\begin{align*}
&\|\Phi^n(x_0) - h(x_0)1\|_{C_r(\G)} \le \sum_{k=1}^{r+2n}\|\lambda(z_k)\|_{C_r(\G)} \le\sum_{k=1}^{r+2n} (2k+1)\|\Lambda_h(z_k)\|_{L^2(\G)} \\
&\le (2r+4n+1)\sqrt{r+2n}\Big(\sum_{k=1}^{r+2n}\|\Lambda_h(z_k)\|_{L^2(\G)}^2\Big)^{1/2} \\
&= (2r+4n+1)\sqrt{r+2n} \|\hat\Phi^n(\Lambda_h(a-h(a)1))\|_{L^2(\G)} \\
& \le (2r+4n+1)\sqrt{r+2n} \|\hat\Phi|_{\C\xi_0^\perp}\|^n \|\Lambda_h(a-h(a)1)\|_{L^2(\G)}. 
\end{align*}

Now assume that $\delta = \sqrt{\dim B} \ge \sqrt{8}$.  From \eqref{eqn:lowerbound} and Lemma \ref{lem:increasing}, it follows that $T^*T|_{\C\xi_0^\perp} \ge [3]_qf(\delta)^2\id_{\C\xi_0^\perp}$ and therefore \[ \hat \Phi|_{\C\xi_0^\perp} = \id_{\C\xi_0^\perp} - \frac{1}{2[3]_q}T^*T|_{\C\xi_0^\perp} \le \Big(1 -\frac{1}{2}f(\delta)^2\Big)\id_{\C\xi_0^\perp}  \le \Big(1 -\frac{1}{2}f(\sqrt{8})^2\Big)\id_{\C\xi_0^\perp}.  \] 
Since $f(\sqrt{8}) \approx 0.1111$, $\| \hat \Phi|_{\C\xi_0^\perp}\| \le \Big(1 -\frac{1}{2}f(\sqrt{8})^2\Big) < 1$, which gives \begin{align*}
&\limsup_{n \to \infty}\|\Phi^n(x_0) - h(x_0)1\|_{C_r(\G)} \\
&\le \lim_{n\to \infty} (2r+4n+1)\sqrt{r+2n} \Big(1 -\frac{1}{2}f(\sqrt{8})^2\Big)^n \|\Lambda_h(a-h(a)1)\|_{L^2(\G)} = 0.
\end{align*} 
\end{proof}

Simplicity and uniqueness of trace for $C_r(\G)$ are now immediate consequences (as long as $\dim B \ge 8$).

\begin{cor} \label{cor:simplicity}
Let $\psi$ be the $\delta$-trace on a finite dimensional $C^\ast$-algebra $B$ and let $\G = \G_{\text{aut}}(B, \psi)$.  If $\dim B \ge 8$, then $C_r(\G)$ is simple with unique trace.
\end{cor}

\begin{proof}
Let $\mc I \subseteq C_r(\G)$ be a non-zero closed two-sided ideal and let $a \in \mc I$ be such that $h(a) \ne 0$.  Then by Proposition
\ref{prop:converge_to_trace}, $h(a)1 = \lim_{n \to \infty} \Phi^n(a) \in \mc I$.  Therefore $\mc I = C_r(\G)$ and $C_r(\G)$ is simple.

Now let $\tau$ be any tracial state on $C_r(\G)$.  Since $\tau \circ \Phi^n = \tau$ for all $n$,   we have $\tau(x) = \lim_{n \to \infty} \tau(\Phi^n(x)) = \tau(h(x)1) = h(x)$ for all $x \in C_r(\G)$. 
\end{proof}

\subsection{Solidity} \label{section:AO} In this final section we  show that the von Neumman algebras $L^\infty(\G)$ considered in this paper always have the Akemann-Ostrand property.  Recall that a von Neumann $M \subseteq \mc B(H)$ has the \textit{Akemann-Ostrand property} (or \textit{property (AO)}) if there exist $\sigma$-weakly
dense unital C$^\ast$-subalgebras $A \subseteq M$ and $B \subseteq M'$ such that
$A$ is locally reflexive and the $\ast$-homomorphism $\pi \circ m:A \otimes_{\text{alg}}B \to \mc B(H)/\mc K(H)$ is continuous with respect to the minimal tensor norm.  Here $m:A \otimes_{\text{alg}}B \to \mc B(H)$ is the multiplication map and $\pi:\mc B(H) \to \mc B(H)/\mc K(H)$ is the canonical quotient map.  Note that an exact C$^\ast$-algebra is always locally reflexive. We now state the main result of this section.  Below we denote by $J: L^2(\G) \to L^2(\G)$ the modular conjugation operator for $L^\infty(\G) = \lambda(C_u(G))''$, so that $JL^\infty(\G)J = L^\infty(\G)'$.

\begin{thm}\label{thm:AO}
Let $\G = \G_{\text{aut}}(B,\psi)$, where $B$ is a finite dimensional C$^\ast$-algebra equipped with a $\delta$-form $\psi$.  Then there exists a unital completely positive map $\theta:C_r(\G) \otimes JC_r(\G)J \to \mc B(L^2(\G))$ such that \[\theta(a \otimes JbJ)-aJbJ \in \mc K(L^2(\G)) \qquad (a,b \in C_r(\G)). \]  In particular, $L^\infty(\G)$ has property (AO).
\end{thm}  

In \cite{Oz}, Ozawa proved that a finite von Neumann algebra $M$ with property (AO) is \textit{solid}.  That is, for any diffuse von Neumann subalgebra $A \subset M$, the relative commutant $A'\cap M$ is injective.  Recall that a II$_1$-factor $M$ is said to be \textit{prime} if for every tensor product decomposition $M\cong M_1 \overline{\otimes}M_2$, either $M_1$ or $M_2$ is finite dimensional.  It is easy to see that a solid II$_1$-factor is prime unless it is injective.  As a consequence of Corollary \ref{cor:exactness}, Theorem \ref{thm:factoriality_tracial} and the above discussion, the following result is immediate. 

\begin{cor} \label{cor:solidity_prime}
Let $\psi$ be the $\delta$-trace on $B$ and let $\G = \G_{\text{aut}}(B,\psi)$.  Then $L^\infty(\G)$ is solid and it is a prime II$_1$-factor if $\dim B \ge 8$.
\end{cor}

\begin{rem} \label{rem:no-reflexivity} 
It is perhaps worthwhile pointing out that the local reflexivity of $C_r(\G)$ is not actually required to establish the solidity of $L^\infty(\G)$.  To see this, note that an examination of the proof that property (AO) implies solidity in \cite{Oz} shows that local reflexivity is only needed to establish, for each finite dimensional operator system $E \subset L^\infty(\G)$, the existence of a net of unital completely positive maps $\{\Psi_t^E:E \to C_r(\G)\}_t$ converging pointwise $\sigma$-weakly to $\id_E$.  In our case, we can just take $\Psi_t^E = \Phi_t|_{E}$, where $\{\Phi_t\}_{t < \dim B}$ is the net constructed in the proof of Theorem \ref{thm:HAP_Kac}.  Indeed, a standard estimate using property (RD) (see for example \cite{Br}, Proposition 5.3) shows that $\Phi_t$ is \textit{ultracontractive}, that is $\hat \Phi_t(L^2(\G)) \subseteq \Lambda_h(C_r(\G))$ for all $t$.  In particular, $\Phi_t(L^\infty(\G)) \subseteq C_r(\G)$ for all $t$.     
\end{rem}

\begin{rem}
The existence of a completely positive map $\theta$ in Theorem \ref{thm:AO} such that $\pi \circ \theta = \pi \circ m$ imposes a structural property on $L^\infty(\G)$ which is stronger than property (AO).  Using the terminology of Isono \cite{Is}, Theorem \ref{thm:AO} says that $L^\infty(\G)$ has \textit{property (AO)$^+$}.  Property (AO)$^+$ can be regarded as a generalization to finite von Neumann algebras of the description of bi-exactness for discrete groups in terms of their reduced C$^\ast$-algebras (see \cite{BrOz}, Chapter 15).  In \cite{PoVa}, Popa and Vaes proved that the II$_1$-factors $M = L(\Gamma)$ associated to i.c.c., weakly amenable, bi-exact groups $\Gamma$ are \textit{strongly solid}.  That is, for any diffuse injective von Neumann subalgebra $A \subseteq M$, the von Neumann algebra generated by the normalizer of $A$, $\mc N_M(A) = \{u \in \mc U(M): uAu^* = A\}$, is injective.  Recently, Isono \cite{Is} extended the results of \cite{PoVa} and proved that a finite von Neumann algebra with property (AO)$^+$ and the weak$^\ast$-completely bounded approximation property (W$^\ast$-CBAP) is strongly solid.  As a result, if one could prove that the von Neumann algebras $L^\infty(\G)$ considered here have the W$^\ast$-CBAP, their strong solidity would be an immediate consequence.  Based on recent work of Freslon \cite{Fr} on the weak amenability free orthogonal quantum groups, we expect that $L^\infty(\G)$ always has this approximation property. 
\end{rem}

\begin{proof}[Proof of Theorem \ref{thm:AO}]    Our arguments follow very closely those of Vergnioux \cite{Ve05}, where property (AO)$^+$ was established for the duals of universal discrete quantum groups.  We note that the proof in \cite{Ve05} makes extensive use of the notion of a Cayley graph associated to a discrete quantum group $\hG$, and assumes that this graph is a tree.  In the case of $\G = \G_{\text{aut}}(B,\psi)$ of interest here, the Cayley graph of $\hG$ is no longer a tree.  In this case,  the Cayley graph is obtained by taking the infinite tree $A_\infty$ and adding a loop starting and ending at each vertex except the origin.  Since such a graph is not ``too far'' from being a tree, we can readily adapt the proof from \cite{Ve05} to our context.  Below we freely use the notation from Sections \ref{expmodels}, \ref{section:RD} and \ref{section:factoriality}.  We also assume that $\delta^2 = q+q^{-1} \ge 5$, so that $C_r(\G)$ is not nuclear.

For each $l \in \N$, put $m_l = [2l+1]_q$, $\alpha_l = \Big(\sum_{n,k \in \N : \ n+k = l}\frac{m_nm_k}{m_l}\Big)^{-\frac{1}{2}}$ and define a linear map \[V:\C[\hG] \to \C[\hG] \otimes_{\text{alg}}\C[\hG]; \quad V(x) =\sum_{l \in \N} \alpha_l \sum_{n,k \in \N : \ n+k = l} \hat \Delta(\hat P_l x)(\hat P_n \otimes \hat P_k) \qquad (x \in \C[\hG]). \]
Using Lemma \ref{lem:RD_L2norm} and its proof, it follows that $V$ extends to an isometry from $\ell^2(\hG) \to \ell^2(\hG)\otimes\ell^2(\hG)$, and $V^*(\hat P_n \otimes \hat P_k)y = \alpha_{n+k} \hat P_{n+k}\Conv ( (\hat P_n \otimes \hat P_k)y)$ ($y \in \ell^2(\hG) \otimes \ell^2(\hG)$, $n, k \in \N$).  Identifying $\ell^2(\hG) \cong L^2(\G)$ via the Fourier transform $\mc F$, we can form a unital completely positive map \[\theta:C_r(\G) \otimes JC_r(\G)J \to \mc B(L^2(\G)); \quad \theta(x) =V^*xV \qquad (x\in C_r(\G) \otimes JC_r(\G)J). \]  Our goal is to prove that \begin{align}\label{eqn:compactness}aJbJV^* - V^*(a \otimes JbJ) \in \mc K(L^2(\G)\otimes L^2(\G), L^2(\G))\end{align} for all $a,b \in C_r(\G)$, which will show that $\theta$ satisfies the hypotheses of the theorem. Note that since $aJbJV^* - V^*(a \otimes JbJ)  = a(JbJV^* - V^*(1 \otimes JbJ) ) + (aV^* - V^*(a \otimes 1))(1 \otimes JbJ)$, it suffices to prove \eqref{eqn:compactness} when one of $a$ or $b$ equals $1$. Let us start with $b=1$.  Note furthermore that we only need to prove \eqref{eqn:compactness} for $a = \lambda(v) \in C_r(\G)$, where $v \in C_u(\G)$ is a matrix element of the irreducible representation $U^1 = U \ominus U^0$.  This follows because the norm-closed subalgebra of $C_r(\G)$ generated by all such $a = \lambda(v)$ equals $C_r(\G)$, and \eqref{eqn:compactness} is preserved by sums, products and norm-limits. 

Fix for the remainder $a = \lambda(v)$ as above, and put $K = \sum_{l \in \N} \alpha_l \hat P_l \in \mc B(L^2(\G))$.   Since $1 \le \frac{m_nm_k}{m_l} = \frac{(1-q^{4n+2})(1-q^{4k+2})}{(1-q^2)(1-q^{4l+2})} \le \frac{1}{(1-q^2)^2}$ for $k+n = l$,
it follows that $(1-q^2)(l+1)^{-1/2} \le \alpha_l \le (l+1)^{-1/2}$.  Thus $K$ is compact with (unbounded) inverse $K^{-1} = \sum_{l \in \N} \alpha_l^{-1} \hat P_l$.  Using $K$, we can write 
\[aV^*-V^*(a \otimes 1) = K(K^{-1}a-aK^{-1})V^* - K(aK^{-1}V^* - K^{-1}V^*(a \otimes 1)), \] so it suffices to show that $A_1 = K^{-1}a-aK^{-1}$ and $A_2= aK^{-1}V^* - K^{-1}V^*(a \otimes 1)$ define bounded linear maps.

First consider $A_1$.  Since $a= \lambda(v)=\hat P_0a\hat P_1+ \sum_{l \ge 1} \sum_{l' \in \N: \ |l-l'|\le 1}\hat P_{l}a\hat P_{l'}$,
we have
\[
A_1 =(1-\alpha_1^{-1}) \hat P_0 a\hat P_1 + \sum_{l \ge 1} \sum_{l' \in \N: \ |l-l'|\le 1}(\alpha_{l}^{-1}-\alpha_{l'}^{-1})\hat P_{l}a \hat P_{l'}. 
\]
It therefore suffices to show that $S := \sup_{l\ge 1}\{\max_{|l'-l|\le 1}|\alpha_l^{-1}-\alpha_{l'}^{-1}| \} < \infty$.  Since $(l+1)^{1/2} \le \alpha_l^{-1} \le \frac{(l+1)^{1/2}}{1-q^2}$, the fact that $S < \infty$ is clear.
 
Now consider $A_2$.  Let $\hat a = \mc F(v) \in \mc B(H_1) \subset \C[\hG]$, let $y \in \C[\hG] \otimes_{\text{alg}}\C[\hG]$, and put $y_{k,n} = (\hat P_k \otimes \hat P_n) y$ for each $k,n \in \N$.  For each inclusion $U^l \subset U^k \boxtimes U^n$ ($k,l,n \ge 0$), let $\phi_{l}^{k \boxtimes n} = C_{(k,n,l)}^{-1/2} \rho_{l}^{k \boxtimes n} :H_l \to H_k \otimes H_n$ be the (unique, up to multiplication by $\T$) isometric morphism defined by \eqref{eqn:rho}-\eqref{eqn:norm_intertwiner}. 
Since $V^*(y_{k,n}) = K \hat P_{k+n}\Conv(y_{k,n})$ and  (by Lemma \ref{lem:RD_L2norm} and its proof)  $\hat P_{l}\Conv(y_{k,n})  = \frac{m_km_n}{m_l}(\phi_{l}^{k \boxtimes n})^*y_{k,n}\phi_{l}^{k \boxtimes n}= \frac{m_km_n}{m_l}\text{Ad}((\phi_{l}^{k \boxtimes n})^*)y_{k,n}$,
it follows that 
\begin{align*}
&A_2(y_{k,n})  = \Conv (\hat a \otimes \hat P_{k+n}\Conv(y_{k,n})) - K^{-1}V^*((\Conv \otimes \id) (\hat a \otimes y_{k,n})  \\
&= \sum_{r = 0, \pm 1} \hat P_{n+k+r} \Big(\Conv (\hat a \otimes \hat P_{k+n}\Conv(y_{k,n})) -  \Conv((\hat P_{k+r}\Conv \otimes \id) (\hat a \otimes y_{k,n}))\Big)  \\
&= \sum_{r=0,\pm 1} \frac{m_1m_km_n}{m_{n+k+r}} \big(\text{Ad}([(\id \otimes \phi_{k+n}^{k\boxtimes n} )\phi_{k+n+r}^{1 \boxtimes k+n}]^*)  - \text{Ad}([(\phi_{k+r}^{1\boxtimes k} \otimes \id)\phi_{k+n+r}^{k+r \boxtimes n}]^*)\big)[\hat a \otimes y_{k,n}], 
\end{align*}
where $\hat P_{-1}$ is zero by convention. Let $\phi_{k+n+r,L}^{1 \boxtimes k\boxtimes n} = (\id \otimes \phi_{k+n}^{k\boxtimes n} )\phi_{k+n+r}^{1 \boxtimes k+n} $ and $\phi_{k+n+r,R}^{1 \boxtimes k\boxtimes n} =(\phi_{k+r}^{1\boxtimes k} \otimes \id)\phi_{k+n+r}^{k+r \boxtimes n}$. By appealing to \cite{VaVe}, Lemma A.1, note that we can find a constant $C = C(q) > 0$ (independent of $k$ and $n$) such that \begin{align} \label{eqn:asym_comm}\inf_{\lambda \in \T} \|\phi_{k+n+r,L}^{1 \boxtimes k\boxtimes n} - \lambda \phi_{k+n+r,R}^{1 \boxtimes k\boxtimes n}\|_{\TL_{2(k+n+r),2(k+n+1)}(\delta)} \le Cq^{2k+r-1} \qquad (n,k \in \N, \ r = 0, \pm 1).  \end{align}  Since we also have $\|T^*(\hat a \otimes y_{k,n})S\|_{\ell^2} \le\Big(\frac{m_{n+k+r}}{m_1m_km_n}\Big)^{1/2}\|T\|\|S\|\|\hat a \otimes y_{k,n}\|_{{\ell^2}^{\otimes 3}}$ for any $S,T \in \Mor(H_{n+k+r}, H_1 \otimes H_k \otimes H_n)$,   
an application of the triangle inequality and \eqref{eqn:asym_comm} yields \[\|\text{Ad}([\phi_{k+n+r,L}^{1 \boxtimes k\boxtimes n}]^*)  - \text{Ad}([\phi_{k+n+r,R}^{1 \boxtimes k\boxtimes n}]^*)[\hat a \otimes y_{k,n}]\|_{\ell^2(\hG)} \le  2Cq^{2k+r-1}\Big(\frac{m_{n+k+r}}{m_1m_km_n}\Big)^{1/2}\|\hat a \otimes y_{k,n}\|_{\ell^2(\hG)^{\otimes 3}}.\]  Using this last inequality and the Cauchy-Schwarz inequality, we obtain the estimate
\begin{align*}
&\|A_2y\|^2_{\ell^2(\hG)} = \sum_{l\ge 0} \Big\|\sum_{k,n: |n+k- l| \le 1 }\hat P_l A_2(y_{k,n}) \Big\|_{\ell^2(\hG)}^2 \\
&=\sum_{l\ge 0} \Big\|\sum_{k,n: |n+k- l| \le 1 } \frac{m_1m_km_n}{m_l} \big(\text{Ad}([\phi_{l,L}^{1 \boxtimes k\boxtimes n}]^*)  - \text{Ad}([\phi_{l,R}^{1 \boxtimes k\boxtimes n}]^*)\big)[\hat a \otimes y_{k,n}]\Big\|_{\ell^2(\hG)}^2 \\
&\le \sum_{l\ge 0} \Big(\sum_{k,n: |n+k- l| \le 1 } 2Cq^{k+l-n-1} \Big(\frac{m_1m_km_n}{m_l}\Big)^{1/2}  \|\hat a\|_{\ell^2(\hG)} \| y_{k,n}\|_{\ell^2(\hG)^{\otimes 2}}\Big)^2 \\
& \le 4C^2m_1 \|\hat a\|_{\ell^2(\hG)}^2  \sum_{l\ge 0}  \Big(\sum_{k,n: |n+k- l| \le 1 } \frac{m_km_n}{m_l}q^{2k+2l-2n-2}\Big)\Big( \sum_{k,n: |k+n-l|\le 1}\| y_{n,k}\|_{\ell^2(\hG)^{\otimes 2}}^2\Big) \\
&\le   4C^2m_1 \|\hat a\|_{\ell^2(\hG)}^2 \sup_{l \ge 0}\Big\{ \sum_{k,n: |n+k- l| \le 1 } \frac{m_km_n}{m_l}q^{2k+2l-2n-2}\Big\} \sum_{l \ge 0} \sum_{k,n: |k+n-l|\le 1}\| y_{n,k}\|_{\ell^2(\hG)^{\otimes 2}}^2.
\end{align*}  
Since the above sum over $l$ is dominated by $3\|y\|_{\ell^2(\hG)^{\otimes 2}}^2$,  $\frac{m_km_n}{m_l} \le \frac{q^{-2}}{(1-q^2)^2}$ for $|k+n-l| \le 1$, and 
\begin{align*}
&\sum_{k,n: |n+k- l| \le 1 } \frac{m_km_n}{m_l}q^{2k+2l-2n-2} \le \sum_{k,n: |n+k- l| \le 1 } \frac{q^{-2}}{(1-q^2)^2}q^{2k+2l-2n-2}  \\
&= \frac{q^{-4}}{(1-q^2)^2}\sum_{r=0,\pm 1} \sum_{k=0}^{l+r}  q^{4k -2r}
 \le  \frac{3q^{-6}}{(1-q^2)^2(1-q^4)} \qquad (l \ge 0),  
\end{align*} we conclude that $A_2$ is bounded.

The proof of  \eqref{eqn:compactness} when $a=1$ and $b \ne 1$ is essentially the same as the above argument.  The only significant difference is that one must use a right-sided version of \eqref{eqn:asym_comm}.  The required inequality  is given by Lemma A.2 from \cite{VaVe}.  We leave the details to the reader.   
\end{proof}

\section{Appendix -- Proof of Theorem \ref{thm:decomp_T}} \label{section:appendix}

In this section we prove Theorem \ref{thm:decomp_T}.  We will freely use the notation and assumptions of Section \ref{section:factoriality}. 

\begin{proof}[Proof of Theorem \ref{thm:decomp_T}]
We can write the commutator map $T:L^2(\G) \to H_1 \otimes L^2(\G) \otimes H_1$ as the sum $T = T_L -T_R$, where $T_L = \sum_{i,j=1}^{d_1} F_1e_j \otimes \lambda(u_{ij}^1) \otimes e_i$ is the part of $T$ corresponding to the left action of $\{\lambda(u_{ij}^1)\}_{i,j=1}^{d_1}$ on $L^2(\G)$ and $-T_R$ is the remaining part of $T$ corresponding to the right action of $\{\lambda(u_{ij}^1)\}_{i,j=1}^{d_1}$ on $L^2(\G)$.  Observe that $T\xi_0 = T_L\xi_0-T_R\xi_0 = 0$.

In terms of matrix elements of irreducible representations of $\G$, the map $T_L$ (respectively $T_R$) corresponds to the operation of tensoring on the left (respectively on the right) by $U^1$.  Since  $U^1 \boxtimes U^k \cong U^{k+1} \oplus U^k \oplus U^{k-1} \cong U^k \boxtimes U^1$ for all $k \ge 1$,
 we can uniquely decompose $T_L$ and $T_R$ as the sums $T_{L,R} = \sum_{\alpha = 0,\pm 1} T^{(\alpha)}_{L,R}$, where \[T^{(0)}_{L,R}\xi_0 = T^{(-1)}_{L,R}\xi_0 = 0, \qquad T^{(+1)}_{L}\xi_0 = T^{(+1)}_{R}\xi_0   = \sum_{i,j=1}^{d_1} F_1e_j \otimes \lambda(u_{ij}^1)\xi_0 \otimes e_i, \]
and \[T^{(\alpha)}_{L,R}: L^2_k(\G) \to H_1 \otimes L^2_{k+\alpha}(\G) \otimes H_1;  \qquad  T^{(\alpha)}_{L,R}|_{L^2_k(\G)} =(\id _{H_1} \otimes P_{k+\alpha} \otimes \id_{H_1})T_{L,R}|_{L^2_k(\G)}   \qquad (k \ge 1).\]
Setting $T^{(\alpha)} := T^{(\alpha)}_L -  T^{(\alpha)}_R$, we obtain the decomposition $T = \sum_{\alpha = 0, \pm 1} T^{(\alpha)}$ stated in Theorem \ref{thm:decomp_T}.
 
To verify the expressions for $T^{(\alpha)}$ given by Theorem \ref{thm:decomp_T}, fix $k \ge 1$, $\xi, \eta \in H_k$ and consider the matrix element  $\lambda\big((\omega_{\eta,\xi} \otimes \id )U^k\big)\xi_0 = \Lambda_h((\omega_{\eta,\xi} \otimes \id )U^k) \in L^2_k(\G)$. 
In the following calculations, we freely use the identities given in Remarks \ref{rem:m_nu_relns} and \ref{rem:tmaps}, the unitary identifications  \[L^2_k(\G) \cong H_k \otimes H_k; \qquad \lambda\big((\omega_{\eta,\xi} \otimes \id )U^k\big)\xi_0 \mapsto \xi \otimes (1 \otimes \eta^*)t_{2k} \qquad (\xi, \eta \in H_k,   k \ge 0), \] given by \eqref{eqn:L2unitarymap}, as well as the identities \[ \sum_{j=1}^{d_1} F_1e_j \otimes e_j = \sum_{j=1}^{d_1} F_1F_1^*F_1^{tr}e_j \otimes e_j = \sum_{j=1}^{d_1} F_1^{tr}e_j \otimes e_j = [3]_q^{1/2} t_2,\]
\[(p_x\otimes p_y)p_{x+y} =p_{x+y} = p_{x+y}(p_x\otimes p_y) \qquad (x,y \in \N).\] 

Since $\bigoplus_{\alpha = 0,\pm 1} \phi_{k+\alpha,L}^{(-\alpha)}: \bigoplus_{\alpha = 0,\pm 1} H_{k+\alpha} \to H_1 \otimes H_k$ is a unitary intertwiner, it follows that \begin{align*} &T_L \big(\lambda\big((\omega_{\eta,\xi} \otimes \id )U^k\big)\xi_0\big) =  \sum_{i,j=1}^{d_1} F_1e_j \otimes \lambda(u_{ij}^1 (\omega_{\eta,\xi} \otimes \id )U^k)\xi_0 \otimes e_i \\
&= \sum_{i,j=1}^{d_1} F_1e_j \otimes \lambda((\omega_{e_i \otimes \eta, e_j \otimes \xi} \otimes \id )(U^1 \boxtimes U^k))\xi_0 \otimes e_i  \\
&= \sum_{\alpha = 0, \pm 1}  \underbrace{ \sum_{i,j=1}^{d_1} F_1e_j \otimes \lambda((\omega_{(\phi_{k+\alpha, L}^{(-\alpha)})^*(e_i \otimes \eta), (\phi_{k+\alpha, L}^{(-\alpha)})^*(e_j \otimes \xi)} \otimes \id )U^{k+\alpha} )\xi_0 \otimes e_i}_{\in H_1 \otimes L^2_{k+\alpha}(\G) \otimes H_1}.    \end{align*} 
Therefore \begin{align*}
&T^{(\alpha)}_L(\xi \otimes (1 \otimes \eta^*)t_{2k}) =  \sum_{i,j=1}^{d_1} F_1e_j \otimes \lambda((\omega_{(\phi_{k+\alpha, L}^{(-\alpha)})^*(e_i \otimes \eta), (\phi_{k+\alpha, L}^{(-\alpha)})^*(e_j \otimes \xi)} \otimes \id )U^{k+\alpha} )\xi_0 \otimes e_i \\
&=  \sum_{j = 1}^{d_1} F_1e_j \otimes  (\phi_{k+\alpha, L}^{(-\alpha)})^*(e_j \otimes \xi) \otimes \sum_{i=1}^{d_1} (1 \otimes (e_i^* \otimes \eta^*)\phi_{k+\alpha, L}^{(-\alpha)})t_{2k+2\alpha} \otimes e_i\\
&= \underbrace{\Big[(1_2 \otimes  (\phi_{k+\alpha, L}^{(-\alpha)})^* ) \Big(  \sum_j F_1e_j \otimes e_j \otimes \xi\Big)\Big]}_{=:A_1^{(\alpha)} \xi} \otimes \underbrace{\Big[(1_{2k+2\alpha} \otimes 1_2 \otimes \eta^*)(1_{2k+2\alpha} \otimes \phi_{k+\alpha, L}^{(-\alpha)}) )t_{2k+2\alpha}\Big]}_{=:A_2^{(\alpha)}(1 \otimes \eta^*)t_{2k}}, \end{align*}
and similarly
\begin{align*} &T^{(\alpha)}_R(\xi \otimes (1 \otimes \eta^*)t_{2k}) \\
& =\underbrace{\Big[ (1_2 \otimes (\phi_{k+\alpha, R}^{(-\alpha)})^*)\Big(\sum_{j=1}^{d_1} F_1e_j \otimes \xi \otimes e_j\Big)\Big]}_{=:B_1^{(\alpha)}\xi} \otimes \underbrace{\Big[ (1_{2k+2\alpha} \otimes \eta^* \otimes 1_2)(1_{2k+2\alpha} \otimes \phi_{k+\alpha, R}^{(-\alpha)})t_{2k+2\alpha}\Big]}_{=:B_2^{(\alpha)}(1 \otimes \eta^*)t_{2k}}  .\end{align*}

Let us first consider $T_L^{(\alpha)}|_{H_k \otimes H_k} = A_1^{(\alpha)} \otimes A_2^{(\alpha)}$.  When $\alpha = 1$, we have 
\begin{align*}
A_1^{(1)}\xi &= [3]_q^{1/2} (1_2 \otimes  (\phi_{k+1, L}^{(-1)})^* )  (t_2 \otimes \xi) = [3]_q^{1/2}(1_2 \otimes p_{2k+2}(p_2 \otimes p_{2k}))(t_2 \otimes \xi) \\
& = [3]_q^{1/2} (p_2 \otimes p_{2k+2})(t_2 \otimes 1_{2k})p_{2k}\xi = \Big(\frac{[2k+3]_q}{[2k+1]_q}\Big)^{1/2}  \phi^{(+1)}_{k,L}\xi,
\end{align*}
and
\begin{align*}
&A_2^{(1)}(1 \otimes \eta^*)t_{2k} = (1_{2k+2} \otimes 1_2 \otimes \eta^*)(1_{2k+2} \otimes p_2 \otimes p_{2k})t_{2k+2} \\
&=\Big(\frac{[2k+1]_q[3]_q}{[2k+3]_q}\Big)^{1/2} (1_{2k+2} \otimes 1_2 \otimes \eta^*)(p_{2k+2} \otimes p_2 \otimes p_{2k})(1_{2k} \otimes t_2 \otimes 1_{2k})t_{2k} \\ 
&= (1_{2k+2} \otimes 1_2 \otimes \eta^*) (1_{2k+2} \otimes 1_{2k+2}) (\phi_{k,R}^{(+1)}\otimes 1_{2k})t_{2k} = \phi_{k,R}^{(+1)}(1_{2k} \otimes \eta^*)t_{2k}. 
\end{align*}

When $\alpha = 0$, we have 
\begin{align*}
 A_1^{(0)}\xi &= [3]_q^{1/2} (p_2 \otimes (\phi_{k, L}^{(0)})^*) (t_2 \otimes \xi) = \Big(\frac{[3]_q[2k]_q}{[2k+2]_q}\Big)^{1/2}(p_2 \otimes p_{2k}(m\otimes 1_{2k-2})(p_2 \otimes p_{2k})) (t_2 \otimes \xi)  \\
&= \Big(\frac{[3]_q[2k]_q}{[2k+2]_q}\Big)^{1/2}(p_2 \otimes p_{2k}(m\otimes 1_{2k-2})) (t_2 \otimes \xi)  \\
&= \Big(\frac{[2k]_q}{[2k+2]_q}\Big)^{1/2}(p_2 \otimes p_{2k}(m\otimes 1_{2k-2}))(m^*\nu \otimes 1_{2k}) \xi \\
&= \Big(\frac{[2k]_q}{[2k+2]_q}\Big)^{1/2}(p_2 \otimes p_{2k})(1_2 \otimes m\otimes 1_{2k-2})(m^* \otimes 1_{2k})(\nu \otimes 1_{2k}) \xi \\
&= \Big(\frac{[2k]_q}{[2k+2]_q}\Big)^{1/2}(p_2 \otimes p_{2k})(m^*m \otimes 1_{2k-2})(\nu \otimes 1_{2k}) \xi \qquad (\text{cf. Remark \ref{rem:m_nu_relns}}) \\
&= \Big(\frac{[2k]_q}{[2k+2]_q}\Big)^{1/2}(p_2 \otimes p_{2k})(m^*\otimes 1_{2k-2})\xi = \phi^{(0)}_{k,L}\xi, \end{align*}
and 
\begin{align*}
&A_2^{(0)}(1 \otimes \eta^*)t_{2k}= (1_{2k} \otimes 1_{2} \otimes \eta^*)(1_{2k} \otimes \phi_{k, L}^{(0)})t_{2k}= \phi_{k,R}^{(0)}(1_{2k} \otimes \eta^*)t_{2k}.
\end{align*}
The last equation follows because both of the isometries $(\phi_{k, R}^{(0)} \otimes 1_{2k})t_{2k}$ and $(1_{2k} \otimes \phi_{k, L}^{(0)})t_{2k}$ belong to the one-dimensional space $\Mor(1, U^k \boxtimes U^1 \boxtimes U^k)$ and therefore $(1_{2k} \otimes \phi_{k, L}^{(0)})t_{2k} = z(\phi_{k, R}^{(0)} \otimes 1_{2k})t_{2k}$ for some $z \in \T$.  But then
 \begin{align*} z &= t_{2k}^*((\phi_{k, R}^{(0)})^* \otimes 1_{2k})(1_{2k} \otimes \phi_{k, L}^{(0)})t_{2k} \\
& =  \frac{[2k]_q}{[2k+2]_q}t_{2k}^*(1_{2k-2}\otimes m \otimes 1_{2k})(1_{2k} \otimes p_2 \otimes 1_{2k}) (1_{2k} \otimes m^* \otimes 1_{2k-2})t_{2k} \\
&=  \frac{[2k]_q}{[2k+2]_q}t_{2k}^*(1_{2k-2}\otimes m \otimes 1_{2k})(1_{2k} \otimes (1_2 - \nu\nu^*) \otimes 1_{2k}) (1_{2k} \otimes m^* \otimes 1_{2k-2})t_{2k} \\
&= \frac{[2k]_q}{[2k+2]_q}\Big( t_{2k}^*(1_{2k-2} \otimes m^*m \otimes 1_{2k-2})t_{2k} - t_{2k}^*t_{2k}\Big) \\
&= \frac{[2k]_q}{[2k+2]_q} \Big(\frac{[2]_q[2k+1]_q}{[2k]_q} -1 \Big) = 1,
\end{align*}
where in the last line we have used $(1_{2k-2} \otimes m \otimes 1_{2k-2})t_{2k} = \big([2]_q[2k+1]_q[2k]_q^{-1}\big)^{1/2}t_{2k-1}$. 

When $\alpha = -1$, we have  
\begin{align*}
&A_1^{(-1)}\xi = [3]_q^{1/2} (1_2 \otimes  (\phi_{k-1, L}^{(+1)})^*)(t_2 \otimes \xi) \\
&= \Big( \frac{ [3]_q^2[2k-1]_q}{[2k+1]_q} \Big)^{1/2}(1_2 \otimes p_{2k-2}(t_2^* \otimes 1_{2k-2})(p_2 \otimes p_{2k})) (t_2 \otimes \xi) \\
&= \Big( \frac{ [3]_q^2[2k-1]_q}{[2k+1]_q} \Big)^{1/2}(p_2 \otimes p_{2k-2})(1_2 \otimes t_2^* \otimes 1_{2k-2}) (t_2 \otimes 1_{2k})\xi \\
&= \Big( \frac{[2k-1]_q}{[2k+1]_q} \Big)^{1/2}(p_2 \otimes p_{2k-2})\xi = \Big( \frac{[2k-1]_q}{[2k+1]_q} \Big)^{1/2} \phi^{(-1)}_{k,L} \xi \qquad (\text{using $(1_2 \otimes t_2^*)(t_2 \otimes 1_2) = [3]_q^{-1}p_2$}),  
\end{align*}
\begin{align*}
&\text{and,} \\
&A_2^{(-1)}(1 \otimes \eta^*)t_{2k} \\
&= \Big(\frac{[3]_q[2k-1]_q}{[2k+1]_q}\Big)^{1/2}(1_{2k-2} \otimes 1_2 \otimes \eta^*) (p_{2k-2}\otimes p_{2} \otimes p_{2k})(1_{2k-2} \otimes t_2\otimes 1_{2k-2})t_{2k-2} \\
&= (1_{2k-2} \otimes 1_2 \otimes \eta^*) (p_{2k-2} \otimes p_{2} \otimes 1_{2k}) \Big(\frac{[3]_q[2k-1]_q}{[2k+1]_q}\Big)^{1/2}(1_{2k} \otimes  p_{2k})(1_{2k-2} \otimes t_2\otimes 1_{2k-2})t_{2k-2} \\
&= (1_{2k} \otimes \eta^*) (\phi^{(-1)}_{k,R} \otimes 1_{2k})t_{2k} = \phi^{(-1)}_{k,R} (1_{2k} \otimes \eta^*)t_{2k}.
\end{align*}

Let us now consider $T_R^{(\alpha)}|_{H_k \otimes H_k} = B_1^{(\alpha)} \otimes B_2^{(\alpha)}.$  When $\alpha = 1,$ we have
\begin{align*}
&B_1^{(1)}\xi  = [3]_q^{1/2} \sigma \big(  (\phi_{k+1, R}^{(-1)})^* \otimes 1_2 \big)(\xi \otimes t_2) = [3]_q^{1/2} \sigma (p_{2k+2}(p_{2k} \otimes p_2) \otimes 1_2)  (1_{2k}\otimes t_2)\xi\\
&= [3]_q^{1/2} \sigma (p_{2k+2}\otimes p_2) (1_{2k}\otimes t_2)\xi = \Big(\frac{[2k+3]_q}{[2k+1]_q}\Big)^{1/2} \sigma \phi^{(+1)}_{k,R}\xi, \end{align*}
and 
\begin{align*}
&B_2^{(1)}(1 \otimes \eta^*)t_{2k} =(1_{2k+2} \otimes \eta^* \otimes 1_2)(1_{2k+2} \otimes \phi_{k+1, R}^{(-1)})t_{2k+2} \\
&= \Big(\frac{[2k+1]_q}{[2k+3]_q}\Big)^{1/2} (1_{2k+2} \otimes \eta^* \otimes 1_2)(p_{2k+2} \otimes 1_{2k+2})\Big(\sum_{j=1}^{d_1}e_j \otimes 1_{2k} \otimes 1_{2k} \otimes F_1e_j\Big)t_{2k} \\
&= \Big(\frac{[2k+1]_q}{[2k+3]_q}\Big)^{1/2} \sigma^*(1_2 \otimes 1_{2k+2}\otimes \eta^*)(1_{2} \otimes p_{2k+2} \otimes 1_{2k})\Big(\sum_{j=1}^{d_1}F_1 e_j \otimes e_j \otimes 1_{2k} \otimes 1_{2k}\Big)t_{2k} \\
&= \Big(\frac{[2k+1]_q}{[2k+3]_q}\Big)^{1/2} \sigma^* (1_{2} \otimes p_{2k+2})\Big(\sum_{j=1}^{d_1}F_1 e_j \otimes e_j \otimes 1_{2k} \otimes \eta^*\Big)t_{2k} \\
&= \Big(\frac{[3]_q[2k+1]_q}{[2k+3]_q}\Big)^{1/2} \sigma^* (1_{2} \otimes p_{2k+2})(t_2 \otimes 1_{2k} \otimes \eta^*)t_{2k} \\ 
&= \Big(\frac{[3]_q[2k+1]_q}{[2k+3]_q}\Big)^{1/2} \sigma^* (1_{2} \otimes p_{2k+2})(t_2 \otimes 1_{2k})(1_{2k} \otimes \eta^*)t_{2k} =  \sigma^*\phi_{k,L}^{(+1)}((1_{2k} \otimes \eta^*)t_{2k}).
\end{align*}

When $\alpha = 0$, we have
\begin{align*}
&B_1^{(0)}\xi  = [3]_q^{1/2}\sigma ((\phi_{k,R}^{(0)})^* \otimes 1_2 )(\xi \otimes t_2) =  [3]_q^{1/2}\Big(\frac{[2k]_q}{[2k+2]_q}\Big)^{1/2}\sigma (p_{2k}(1_{2k-2} \otimes m)(p_{2k} \otimes p_2) \otimes 1_2 )(\xi \otimes t_2) \\
&=[3]_q^{1/2}\Big(\frac{[2k]_q}{[2k+2]_q}\Big)^{1/2}\sigma (p_{2k} \otimes p_2)(1_{2k-2} \otimes m \otimes 1_2 )(1_{2k} \otimes t_2)\xi \\
&=\Big(\frac{[2k]_q}{[2k+2]_q}\Big)^{1/2}\sigma (p_{2k} \otimes p_2)(1_{2k-2} \otimes m \otimes 1_2 )(1_{2k} \otimes 1_{2} \otimes p_2)(1_{2k} \otimes m^*)(1_{2k} \otimes \nu)\xi \\
&=\Big(\frac{[2k]_q}{[2k+2]_q}\Big)^{1/2}\sigma (p_{2k} \otimes p_2)(1_{2k-2} \otimes m \otimes 1_2 )(1_{2k} \otimes m^*)(1_{2k} \otimes \nu)\xi \\
&=\Big(\frac{[2k]_q}{[2k+2]_q}\Big)^{1/2}\sigma (p_{2k} \otimes p_2)(1_{2k-2} \otimes m^*)\xi = \sigma \phi^{(0)}_{k,R} \xi.\\
\end{align*}
Using the fact that $p_2^{\otimes 3}(1_2 \otimes m^*)t_2 = [3]_q^{1/2}p_2^{\otimes 3}(1_2 \otimes m^*)m^*\nu = [3]_q^{1/2}p_2^{\otimes 3}(m^* \otimes 1_2)m^*\nu =
p_2^{\otimes 3}(m^* \otimes 1_2)t_2$, we also have
\begin{align*}
&  B_2^{(0)}(1 \otimes \eta^*)t_{2k} =\Big(\frac{[2k]_q}{[2k+2]_q}\Big)^{1/2} (1_{2k} \otimes \eta^* \otimes 1_2) (p_{2k} \otimes (p_{2k} \otimes p_2)(1_{2k-2} \otimes m^*))t_{2k} \\
&= \Big(\frac{[3]_q[2k]_q[2k-1]_q}{[2k+2]_q[2k+1]_q}\Big)^{1/2} (p_{2k} \otimes \eta^* \otimes p_2) (1_{2k} \otimes 1_{2k-2} \otimes m^*)(1_{2} \otimes t_{2k-2} \otimes 1_2) t_2 \\
&= \Big(\frac{[3]_q[2k]_q[2k-1]_q}{[2k+2]_q[2k+1]_q}\Big)^{1/2} (p_{2k} \otimes \eta^* \otimes p_2) (1_{2} \otimes t_{2k-2} \otimes 1_2 \otimes 1_2) (1_{2} \otimes  m^*)t_2 \\
&= \Big(\frac{[3]_q[2k]_q[2k-1]_q}{[2k+2]_q[2k+1]_q}\Big)^{1/2} (p_{2k} \otimes \eta^* \otimes p_2) (1_{2} \otimes t_{2k-2} \otimes 1_2 \otimes 1_2) (m^* \otimes 1_{2})t_2 \\
&= \Big(\frac{[3]_q[2k]_q[2k-1]_q}{[2k+2]_q[2k+1]_q}\Big)^{1/2}\sigma^*(p_2 \otimes p_{2k} \otimes \eta^*) (1_2 \otimes 1_{2} \otimes t_{2k-2} \otimes 1_2 ) (1_2 \otimes m^*)t_2 \\
&= \Big(\frac{[3]_q[2k]_q[2k-1]_q}{[2k+2]_q[2k+1]_q}\Big)^{1/2}\sigma^*(p_2 \otimes p_{2k} \otimes \eta^*) (1_2 \otimes 1_{2} \otimes t_{2k-2} \otimes 1_2 ) (m^* \otimes 1_{2})t_2 \\
&= \Big(\frac{[3]_q[2k]_q[2k-1]_q}{[2k+2]_q[2k+1]_q}\Big)^{1/2}\sigma^*(p_2 \otimes p_{2k} \otimes \eta^*) (m^* \otimes 1_{4k-2}) (1_{2} \otimes t_{2k-2} \otimes 1_2 )t_2\\
&= \Big(\frac{[2k]_q}{[2k+2]_q}\Big)^{1/2}\sigma^*(p_2 \otimes p_{2k}) (m^* \otimes 1_{2k-2})(1_{2k} \otimes \eta^*) t_{2k} = \sigma^* \phi^{(0)}_{k,L}(1_{2k} \otimes \eta^*) t_{2k}.
\end{align*}

Finally, when $\alpha = -1$, we have
\begin{align*}
&B_1^{(-1)}\xi =   [3]_q^{1/2}\sigma ( (\phi_{k-1,R}^{(+1)})^* \otimes 1_2) ( \xi \otimes t_2) \\
&=  [3]_q^{1/2}\Big(\frac{[3]_q[2k-1]_q}{[2k+1]_q}\Big)^{1/2}\sigma ( p_{2k-2}(1_{2k-2} \otimes t_2^*)(p_{2k} \otimes p_2) \otimes 1_2) ( \xi \otimes t_2) \\
&=  [3]_q^{1/2}\Big(\frac{[3]_q[2k-1]_q}{[2k+1]_q}\Big)^{1/2}\sigma ( p_{2k-2}(1_{2k-2} \otimes t_2^*) \otimes 1_2) ( 1_{2k} \otimes t_2)\xi \\
&=  [3]_q^{1/2}\Big(\frac{[3]_q[2k-1]_q}{[2k+1]_q}\Big)^{1/2}\sigma ( p_{2k-2} \otimes 1_2)(1_{2k-2} \otimes t_2^* \otimes 1_2) ( 1_{2k} \otimes t_2)\xi \\
&=  \Big(\frac{[2k-1]_q}{[2k+1]_q}\Big)^{1/2}\sigma ( p_{2k-2} \otimes p_2)\xi = \Big(\frac{[2k-1]_q}{[2k+1]_q}\Big)^{1/2}\sigma \phi_{k,R}^{(-1)}\xi \quad (\text{using } [3]_q^{-1} = (t_2^* \otimes 1_2)(1_2 \otimes t_2)), \\
&\text{and} \\
&B_2^{(-1)}(1 \otimes \eta^*)t_{2k} = \Big(\frac{[3]_q[2k-1]_q}{[2k+1]_q}\Big)^{1/2}(1_{2k-2} \otimes \eta^* \otimes 1_2) (1_{2k-2} \otimes (p_{2k}\otimes p_{2})(1_{2k-2} \otimes t_2)p_{2k-2})t_{2k-2} \\
&= \Big(\frac{[3]_q[2k-1]_q}{[2k+1]_q}\Big)^{1/2}(p_{2k-2} \otimes \eta^* \otimes p_2) (1_{2k-2} \otimes 1_{2k-2} \otimes t_2)t_{2k-2} \\
&= \Big(\frac{[3]_q[2k-1]_q}{[2k+1]_q}\Big)^{1/2}\sigma^*(p_2 \otimes p_{2k-2} \otimes \eta^*) (1_{2} \otimes t_{2k-2} \otimes 1_2)t_{2} \\
&= \sigma^*(p_2 \otimes p_{2k-2} \otimes \eta^*) t_{2k} = \sigma^* \phi_{k,L}^{(-1)}(1_{2k} \otimes \eta^*)t_{2k}.
\end{align*}
The theorem now follows from these identities. 

\end{proof}

\end{document}